\documentclass[a4paper,12pt,twoside]{amsart}
\usepackage{amsmath}
\usepackage{amsthm}
\usepackage{amssymb}
\usepackage{mathrsfs}
\theoremstyle{plain}
\newtheorem{thm}{\indent\bf Theorem}[section]
\newtheorem{lem}[thm]{\indent\bf Lemma}
\newtheorem{prop}[thm]{\indent\bf Proposition}
\newtheorem{cor}[thm]{\indent\bf Corollary}

\theoremstyle{definition}
\newtheorem{rem}{\indent\it Remark}[section]

\newtheorem{exa}{\indent\bf Example}[section]
\newtheorem{defi}{\indent\bf Definition}[section]
\numberwithin{equation}{section}
\numberwithin{figure}{section}

\def \re {\mathrm{Re\,}}
\def \im {\mathrm{Im\,}}

\def \z {\mathfrak{z}}
\begin{document}
\title[Painlev\'e V]{Full characterisation of Painlev\'e V 
asymptotics and nonlinear monodromy-Stokes structure}
\author[Shun Shimomura]{Shun Shimomura} 
\address{Department of Mathematics, 
Keio University, 
3-14-1, Hiyoshi, Kohoku-ku,
Yokohama 223-8522 
Japan \quad
{\tt shimomur@math.keio.ac.jp}}
\date{}
\maketitle
\begin{abstract}
For a generic Painlev\'e V equation we characterise all the asymptotics  
in a right half plane near the point at infinity, 
that is, we find classified explicit solutions that are,
by the Riemann-Hilbert correspondence, labelled 
with monodromy data filling up the whole monodromy manifold. 
To do so, in addition to the asymptotics by Andreev and Kitaev along 
the positive real axis, we require elliptic asymptotics along generic
directions and newly provided truncated solutions arising 
from a general solution 
along the imaginary axes. To know analytic continuations outside 
this region we formulate a nonlinear monodromy-Stokes structure, which is
observed as changes of monodromy data contained in the explicit expressions of
solutions. 
\end{abstract}
\allowdisplaybreaks
\section{Introduction}\label{sc1}
In comparison with Painlev\'e I and II equations, the full picture of Painlev\'e
V solutions has not been given light on. Attempts to capture or to tabulate 
all Painlev\'e solutions have been made by Guzzetti \cite{G1} for Painle\'e VI 
around $x=0$, by Vartanian and Kitaev \cite{KV} for Painlev\'e III$(D_7)$ around
$x=0$, and by Andreev and Kitaev \cite{Andreev-Kitaev} for Painlev\'e V  
as $x\to \infty$ along the positive real axis in the context 
of the connection problem to $x=0$. 
The result \cite{S} providing the elliptic asymptotics of Painlev\'e V
suggests the possibility of extending Andreev-Kitaev's study to that 
in generic directions including the positive real axis. 
We are now interested in capturing all the asymptotic solutions of 
Painlev\'e V, say, in a specified sectorial region.
Our study is based on the isomonodromy deformation technique.
\par
Let $\theta_0,$ $\theta_1,$ $\theta_{\infty} \in \mathbb{C}$.  
The isomonodromy deformation of the two-dimensional linear system 
\begin{align} \label{1.1}
& \frac{d\Xi}{d\xi} = \Bigl(\frac x2 \sigma_3 + \frac{\mathcal{A}_0}{\xi}
 + \frac{\mathcal{A}_1}{\xi-1} \Bigr) \Xi,
\qquad  
\sigma_3 = \begin{pmatrix}  1  &  0  \\  0  &  -1
\end{pmatrix},
\\ \notag
& \mathcal{A}_{0} =
\begin{pmatrix}
 \z+ \frac 12\theta_0  & -u(\z+ \theta_0) \\
u^{-1}\z   &            -\z-\frac 12\theta_0 \\
\end{pmatrix},
\\ \notag
& \mathcal{A}_1=
\begin{pmatrix}
-\z-\frac 12(\theta_0+\theta_{\infty})  & uy(\z+\frac 12(\theta_0-\theta_1
 +\theta_{\infty})) \\
-(uy)^{-1}(\z+\frac 12(\theta_0+\theta_1 +\theta_{\infty})) & 
\z+\frac 12(\theta_0+\theta_{\infty})  \\
\end{pmatrix}
\end{align}
containing the parameters $x$ and $(y, \z, u)$ 
is governed by the system of equations
\begin{equation}\label{1.2}
\begin{split}
& xy' =xy -2\z (y-1)^2 -(y-1) \bigl(\tfrac 12 (\theta_0-\theta_1
+\theta_{\infty})y -\tfrac 12 (3\theta_0 +\theta_1 +\theta_{\infty})\bigr),
\\
& x\z' =y\z  \bigl(\z+ \tfrac 12 (\theta_0-\theta_1
+\theta_{\infty})\bigr)  - y^{-1} {(\z+\theta_0)} \bigl(\z+ \tfrac 12 
(\theta_0 +\theta_1 +\theta_{\infty})\bigr),
\\
& x(\ln u)' = -2\z -\theta_0 + y \bigl(\z+ \tfrac 12 (\theta_0-\theta_1
+\theta_{\infty})\bigr)  + {y}^{-1} \bigl(\z+ \tfrac 12 
(\theta_0 +\theta_1 +\theta_{\infty})\bigr)
\end{split}
\end{equation}
$(y'=dy/dx)$, which is equivalent to Painlev\'e V equation
\begin{equation}
\tag*{(P$_\mathrm{V}$)}
y''=  \Bigl(\frac 1{2y} + \frac 1{y-1} \Bigr) (y')^{2}- \frac {y'}{x}
  +\frac{(y-1)^2}{x^2} \Bigl(a_{\theta} y - \frac{b_{\theta}} 
{y}\Bigr) + c_{\theta} \frac{y}{x} -\frac{y(y+1)}{2(y-1)}
\end{equation}
with $ a_{\theta}=\tfrac 18(\theta_0-\theta_1+\theta_{\infty})^2, 
b_{\theta}=\tfrac 18 (\theta_0 -\theta_1 - \theta_{\infty})^2,  
c_{\theta}=1-\theta_0-\theta_1$ 
\cite[Appendix C]{JM}, \cite{Andreev-Kitaev}; that is, 
$y(x)$ solves (P$_{\mathrm{V}}$) if and only if the monodromy data 
for \eqref{1.1} remain invariant under a small change of $x.$  
As explained in Section \ref{sc2} by the Riemann-Hilbert correspondence 
every solution of (P$_{\mathrm{V}}$) is parametrised by monodromy data 
in $SL_2(\mathbb{C})$. 
\par
In this paper, under a generic condition on $(\theta_0,\theta_1,
\theta_{\infty})$, we find classified families of asymptotic solutions 
as $x\to \infty$ along every 
direction contained in $|\arg x|<\pi /2$ and show that, by the bijective
Riemann-Hilbert correspondence, 
these solutions may be labelled with monodromy data filling up 
the whole monodromy manifold, which explicitly characterises all 
the asymptotic behaviours (cf. Theorem \ref{thm2.12}). 
To do so, in addition to the asymptotics 
by Andreev and Kitaev \cite{Andreev-Kitaev} on the positive real axis,  
we require elliptic asymptotics along generic directions \cite{S} 
and truncated solutions \cite{Andreev} arising from a two-parameter 
oscillatory family along the imaginary axes. 
Indeed \cite{Andreev-Kitaev} has not
discussed the solutions corresponding to the regions $\mathcal{R}_3\cup 
\mathcal{R}_4\cup\mathcal{R}_5$ on the manifold of monodromy data, 
and to complete the characterisation on the positive real axis it is necessary 
to examine in detail the one-parameter solutions \cite{Andreev} mentioned above.
To know analytic continuations in $|\arg x|>\pi/2$ 
we formulate a nonlinear monodromy-Stokes structure 
based on the monodromy data, which provides nonlinear
monodromy actions on the character variety agreeing with that of
\cite{Kli, PR}. 
\par
The results are summed up in Section \ref{sc2}. 
Theorems \ref{thm2.1} through \ref{thm2.4} recall and exhibit elliptic, 
trigonometric and related truncated solutions, and Theorems \ref{thm2.6}, 
\ref{thm2.8} and \ref{thm2.9} provide another new kind of truncated solutions. 
Regions of the monodromy manifold, which the corresponding
monodromy data fall on, are described in Theorems \ref{thm2.10}, \ref{thm2.11} 
and \ref{thm2.12}, showing that, under the generic condition, 
the monodromy manifold is filled up with the monodromy data of these solutions.
Concerning analytic continuations of
solutions, Subsection \ref{ssc2.5} formulates nonlinear monodromy-Stokes 
structure described in terms of the monodromy data, and
presents some transforms in Theorems \ref{thm2.14} through 
\ref{thm2.17} including nonlinear monodromy and Stokes relations.  
In our framework we may obtain the character variety, on which
the actions for these nonlinear operations are described. 
Subsection \ref{ssc2.6} illustrates applications of our results, in which
analytic continuations are calculated for some solutions, and consequently
the nonlinear monodromy-Stokes relations may be observed as changes 
of monodromy data
contained in the explicit expressions of solutions (e.g.~Theorems \ref{thm2.1}
and \ref{thm2.2}).
Section \ref{sc3} contains basic facts on the isomonodromy deformation and
monodromy data, which are used in the proofs in later sections.
Section \ref{sc4} is devoted to the proofs of Theorem \ref{thm2.14} and
\ref{thm2.15} about nonlinear Stokes and monodromy relations.
The truncated solutions in Theorems \ref{thm2.6} and \ref{thm2.8}
arise from an oscillatory family of solutions as $x\to \pm i\infty$ 
along the imaginary axes (cf. \cite{Andreev-Kitaev-2019, S-2018} and 
the references in \cite{Andreev-Kitaev-2019}, and 
for $\tau$-functions see \cite{L, N}), and their expressions are 
detailed forms for truncated solutions of \cite[Theorem 1.1]{Andreev}. 
In Section \ref{sc5} these new truncated solutions with the monodromy data are
derived from the asymptotic oscillatory solutions labelled with the monodromy 
data of Proposition \ref{prop5.1}, which are given in \cite{S-2018}. 
Such truncated solutions together with elliptic asymptotics
constitute a new additional part of the asymptotic behaviours. 
Proposition \ref{prop5.1} relies on the solutions of Schlesinger-type equation
\eqref{5.00}, and this fact enables us to calculate $\z$ and $u$
as well associated with these truncated solutions 
(cf. Proposition \ref{prop5.2}). 
The derivation of Theorem \ref{thm2.4} needs a uniqueness property, which
is shown in Section \ref{sc6} by auxiliary use of general solutions 
expressed by convergent series. In the final section, 
under no suppositions on $(\theta_0, \theta_1, \theta_{\infty})$, we derive  
the monodromy data of Proposition \ref{prop5.1} by a simplified process,
which improves that of \cite{S-2018}.  
\par
All classical solutions of (P$_{\mathrm{V}}$) in the sense of Umemura 
\cite{U} are known by \cite{K} for rational solutions and by \cite{W} for
Riccati solutions, and solvable linear monodromy cases are
studied by \cite{KO}. Many of classical solutions appear in non-generic cases,
which are partially discussed only (cf. Theorem \ref{thm2.9},
Remarks \ref{rem2.7}, \ref{rem2.13aaa} and \ref{rem2.18}, 
see also \cite[Section 9]{Andreev-Kitaev-2019}), and the study including 
such cases is incomplete. 
Our nonlinear monodromy-Stokes structure suggests some possible relations to
the wild monodromy \cite{Kli, PR}. 
It may be interesting to study dynamics generated by this structure, say, 
on the character variety 
or in relation to those of Painlev\'e VI \cite{I, I-2, IU, Kli, PR}.  
\section{Settings and Results}\label{sc2}
\subsection{Basic facts}\label{ssc2.1}
Let us recall some facts on elliptic asymptotics \cite[Sections 2 and 3]{S}, 
which is necessary in stating our results.  
To consider \eqref{1.1} along $\arg x=\phi$ with $|\phi|<\pi/2,$ 
and to convert \eqref{1.1} to a symmetric form, set
\begin{equation*}
x= e^{i\phi}t,\quad t>0, \qquad \xi=\tfrac 12(e^{-i\phi}\lambda +1).
\end{equation*}
In our arguments $t$ is permitted to move also in a strip $\re t>0,$ 
$|\im t|\ll 1.$ 
By $Y= \exp(-\varpi(t,\phi)\sigma_3)\Xi$ with
$\varpi(t,\phi)=\frac 14e^{i\phi}t+\frac 12\theta_{\infty}(i\phi +\ln 2)$ system
\eqref{1.1} is taken to
\begin{align}\label{2.1}
\frac{dY}{d\lambda} &= \hat{u}^{\sigma_3/2}\mathcal{B}(x,\lambda)
\hat{u}^{-\sigma_3/2}  Y,
\\ \notag
\mathcal{B}(x,\lambda)&=\mathcal{B}(t,\phi,\lambda)
=\frac t4\sigma_3 +\frac 1{\lambda+e^{i\phi}} \begin{pmatrix}
\z+ \frac 12 \theta_0  &  -\z -\theta_0  \\
\z   &  -\z -\frac 12 \theta_0         \end{pmatrix}
\\ \notag
&+\frac 1{\lambda-e^{i\phi}} \begin{pmatrix}
-\z- \frac 12 (\theta_0 +\theta_{\infty})  &  y(\z+\frac 12(\theta_0-\theta_1
+\theta_{\infty}))  \\
-y^{-1}(\z+\frac 12(\theta_0+\theta_1+\theta_{\infty})) & \z +\frac 12(\theta_0
+\theta_{\infty})     \end{pmatrix},
\end{align}
where $\hat{u}=u\exp(-2\varpi(t,\phi))$. 
Then (P$_{\mathrm{V}}$), i.e., system \eqref{1.2} for $(y,\mathfrak{z},u)$
governs the isomonodromy deformation of linear system \eqref{2.1} 
(cf. Subsection \ref{ssc3.1}).
For every $k\in \mathbb{Z}$, system \eqref{2.1} possesses the matrix 
solution $Y_k(t,\lambda)$ admitting the asymptotic expression
\begin{equation*}
Y_k(t,\lambda) =(I+O(\lambda^{-1})) \exp(\tfrac 14 (t\lambda-2\theta_{\infty}
\ln\lambda) \sigma_3 )
\end{equation*}
as $\lambda\to\infty$ through the sector $|\arg\lambda -\pi/2 -(k-2)\pi|<\pi.$ 
{\small
\begin{figure}[htb]
\begin{center}
\unitlength=0.73mm
\begin{picture}(55,60)(-10,-15)
\thicklines
\put(16,40){\circle*{1}}
\put(16,40){\line(-2,-5){14.13}}
\put(16,40){\line(1,-2){13.57}}
\put(0,0){\circle{10}}
\put(0,0){\circle*{1}}
\put(31.8, 8.2){\circle{10}}
\put(31.8, 8.2){\circle*{1}}
\thinlines
\put(12,25){\line(-2,-5){3}}
\put(12,25){\vector(1,3){0}}
\put(6,20.2){\line(2,5){3}}
\put(6,20.2){\vector(-1,-3){0}}
\qbezier(-2,-7)(0,-8)(2,-7)
\put(2,-7){\vector(3,2){0}}
\put(22,23.5){\line(-1,2){3}}
\put(23,30){\line(1,-2){3}}
\put(22,23.5){\vector(1,-2){0}}
\put(23,30){\vector(-1,2){0}}
\qbezier(34.7,1.9)(36.8,2.6)(37.6,4.7)
\put(37.6,4.7){\vector(1,4){0}}
\put(6,40){\makebox{$\hat{p}_{\mathrm{st}}$}}
\put(38.5,8.5){\makebox{$e^{i\phi}$}}
\put(-17,0){\makebox{$-e^{i\phi}$}}
\put(28.8,19){\makebox{$\hat{l}_{1}$}}
\put(0.5,15){\makebox{$\hat{l}_{0}$}}
\put(-22,-18){\makebox{(a) Loops $\hat{l}_0$ and $\hat{l}_1$ 
on the $\lambda$-plane}}
\end{picture}
\qquad\qquad
\unitlength=0.73mm
\begin{picture}(80,53)(0,-7)
\put(0,17){\makebox{$-1$}}
\put(22,3){\makebox{$-A_{\phi}^{1/2}$}}
\put(73,25){\makebox{$1$}}
\put(52,39){\makebox{$A_{\phi}^{1/2}$}}
\thinlines
\put(10,25.5){\line(2,-1){17}}
\put(10,24.5){\line(2,-1){17}}
\put(70,25.5){\line(-2,1){17}}
\put(70,24.5){\line(-2,1){17}}
\qbezier(35,33.5) (40,37) (46,38)
\qbezier(9,31.5) (14,32.7) (19,30.5)
\put(46,38){\vector(4,1){0}}
\put(9,31.5){\vector(-4,-1){0}}
\put(31,31){\makebox{$\mathbf{a}$}}
\put(20,29){\makebox{$\mathbf{b}$}}
\put(65,8){\makebox{$\Pi_{+}$}}
\thicklines
\put(10,25){\circle*{1}}
\put(27,16.5){\circle*{1}}
\put(53,33.5){\circle*{1}}
\put(70,25){\circle*{1}}
\qbezier(24,18) (25,25) (40,32.5)
\qbezier(40,32.5) (56,40) (56.5,31.7)
\qbezier[15](40,18) (54,25.5) (56,30)
\qbezier[15](24,16) (26,11) (40,18)
\qbezier(6,27) (10,32.5) (25.9,24.8)
\qbezier(29,23.2) (35.5,18) (34,13)
\qbezier(6,27) (3.5,21) (14,15)
\qbezier (14,15)(30,7)(34,13)
\put(12,-10){\makebox{(b) Cycles $\mathbf{a}$ and $\mathbf{b}$ on 
$\Pi_{A_{\phi}}$}}
\end{picture}
\end{center}
\caption{Loops on the $\lambda$-plane and cycles on $\Pi_{A_{\phi}}$}
\label{loops0}
\end{figure}
}
\par
Suppose that $|\phi|<\pi/2.$
Let loops $\hat{l}_0$ and $\hat{l}_1$ in the $\lambda$-plane start from 
$\hat{p}_{\mathrm{st}}$ and surround $\lambda=-e^{i\phi}$ and 
$\lambda=e^{i\phi}$, respectively, where $\arg \hat{p}_{\mathrm{st}}=\pi/2$,
$|\hat{p}_{\mathrm{st}}|>100$ (cf.~Figure \ref{loops0} (a)). 
Let $M^0$ and $M^1$ be the monodromy matrices such that $Y_2(\lambda)
|_{\hat{l}_0}=Y_2(\lambda)M^0$ and 
$Y_2(\lambda)|_{\hat{l}_1} =Y_2(\lambda)M^1$, where
$Y_2(\lambda)|_{\hat{l}_{\iota}} $ $(\iota=0, 1)$ denote the analytic 
continuation of $Y_2(\lambda)=Y_2(t,\lambda)$ along the respective loops 
$\hat{l}_{\iota}$. The Stokes matrices $S_k$'s are defined by 
$Y_{k+1}(\lambda)=Y_k(\lambda)S_k,$ and are written in the forms 
$$
S_{2l+1}=\begin{pmatrix}
1  &  0 \\  s_{2l+1}  & 1 
\end{pmatrix},  \quad
S_{2l}=\begin{pmatrix}
1  &  s_{2l} \\  0  & 1 
\end{pmatrix} \quad (l\in \mathbb{Z}).  
$$
Then $M^0,$ $M^1$, $S_k\in SL_2(\mathbb{C})$ are the same as in 
\cite[(2.3), (2.12)] {Andreev-Kitaev} (cf.~Proposition \ref{prop3.0}), 
which satisfy
\begin{align}\label{2.2}
&M^1M^0=S^{-1}_1 e^{-\pi i\theta_{\infty}\sigma_3} S_2^{-1},
\\  \label{2.3}
& S_{k+2}=e^{\pi i\theta_{\infty}\sigma_3} S_k e^{-\pi i\theta_{\infty}
\sigma_3}.
\end{align}
By \eqref{2.2}
entries of $M^{\iota}=(m^{\iota}_{ij})$ $(\iota=0,1)$ fulfil
\begin{align}\label{2.5}
&m^0_{11} m^1_{11}+m^0_{21}m^1_{12}= e^{-\pi i\theta_{\infty}}, \quad
\mathrm{tr}\,M^0=2\cos \pi \theta_0, \quad \mathrm{tr}\,M^1=2\cos \pi\theta_1,
\\
\notag
& -s_2e^{-\pi i\theta_{\infty}}=m^0_{12}m^1_{11}+m^0_{22}m^1_{12}, \quad
 -s_1e^{-\pi i\theta_{\infty}}=m^0_{11}m^1_{21}+m^0_{21}m^1_{22}. 
\end{align}  
Let $\mathcal{V}=\mathcal{V}(\theta_0,\theta_1,\theta_{\infty})$ be the 
set of all solutions of (P$_{\mathrm{V}}$).
Set 
\begin{align*}
\mathcal{M}(\theta_0,\theta_1,& \theta_{\infty}):=
\{(M^0,M^1)\in \mathrm{SL}_2(\mathbb{C})^2  \,|\, 
\\
&m^0_{11} m^1_{11}+m^0_{21}m^1_{12}= e^{-\pi i\theta_{\infty}}, \quad
\mathrm{tr}\,M^0=2\cos \pi \theta_0, \quad \mathrm{tr}\,M^1=2\cos \pi\theta_1\}
/\sim
\end{align*}
with the gauge equivalence relation $c^{\sigma_3}(M^0,M^1)
c^{-\sigma_3}\sim (M^0,M^1)$ for $c\in \mathbb{C}\setminus \{0\},$ which
we call the {\it monodromy manifold}. Then $\dim_{\mathbb{C}}\mathcal{M}
(\theta_0, \theta_1,\theta_{\infty})=2.$ 
The direct monodromy problem assigns
each $y(x)\in \mathcal{V}$ the corresponding $(M^0,M^1)\in \mathcal{M}(\theta_0,
\theta_1,\theta_{\infty}) $, which defines the surjection $\mathcal{V}\to
\mathcal{M}(\theta_0,\theta_1,\theta_{\infty})$ \cite{Bol, Put}. 
As long as $M^0,$ $M^1 \not= \pm I$, 
the inverse monodromy problem for $(M^0,M^1)$ is uniquely
soluble (cf.~Proposition \ref{prop3.1} or \cite[Proposition 2.1]
{Andreev-Kitaev}) and then one may uniquely 
find $y(x)\in \mathcal{V}$ corresponding to $(M^0,M^1).$ Namely,  
there exists a surjection
\begin{equation*}
\varrho: \mathcal{V} \to \mathcal{M}(\theta_0,\theta_1,\theta_{\infty}),
\end{equation*}
of which inverse $\varrho^{-1}$ is one-to-one on $\mathcal{M}(\theta_0,
\theta_1,\theta_{\infty}) \setminus \mathcal{M}_0,$ where
$\mathcal{M}_0= \{(M^0,M^1)\in \mathcal{M}(\theta_0,\theta_1,\theta_{\infty})\,
|\,M^0=\pm I \,\,\text{or} \,\, M^1=\pm I \}$. 
\par
As in the argument above with $|\phi|<\pi/2$, for every $(M^0,M^1)\in 
\mathcal{M}(\theta_0,\theta_1,\theta_{\infty})$ there exist a solution   
$y(x,M^0,M^1)\in \mathcal{V}$ such that $\varrho(y(x))=(M^0,M^1)$, 
and we call $y(x,M^0,M^1)$ a solution corresponding to
(or labelled with) $(M^0,M^1).$ If $(M^0,M^1) \in \mathcal{M}(\theta_0,\theta_1,
\theta_{\infty}) \setminus \mathcal{M}_0$, then $y(x,M^0,M^1)$ is uniquely
determined. The analytic continuation to $x=t e^{i\phi}$, $|\phi|\ge \pi/2$ 
is also denoted by the same symbol $y(x,M^0,M^1)$.
\begin{rem}\label{rem2.1}
In \cite{Andreev-Kitaev}, the {\it manifold of monodromy data} 
$\mathcal{M}_{5}(\theta_0,\theta_1,\theta_{\infty})$ is defined by
$\{(M^0,M^1)\in \mathrm{SL}_2(\mathbb{C})^2\,|\ 
m^0_{11} m^1_{11}+m^0_{21}m^1_{12}= e^{-\pi i\theta_{\infty}}, 
\mathrm{tr}\,M^0=2\cos \pi \theta_0,  \mathrm{tr}\,M^1=2\cos \pi\theta_1\}$
without the gauge equivalence relation $\sim$ such that $\dim_{\mathbb{C}}
\mathcal{M}_{5}(\theta_0,\theta_1,\theta_{\infty})=3$. 
To consider monodromy data
including the third parameter related to the function $u$ (cf. $c_0$ for
$u(c_0,c,\sigma, x)$ in Section \ref{sc5}) it is necessary to make such setting
of the manifold of monodromy data, on which gauge parameter multipliers are
distinguished.
\end{rem}
\subsection{Elliptic asymptotics}\label{ssc2.2}
For $A\in \mathbb{C}$ with $\re A^{1/2} \ge 0$ consider the elliptic
curve $\Pi_A=\Pi_+\cup\Pi_-$ given by $w(A, z) =\sqrt{(1-z^2)(A-z^2)}$, where
the sheets $\Pi_{\pm}=P^1(\mathbb{C})\setminus ([-1,-A^{1/2}]\cup[A^{1/2},1])$ 
are glued along the cuts $[-1,-A^{1/2}]$ and $[A^{1/2}, 1]$. 
Let the branches of $w(A,z)$ and $\sqrt{(A-z^2)/(1-z^2)}$ 
be such that $z^{-2} \sqrt{(A-z^2)(1-z^2)} \to -1$ and 
$ \sqrt{(A-z^2)/(1-z^2)} \to 1$ as $z\to \infty$ on the upper sheet
$\Pi_+.$ For a given number $\phi \in \mathbb{R}$ there exists a unique 
$A_{\phi}\in \mathbb{C}$ such that for every cycle $\mathbf{c}$ 
on $\Pi_{A_{\phi}}$
$$
\re e^{i\phi} \int_{\mathbf{c}} \sqrt{\frac{A_{\phi}-z^2}{1-z^2} } dz =0,
$$
and then $A_{\phi}$ has the following properties:
(i) $0 \le \re A_{\phi} \le 1$ for $\phi \in \mathbb{R}$, and $A_0=0,$
$A_{\pm \pi/2}=1;$
(ii) $A_{-\phi}=\overline{A_{\phi}},$ $A_{\phi\pm \pi}=A_{\phi}$ for $\phi\in
\mathbb{R};$ and
(iii) for $0 \le \phi \le \pi/2,$ $\im A_{\phi}\ge 0$,
and, for $-\pi/2 \le \phi \le 0,$ $\im A_{\phi}\le 0$.
In particular, for basic cycles $\mathbf{a}$ and $\mathbf{b}$ on 
$\Pi_{A_{\phi}}$ as in Figure \ref{loops0} (b), $A_{\phi}$ is a unique solution 
of the Boutroux equations
\begin{equation*}
\re e^{i\phi} \int_{\mathbf{a}} \sqrt{\frac{A_{\phi}-z^2}{1-z^2} } dz =
\re e^{i\phi} \int_{\mathbf{b}} \sqrt{\frac{A_{\phi}-z^2}{1-z^2} } dz =0
\end{equation*}
\cite[Section 7]{S}. Write
$$
\Omega_{\mathbf{a}}= \Omega_{\mathbf{a}}(\phi) 
=\int_{\mathbf{a}} \frac{dz}{w(A_{\phi}, z)},\quad
 \Omega_{\mathbf{b}} = \Omega_{\mathbf{b}}(\phi) =
\int_{\mathbf{b}} \frac{dz}{w(A_{\phi}, z)}, 
$$
which satisfy $\Omega_{\mathbf{a,b}}(\phi\pm \pi)= \Omega_{\mathbf{a,b}}(\phi),$ 
and let $\mathrm{sn}(u; k)$ denote the Jacobi $\mathrm{sn}$-function with 
modulus $k$. Then we have the following \cite[Theorem 2.1]{S}, \cite[arXiv, 
Theorem 2.1]{S}, in which the solution is denoted by $\operatorname{ellp}(x,
m^0_{21}m^1_{12}, m^0_{11},m^1_{11})$ describing the dependence on the
monodromy data.
\begin{thm}\label{thm2.1}
Suppose that $M^0=(m_{ij}^0)$ and $M^1=(m_{ij}^1)$ 
fulfil $m_{11}^0 m_{21}^0m_{12}^1 \not=0$ (respectively,  
$ m_{11}^1 m_{21}^0m_{12}^1 \not=0$). 
Then, for $-\pi/2<\phi<0$ (respectively, $0<\phi <\pi/2$),   
$y(x)=y(x,M^0,M^1)=
{\mathrm{ellp}}(x, m^0_{21}m^1_{12}, m^0_{11}, m^1_{11}) $ is represented by
\begin{equation*}
\frac{y(x) +1}{y(x)-1}
=A^{1/2}_{\phi} \mathrm{sn} ((x-x_0)/2 ; A^{1/2}_{\phi}) +O(x^{-1}),
\end{equation*}
as $x=e^{i\phi} t \to \infty$ through the cheese-like strip
\begin{align*}
& S(\phi, t_{\infty}, \kappa_0, \delta_0)= \{ x=e^{i\phi}t \,; \,\,
\re t > t_{\infty}, \,\, |\im t |<\kappa_0 \} \setminus \bigcup_{\rho\in
\mathcal{P}_0} \{|x-\rho|<\delta_0 \},
\\
&\mathcal{P}_0 = \{ \rho\,; \,\, \mathrm{sn}((\rho-x_0)/2 ; A_{\phi}^{1/2})
=\infty \}= \{ x_0 + \Omega_{\mathbf{a}}  \mathbb{Z}
 +\Omega_{\mathbf{b}}(2\mathbb{Z}+1) \},
\end{align*}
$\kappa_0 >0$ being a given number, $\delta_0>0$ a given small number,
and $t_{\infty}=t_{\infty}(\kappa_0,\delta_0)$ a large number depending on 
$(\kappa_0,\delta_0).$ 
The phase shift $x_0$ is such that 
$$
x_0 \equiv \frac {-1}{\pi i} \Bigl(\Omega_{\mathbf{b}} 
\ln(e^{\pi i\theta_{\infty}} m^0_{21}m^1_{12})
+ \Omega_{\mathbf{a}}\ln \mathfrak{m}_{\phi,\theta_{\infty}} \Bigr)
 -\Omega_{\mathbf{a}} -\Omega_{\mathbf{b}}   
\quad
\mod 2\Omega_{\mathbf{a}}\mathbb{Z}+ 2\Omega_{\mathbf{b}}\mathbb{Z},
$$
where $\mathfrak{m}_{\phi,\theta_{\infty}} = e^{\frac{\pi i}2\theta_{\infty}} 
m^0_{11}$ if $-\pi/2 <\phi<0,$ and
$= e^{-\frac{\pi i}2\theta_{\infty}} (m^1_{11})^{-1}$ if $0 <\phi<\pi/2.$
\end{thm}
\begin{rem}\label{rem2.2}
The error bound $O(x^{-1})$ is due to \cite[Theorem 2.2 and Remark 2.1]{S-2024},
which provides an explicit error term.
\end{rem}
\begin{rem}\label{rem2.3}
The complementary cases, say $m^1_{11}m^0_{21}m^1_{12}=0$ for $0<\phi<\pi/2$,
correspond to truncated or doubly-truncated (tri-tronqu\'ee) solutions 
\cite[Corollaries 5.2 and 5.3]{Andreev-Kitaev}, \cite[\S 5]{Andreev-Kitaev-2},
\cite[Theorem 1.1]{Andreev} and see Theorems \ref{thm2.4}, \ref{thm2.6} and
\ref{thm2.8}.  
\end{rem}
\begin{rem}\label{rem2.4}
Since $y(x)$ above or in \cite{Andreev-Kitaev} is analytic 
in a strip containing the ray $\arg x=\phi$, the derivative $y'(x)$ 
may be obtained by formal calculation with the same error bound.
Hence 
\begin{align*}
y'(x)=&-\frac{A^{1/2}_{\phi}\mathrm{sn}'((x-x_0)/2)}
{(A^{1/2}_{\phi}\mathrm{sn}\,((x-x_0)/2)-1)^2} +O(x^{-1})
\\
\z(x)=&-\frac{x(y'(x)-y(x))}{2(y(x)-1)^2} +\frac{\theta_0+\theta_1}{2(y(x)-1)}
-\frac 14(\theta_0-\theta_1+\theta_{\infty})
\\
=& \frac x8 (A^{1/2}_{\phi}\mathrm{sn}'((x-x_0)/2)+A_{\phi}\mathrm{sn}^2
((x-x_0)/2)-1) +O(1)   
\end{align*}
as $x\to \infty$ through $S(\phi,t_{\infty},\kappa_0,\delta_0),$ where
$\mathrm{sn}\,u=\mathrm{sn}(u;A^{1/2}_{\phi}),$ $\mathrm{sn}'(u)=\frac{d}{du}
\mathrm{sn}\,u=\mathrm{cn}\,u\,\mathrm{dn}\,u$.
\end{rem}
The analytic continuation of $y(x,M^0,M^1)$ to the left half plane with
$0<|\phi-\pi|<\pi/2$ is as follows.
\begin{thm}\label{thm2.2}
Let $(\breve{M}^0, \breve{M}^1)=  S_2^{-1}(M^0,M^1)S_2$ with
$\breve{M}^0=(\breve{m}^0_{ij})$ and $\breve{M}^1=(\breve{m}^1_{ij})$. 
Suppose that $\breve{m}^0_{22}\breve{m}^0_{12}\breve{m}^1_{21} \not=0$ 
(respectively, $\breve{m}^1_{22}\breve{m}^0_{12}\breve{m}^1_{21} \not=0$). 
Then, for $\pi/2<\phi<\pi$ (respectively, $\pi<\phi<3\pi/2$),  
$y(x)=y(x,M^0,M^1)$ is represented by
\begin{equation*}
\frac{y(x) +1}{y(x)-1}
=A^{1/2}_{\phi} \mathrm{sn} ((x-\breve{x}_0)/2 ; A^{1/2}_{\phi}) +O(x^{-1}),
\end{equation*}
as $x=e^{i\phi} t \to \infty$ through 
$ S(\phi, t_{\infty}, \kappa_0, \delta_0)$, where
$$
\breve{x}_0 \equiv \frac {-1}{\pi i} \Bigl(\Omega_{\mathbf{b}} 
\ln(e^{\pi i\theta_{\infty}} (\breve{m}^0_{12}\breve{m}^1_{21})^{-1})
+ \Omega_{\mathbf{a}}
 \ln \breve{\mathfrak{m}}_{\phi,\theta_{\infty}} \Bigr)
 -\Omega_{\mathbf{a}}    -\Omega_{\mathbf{b}}   
\quad
\mod 2\Omega_{\mathbf{a}}\mathbb{Z}+ 2\Omega_{\mathbf{b}}\mathbb{Z},
$$
where $\breve{\mathfrak{m}}_{\phi,\theta_{\infty}} 
=e^{\frac{\pi i}2\theta_{\infty}}(\breve{m}^0_{22})^{-1}$ if $\pi/2 <\phi<\pi,$ 
and
$= e^{-\frac{\pi i}2\theta_{\infty}} \breve{m}^1_{22}$ if $\pi <\phi<3\pi/2.$
\end{thm}
In Subsection \ref{ssc2.6}, Theorem \ref{thm2.2} is derived from 
Theorem \ref{thm2.1}, which relies on the nonlinear Stokes relation. 
This is also proved by WKB analysis \cite[arXiv, Theorem 2.2]{S}. 
In \cite[Corrigendum, Theorem 2.1]{S}, add 
$-\Omega_{\mathbf{a}}/2$ to the phase shift $x_0$, in 
\cite[Corrigendum, Theorem 2.2]{S}, read $(\breve{m}^0_{22})^{-1}$ and 
$e^{-\pi i\theta_{\infty}} \breve{m}^1_{22}$ instead of
$e^{-\pi i\theta_{\infty}}\breve{m}^0_{22}$ and $(\breve{m}^1_{22})^{-1}$, 
and for the full complete proof of Theorem \ref{thm2.1} see \cite[arXiv]{S}.
\subsection{Trigonometric asymptotics and truncated cases}\label{ssc2.3}
On the positive real axis a general solution behaves trigonometrically 
as shown by Andreev and Kitaev
\cite[Theorems 3.1--3.4, 4.1--4.3, Remarks 3.1, 4.2, 4.3]{Andreev-Kitaev},
which are denoted by $y(x)=y(x,M^0,M^1)=\operatorname{trig}(x,m^0_{21}m^1_{12},
m^0_{11},m^1_{11})$.
\begin{thm}\label{thm2.3}
Suppose that $m^0_{11}m^1_{11}m^0_{21}m^1_{12}\not=0.$
Let 
\begin{align*}
\beta_0&=\frac 1{2\pi i}\ln(1-m^0_{11}m^1_{11}e^{\pi i\theta_{\infty}})
=\frac 1{2\pi i}\ln(m^0_{21}m^1_{12} e^{\pi i\theta_{\infty}})
\quad (\text{cf.~\eqref{2.5}}), 
\\
\hat{v}& =-\frac{\sqrt{2\pi}}{m^0_{11} \Gamma(\beta_0) }\exp(\beta_0\ln 2
-\tfrac 12 \pi i \theta_{\infty} +\tfrac 12 \pi i\beta_0).
\end{align*}
Then $y(x)={\mathrm{trig}}(x,m^0_{21}m^1_{12},m^0_{11}, m^1_{11})$
admits the following expressions as $x\to +\infty$.
\par
$(1)$ If $|\re \beta_0|<\tfrac 14$, then
$$
y(x)=-1+ 4\sqrt{2}e^{-\pi i/4} \beta_0^{1/2} x^{-1/2}\sin \left(\tfrac 12
x+i\beta_0 \ln x +i \ln(\beta_0^{-1/2}\hat{v}) \right)
+O(x^{-1+2|\re \beta_0|})
$$
as $x\to +\infty.$ This is also rewritten in the form:
if $\tfrac 16 <\re \beta_0 <\tfrac 12,$
$$
y(x)=-1+2\sqrt{2} e^{\pi i/4}\hat{v} x^{\beta_0-1/2} e^{-ix/2} +O(x^{-1+2\re
\beta_0}),
$$
and if $-\tfrac 12 <\re \beta_0 <-\tfrac 16$,
$$
y(x)=-1-2\sqrt{2} e^{\pi i/4}\hat{v}^{-1}\beta_0 x^{-\beta_0-1/2} e^{ix/2} 
+O(x^{-1+2|\re \beta_0}|).
$$
\par
$(2)$ If $\tfrac 14 <\re \beta_0 <\tfrac 34$, then
$$
y(x)=\frac{\cos^2 \tilde{x}}{\sin^2 \tilde{x}}+O(x^{-1+2|\re \beta_0-1/2|}),
\quad
\tilde{x}=\frac x4+\frac{1-2\beta_0}{4i}\ln x -\frac 1{2i} \ln (-e^{\pi i/4}
\hat{v}/\sqrt{2})
$$
as $x\to +\infty$ avoiding poles. This is also  
rewritten in the form: if $\tfrac 14 \le \re \beta_0 <\tfrac 12$, 
$$
y(x)=-1+2\sqrt{2} e^{\pi i/4}\hat{v} x^{\beta_0-1/2} e^{-ix/2} +O(x^{-1+2\re
\beta_0}),
$$
and if $\tfrac 12 <\re \beta_0 \le \tfrac 34$, 
$$
y(x)=-1+4\sqrt{2} e^{-\pi i/4}\hat{v}^{-1} x^{-\beta_0+1/2} e^{ix/2} +O(x^{1
-2\re\beta_0}).
$$
\end{thm}
Truncated solutions  
\cite[Corollaries 5.1--5.3]{Andreev-Kitaev}, \cite{Andreev-Kitaev-2} 
arising from the trigonometric case are described by $\operatorname{trunc}
(x,m^0_{11},m^1_{11})$ omitting $m^0_{21}m^1_{12}=e^{-\pi i\theta_{\infty}}$.
For the derivation of the following forms and the validity
in each sector, see Subsection \ref{ssc6.1} and \cite{S0}. 
\begin{thm}\label{thm2.4}
$(1)$ If $m^0_{11}=m^1_{11}=0,$ $m^0_{21}m^1_{12}
=e^{-\pi i\theta_{\infty}},$ then 
$$
y(x)={\mathrm{trunc}}(x, 0, 0)
= -1+4(\theta_0+\theta_1-1)x^{-1} +O(x^{-2})
$$
as $x\to\infty$ through the sector $-\pi <\arg x<\pi.$
\par
$(2)$ If $m^0_{11}=0,$ $m^1_{11}=i\sqrt{2\pi}\hat{v} e^{-\frac{\pi i}2
\theta_{\infty}}\not=0,$ $m^0_{21}m^1_{12} =e^{-\pi i\theta_{\infty}}$, then
\begin{align*}
y(x)=&{\mathrm{trunc}}(x, 0, m^1_{11}) 
\\
=&{\mathrm{trunc}}(x, 0, 0)
+2\sqrt{2}e^{\pi i/4}\hat{v} x^{-1/2}e^{-ix/2}(1 +O(x^{-1/2}))
\end{align*}
as $x\to \infty$ through the sector $-\pi<\arg x\le 0.$ 
\par
$(3)$ If $m^0_{11}=\sqrt{2\pi}\hat{w}e^{-\frac{\pi i}2\theta_{\infty}}\not=0,$
$m^1_{11}=0,$ $m^0_{21}m^1_{12}=e^{-\pi i\theta_{\infty}}$, then
\begin{align*}
y(x)=&{\mathrm{trunc}}(x, m^0_{11}, 0) 
\\
=&{\mathrm{trunc}}(x, 0, 0)
+2\sqrt{2}e^{\pi i/4}\hat{w} x^{-1/2}e^{ix/2}(1 +O(x^{-1/2}))
\end{align*}
as $x\to \infty$ through the sector $0\le \arg x<\pi.$
\par
\end{thm}
\begin{prop}\label{prop2.5}
The doubly-truncated solution above admits the asymptotic representation
$\mathrm{trunc}(x, 0, 0) \sim -1
+4(\theta_0+\theta_1-1)x^{-1}+\sum_{j=2}^{\infty}c_j x^{-j}$ 
as $x\to \infty$ through the sector $|\arg x|<\pi$, where the 
right-hand member is a formal solution of $(\mathrm{P}_{\mathrm{V}})$. 
\end{prop}
\begin{rem}\label{rem2.5}
These truncated solutions also arise from  
Theorem \ref{thm2.3} as limit cases, say,
$\mathrm{trunc}(x, 0, m^1_{11}) 
=\lim_{m^0_{11}\to 0}\mathrm{trig}(x,m^0_{21}m^1_{12},m^0_{11}, m^1_{11})$,
in which $4(\theta_0+\theta_1-1)x^{-1}$ is hidden when considered on the 
positive real axis \cite[Corollary 5.1]{Andreev-Kitaev}, 
see also Section \ref{sc6}. 
\end{rem}
The following two theorems present detailed expressions for 
\cite[Theorem 1.1]{Andreev} revealing exponentially small terms.
These are another kind of truncated solutions denoted by $\operatorname{trunc}
^0_{0,\,\infty}(x,m^1_{21}/m^0_{21},e^{\pm \pi i\theta_1})$ and
$\operatorname{trunc}^1_{0,\,\infty}(x,m^0_{12}/m^1_{12},e^{\mp\pi i\theta_0})$, 
and the corresponding monodromy data are such that $(m^1_{11},m^1_{12})=
(e^{\pm\pi i\theta_1},0)$ and $(m^0_{11},m^0_{21})=(e^{\mp\pi i\theta_0},0).$
In these theorems $r$ denotes a given number such that $0<r\le 1,$ and for
the proofs, see Subsection \ref{ssc5.2}. 
\begin{thm}\label{thm2.6}
$(1)$ Suppose that 
$\theta_0-\theta_1 -\theta_{\infty}\not\in 2\mathbb{N}\cup\{0\}$, 
$\theta_0+\theta_1+\theta_{\infty} \not\in -2\mathbb{N}\cup\{0\}$ and
$\theta_1 \not\in \mathbb{N}\cup\{0\}.$ Let 
$$
M^0=\begin{pmatrix}
 e^{-\pi i(\theta_1+\theta_{\infty})} & m^0_{12} \\
\frac{2\pi i e^{-\pi i\theta_{\infty}}} 
{\Gamma(-\frac 12(\theta_0-\theta_1-\theta_{\infty}))
\Gamma(\frac 12(\theta_0+\theta_1+\theta_{\infty}))  }c_0^{-1} & m^0_{22}
\end{pmatrix},
\quad 
M^1=\begin{pmatrix}
e^{\pi i\theta_1} &  0 \\ \frac{2\pi i e^{\pi i\theta_1}}
{\Gamma(1-\theta_1)}c_x^{-1}
  & e^{-\pi i\theta_1}  
\end{pmatrix}.
$$
Then the corresponding solution 
$y(x)= \mathrm{trunc}_0^0(x, m^1_{21}/m^0_{21}, e^{\pi i\theta_1})$
is given by
\begin{align*}
xy(x)
&=\frac12 (\theta_0-\theta_1-\theta_{\infty}) (1+O(x^{-1}))-c^-
x^{2\theta_1+\theta_{\infty}-1}e^{-x}(1 +O(x^{-r})),
\\
&c^-=\frac{c_0}{c_x}= \frac{e^{-\pi i(\theta_1+\theta_{\infty})}
 \Gamma(1-\theta_1)} {\Gamma(-\frac {\theta_0-\theta_1-\theta_{\infty}}2)
 \Gamma(\frac {\theta_0+\theta_1+\theta_{\infty}} 2)  } 
 \frac{m^1_{21}}{m^0_{21}}
\end{align*}
as $x\to \infty$ through the sector $-\pi/2<\arg x\le \pi/2$ with 
$c^-x^{2\theta_1+\theta_{\infty}-1}e^{-x}\ll x^{-r}.$
If $c^-=0,$ then
$x \operatorname{trunc}_0^0(x, 0, e^{\pi i\theta_1} )
=\tfrac 12 (\theta_0-\theta_1-\theta_{\infty})(1 +O(x^{-1}))$
is valid in $|\arg x-\pi/2|<\pi.$
\par
$(2)$ Suppose that 
$\theta_0-\theta_1-\theta_{\infty} \not\in -2\mathbb{N}\cup\{0\}$, 
$\theta_0+\theta_1 -\theta_{\infty}\not\in - 2\mathbb{N}\cup\{0\}$ and
$\theta_0\not\in \mathbb{N}\cup\{0\}.$ Let 
$$
M^0=\begin{pmatrix}
e^{-\pi i\theta_0}  & \frac{2\pi i e^{\pi i(\theta_{\infty}
-\theta_0)} } {\Gamma(1-\theta_0)}c_0
 \\  0  & e^{\pi i\theta_0}  
\end{pmatrix}, \quad
M^1=\begin{pmatrix}
 e^{\pi i(\theta_0-\theta_{\infty})} &    \frac{2\pi i }
{\Gamma(\frac 12(\theta_0-\theta_1-\theta_{\infty}))
\Gamma(\frac 12(\theta_0+\theta_1-\theta_{\infty})) } c_x \\
m^1_{21} & m^1_{22}
\end{pmatrix}.
$$
Then the corresponding solution 
$y(x)=\mathrm{trunc}_0^1(x,m^0_{12}/m^1_{12},e^{-\pi i\theta_0})$
is given by
\begin{align*}
xy(x)
&= -\frac 12 (\theta_0-\theta_1-\theta_{\infty})(1+O(x^{-1}))-c^-
x^{2\theta_0-\theta_{\infty}-1}e^{-x}(1 +O(x^{-r})),
\\
&c^-=\frac{c_0}{c_x}= \frac
{e^{\pi i(\theta_0-\theta_{\infty})}\Gamma(1-\theta_0)}
{ \Gamma(\frac {\theta_0-\theta_1-\theta_{\infty}}2)
\Gamma(\frac{\theta_0+\theta_1-\theta_{\infty}}2) }\frac{m^0_{12}}{m^1_{12}}
\end{align*}
as $x\to \infty$ through the sector $-\pi/2\le \arg x<\pi/2$ with 
$c^-x^{2\theta_0-\theta_{\infty}-1}e^{-x}\ll x^{-r}.$
If $c^-=0,$ then
$ x\,\mathrm{trunc}_0^1(x, 0, e^{-\pi i\theta_0} )
=-\tfrac 12 (\theta_0-\theta_1-\theta_{\infty})(1 +O(x^{-1}))$
is valid in $|\arg x+\pi/2|<\pi.$
\end{thm}
\begin{thm}\label{thm2.8}
$(1)$ Suppose that 
$\theta_0+\theta_1-\theta_{\infty} \not\in 2\mathbb{N}$,
$\theta_0-\theta_1 +\theta_{\infty}\not\in -2\mathbb{N} \cup\{0\}$ and
$\theta_1 \not\in -\mathbb{N} \cup\{0\}.$ 
Let 
$$
M^0=\begin{pmatrix}
 e^{\pi i(\theta_1-\theta_{\infty})} & m^0_{12} \\
\frac{2\pi i e^{-\pi i\theta_{\infty}}}
{\Gamma(1-\frac 12(\theta_0+\theta_1-\theta_{\infty}))
\Gamma(\frac 12(\theta_0-\theta_1+\theta_{\infty})) } c_0^{-1} & m^0_{22}
\end{pmatrix},
\quad 
M^1=\begin{pmatrix}
e^{-\pi i\theta_1} &  0 \\ \frac{ 2\pi ie^{-\pi i\theta_1}}
{\Gamma(\theta_1)} c_x^{-1}
  & e^{\pi i\theta_1}  
\end{pmatrix}.
$$
Then the corresponding solution 
$y(x)=\mathrm{trunc}^0_{\infty}(x,m^1_{21}/m^0_{21},e^{-\pi i\theta_1})$
is given by
\begin{align*}
\frac x{y(x)}
&=\frac 12 (\theta_1-\theta_0-\theta_{\infty})(1 +O(x^{-1})) + c^-
x^{1- 2\theta_1+\theta_{\infty}}e^{-x}(1 +O(x^{-r})),
\\
&c^-=\frac{c_0}{c_x} 
=\frac { e^{\pi i(\theta_{1}-\theta_{\infty})}\Gamma(\theta_1)}
{\Gamma(1-\frac {\theta_0+\theta_1-\theta_{\infty}}2)
\Gamma(\frac {\theta_0-\theta_1+\theta_{\infty}}2)} 
\frac{m^1_{21}}{m^0_{21}}
\end{align*}
as $x\to \infty$ through the sector $-\pi/2<\arg x\le \pi/2$ with 
$c^-x^{-1-2\theta_1+\theta_{\infty}}e^{-x}\ll x^{-r}.$
If $c^-=0,$ then
$x/ \mathrm{trunc}^0_{\infty}(x, 0, e^{-\pi i\theta_1} )
= \frac 12 (\theta_1-\theta_0-\theta_{\infty})(1 +O(x^{-1}))$
is valid in $|\arg x-\pi/2|<\pi.$
\par
$(2)$ Suppose that 
$\theta_0+\theta_1+\theta_{\infty} \not\in 2\mathbb{N},$ 
$\theta_0-\theta_1 +\theta_{\infty}\not\in 2\mathbb{N}\cup\{0\}$ and 
$\theta_0 \not\in -\mathbb{N}\cup\{0\}.$ 
Let 
$$
M^0=\begin{pmatrix}
e^{\pi i\theta_0}  & \frac{ 2\pi ie^{\pi i(\theta_{\infty}+\theta_0)}}
{\Gamma(\theta_0)}c_0
 \\  0  & e^{-\pi i\theta_0}  
\end{pmatrix}, \quad
M^1=\begin{pmatrix}
 e^{-\pi i(\theta_0+\theta_{\infty})} &    \frac{2\pi i }
{\Gamma(1-\frac 12(\theta_0+\theta_1+\theta_{\infty}))
\Gamma(-\frac 12(\theta_0-\theta_1+\theta_{\infty})) }c_x \\
m^1_{21} & m^1_{22}
\end{pmatrix}.
$$
Then the corresponding solution 
$y(x)=\mathrm{trunc}^1_{\infty}(x, m^0_{12}/m^1_{12}, e^{\pi i\theta_0})$
is given by
\begin{align*}
\frac x{y(x)}
&=\frac 12 (\theta_0-\theta_1+\theta_{\infty})(1+O(x^{-1}))+c^-
x^{1-2\theta_0-\theta_{\infty}}e^{-x}(1 +O(x^{-r})), 
\\
&c^-=\frac{c_0}{c_x} =
\frac{ e^{-\pi i(\theta_{\infty}+\theta_{0}) } \Gamma(\theta_0)}
{\Gamma(1-\frac {\theta_0+\theta_1+\theta_{\infty}}2)
\Gamma(-\frac{\theta_0-\theta_1+\theta_{\infty}}2)}\frac{m^0_{12}}{m^1_{12}}
\end{align*}
as $x\to \infty$ through the sector $-\pi/2\le \arg x<\pi/2$ with 
$c^-x^{-1-2\theta_0-\theta_{\infty}}e^{-x}\ll x^{-r}.$
If $c^-=0$, then
$x/ \mathrm{trunc}^1_{\infty}(x, 0, e^{\pi i\theta_0})
= \frac 12 (\theta_0-\theta_1+\theta_{\infty})(1 +O(x^{-1}))$
is valid in $|\arg x+\pi/2|<\pi.$
\end{thm}
\begin{rem}\label{rem2.60}
In Theorems \ref{thm2.6} and \ref{thm2.8} by the suppositions on 
$(\theta_0,\theta_1,\theta_{\infty})$, 
the $(y,\z)$ is uniquely determined for each monodromy data (cf. Proposition 
\ref{prop3.1}), and the proper coordinates of the monodromy manifold are
proportional to $c^-$ (cf. Subsection \ref{ssc2.4}). (For cases without  
the suppositions, cf. Theorem \ref{thm2.9}, 
Remark \ref{rem2.7}, \cite[Remark 1]{Andreev-Kitaev-2019}).
\end{rem}
\begin{rem}\label{rem2.6}
The doubly-truncated solutions $\mathrm{trunc}^0_0(x,0,e^{\pi i\theta_1})$ 
and $\mathrm{trunc}^0_{\infty}(x,0,e^{-\pi i\theta_1})$ of Theorems 
\ref{thm2.6},(1) and \ref{thm2.8},(1) are derived by conventional setting 
$c_x^{-1}=0$, which leads to $m^1_{21}=0$ and $c^-=c_0c_x^{-1}=0$; and 
$\mathrm{trunc}^1_0(x,0,e^{-\pi i\theta_0})$ and
$\mathrm{trunc}^1_{\infty}(x,0,e^{\pi i\theta_0})$ with $m^0_{12}=0$ are
obtained by $c_0=0.$
\end{rem}
\begin{prop}\label{prop2.7}
The doubly-truncated solution, say, of Theorem $\ref{thm2.6},(1)$,  
admits the asymptotic representation  
$ \mathrm{trunc}^0_0(x,0,e^{\pi i\theta_1})
 \sim \frac 12 (\theta_0-\theta_1-\theta_{\infty})( x^{-1} +\sum_{j=2}^{\infty}
c_jx^{-j})$ as $x\to \infty$ through the sector $|\arg x-\pi/2| <\pi$, where
the right-hand member is a formal solution of $(\mathrm{P}_{\mathrm{V}})$. 
In the sector $-\pi/2<\arg x \le \pi/2$ with 
$c^-x^{2\theta_1+\theta_{\infty}-1}e^{-x}\ll x^{-r},$ we have
$ \mathrm{trunc}^0_0(x,m^1_{21}/m^0_{21},e^{\pi i\theta_1})
= \mathrm{trunc}^0_0(x,0,e^{\pi i\theta_1})
 +c^-x^{2\theta_1+\theta_{\infty}-2}e^{-x}(1+O(x^{-r})).$ 
\end{prop}
In non-generic cases as well there exist truncated solutions as in the following
theorem, whose proof is given in Subsection \ref{ssc5.3}.
\begin{thm}\label{thm2.9}
Let $\nu\in\mathbb{N}$ denote a given positive number.
\par
$(1)$ Let $\theta_1\not \in \mathbb{N}\cup\{0\}$ and 
$\theta_0\not\in\mathbb{Z}.$ 
Suppose that $\theta_0-\theta_1-\theta_{\infty} = 2(\nu-1)$ or
$\theta_0+\theta_1+\theta_{\infty} = -2(\nu-1)$. Then
the monodromy data
$$
M^0=\begin{pmatrix}
 m^0_{11} & m^0_{12} \\ 0 &  (m^0_{11})^{-1}
\end{pmatrix},
\quad 
M^1=\begin{pmatrix}
e^{\pi i\theta_1} &  0 \\  
\frac{2\pi ie^{\pi i\theta_1}}
{\Gamma(1-\theta_1) }c_x^{-1}
& e^{-\pi i\theta_1} 
\end{pmatrix}
$$
correspond to $y(x)= \mathrm{trunc}_{00}^0(x, m^0_{12}m^1_{21}, m^0_{11})$ 
such that
\begin{equation*}
xy(x)
=\frac 12 (\theta_0-\theta_1-\theta_{\infty})(1+O(x^{-1}))+c^-
x^{2\theta_1+\theta_{\infty}-1}e^{-x}(1 +O(x^{-r}))
\end{equation*}
in the sector $-\pi/2<\arg x\le \pi/2$ with 
$c^-x^{2\theta_1+\theta_{\infty}-1}e^{-x}\ll x^{-r},$
where $m^0_{11}=e^{-\pi i\theta_0},$ $m^0_{12}=-\frac{2\pi i e^{-\pi i\theta_1}
}{\Gamma(\nu)\Gamma(\nu-\theta_0)}c_0$,
$c^-=-\frac{\Gamma(1-\theta_1)\Gamma(\nu)\Gamma(\nu-\theta_0)}{(2\pi i)^2}
m_{12}^0m_{21}^1$
if $\theta_0-\theta_1-\theta_{\infty}=2(\nu-1),$ and
$m^0_{11}=e^{\pi i\theta_0},$ $m^0_{12}=- \frac{2\pi i e^{-\pi i \theta_1}}
{\Gamma(\nu)\Gamma(\theta_0+\nu)}c_0$,
$c^-=-\frac{\Gamma(1-\theta_1)\Gamma(\nu)\Gamma(\theta_0+\nu)}{(2\pi i)^2}
m^0_{12}m^1_{21}$
if $\theta_0+\theta_1+\theta_{\infty}=-2(\nu-1).$ 
\par 
$(2)$ Let $\theta_0\not\in \mathbb{N}\cup\{0\}$ and $\theta_1\not\in\mathbb{Z}$.
Suppose that 
$\theta_0-\theta_1-\theta_{\infty} = -2(\nu-1)$ or
$\theta_0+\theta_1 -\theta_{\infty}= - 2(\nu-1)$. 
Then the monodromy data 
$$
M^0=\begin{pmatrix}
e^{-\pi i\theta_0}  & \frac{ 2\pi ie^{\pi i(\theta_{\infty}
-\theta_0)}}{\Gamma(1-\theta_0)}c_0 \\
 0  & e^{\pi i\theta_0}  
\end{pmatrix}, \quad
M^1=\begin{pmatrix}
 m^1_{11} &   0  \\ 
m^1_{21} & (m^1_{11})^{-1}
\end{pmatrix}
$$
correspond to 
$y(x)= \mathrm{trunc}_{00}^1(x, m^0_{12}m^1_{21}, m^1_{11})$ such that
\begin{equation*}
xy(x)=
 -\frac 12 (\theta_0-\theta_1-\theta_{\infty})(1+O(x^{-1}))+c^-
x^{2\theta_0-\theta_{\infty}-1}e^{-x}(1 +O(x^{-r}))
\end{equation*}
in the sector $-\pi/2 \le\arg x <\pi/2$ with 
$c^-x^{2\theta_0-\theta_{\infty}-1}e^{-x}\ll x^{-r}$,
where $m^1_{11}=e^{\pi i\theta_1},$ $m^1_{21}=-\frac{2\pi i e^{\pi i(\theta_0
-\theta_{\infty})}}{\Gamma(\nu)\Gamma(\nu-\theta_1)}c_x^{-1}$,
$c^-=-\frac{\Gamma(1-\theta_0)\Gamma(\nu)\Gamma(\nu-\theta_1)}{(2\pi i)^2}
m^0_{12}m^1_{21}$
if $\theta_0-\theta_1-\theta_{\infty}=-2(\nu-1),$ and
$m^1_{11}=e^{-\pi i\theta_1},$ $m^1_{21}=-\frac{2\pi i e^{\pi i(\theta_0-\theta
_{\infty})}}{\Gamma(\nu)\Gamma(\theta_1+\nu)}c_x^{-1}$,
$c^-=-\frac{\Gamma(1-\theta_0)\Gamma(\nu)\Gamma(\theta_1+\nu)}{(2\pi i)^2}
m^0_{12}m^1_{21}$
if $\theta_0+\theta_1-\theta_{\infty}=-2(\nu-1).$ 
\par
$(3)$ Let $\theta_1 \not\in -\mathbb{N} \cup\{0\}$ and $\theta_0\not\in \mathbb
{Z}.$ 
Suppose that $\theta_0+\theta_1-\theta_{\infty} = 2\nu$ or
$\theta_0-\theta_1 +\theta_{\infty}= -2(\nu-1)$. 
Then the monodromy data 
$$
M^0=\begin{pmatrix}
 m^0_{11} & m^0_{12} \\ 0  & (m^0_{11})^{-1}
\end{pmatrix}, \quad
M^1=\begin{pmatrix}
e^{-\pi i\theta_1} &  0 \\  
\frac{2\pi i e^{-\pi i\theta_1}}{\Gamma(\theta_1)} c_x^{-1}
  & e^{\pi i\theta_1}  
\end{pmatrix}
$$
correspond to 
$y(x)=\mathrm{trunc}^0_{\infty 0}(x, m^0_{12}m^1_{21}, m^0_{11})$ such that
\begin{equation*}
\frac x{y(x)}
= \frac 12 (\theta_1-\theta_0-\theta_{\infty})(1+O(x^{-1})) +c^-
x^{1- 2\theta_1+\theta_{\infty}}e^{-x}(1 +O(x^{-r})) 
\end{equation*}
in the sector $-\pi/2<\arg x\le \pi/2$ with 
$c^-x^{-1-2\theta_1+\theta_{\infty}}e^{-x}\ll x^{-r}$,
where $m^0_{11}=e^{-\pi i\theta_0},$ $m^0_{12}=\frac{2\pi i e^{\pi i
\theta_1}}{\Gamma(\nu)\Gamma(1-\theta_0+\nu)}c_0$,
$c^-=\frac{\Gamma(\theta_1)\Gamma(\nu)\Gamma(1-\theta_0+\nu)}{(2\pi i)^2}
m^0_{12}m^1_{21}$
if $\theta_0+\theta_1-\theta_{\infty}=2\nu,$ and
$m^0_{11}=e^{\pi i\theta_0},$ $m^0_{12}=\frac{2\pi i e^{\pi i\theta_1}}
{\Gamma(\nu)\Gamma(\theta_0+\nu-1)}c_0$,
$c^-=\frac{\Gamma(\theta_1)\Gamma(\nu)\Gamma(\theta_0+\nu-1)}{(2\pi i)^2}
m^0_{12}m^1_{21}$
if $\theta_0-\theta_1+\theta_{\infty}=-2(\nu-1).$ 
\par
$(4)$ Let $\theta_0 \not\in -\mathbb{N}\cup\{0\}$ and $\theta_1\not\in \mathbb
{Z}$. 
Suppose that 
$\theta_0+\theta_1+\theta_{\infty} = 2\nu$ or 
$\theta_0-\theta_1 +\theta_{\infty}= 2(\nu-1)$. 
Then the monodromy data 
$$
M^0=\begin{pmatrix}
e^{\pi i\theta_0}  &  \frac{ 2\pi ie^{\pi i(\theta_{\infty}+
\theta_0)}}{\Gamma(\theta_0)}c_0 \\
 0  & e^{-\pi i\theta_0}  
\end{pmatrix}, \quad
M^1=\begin{pmatrix}
 m^1_{11} &   0 \\
m^1_{21} &  (m^1_{11})^{-1}
\end{pmatrix}
$$
correspond to 
$y(x)= \mathrm{trunc}^1_{\infty 0}(x, m^0_{12}m^1_{21}, m^1_{11})$ such that
\begin{equation*}
\frac x{y(x)}
=\frac 12 (\theta_0-\theta_1+\theta_{\infty})(1 +O(x^{-1}))+c^-
x^{1-2\theta_0-\theta_{\infty}}e^{-x}(1 +O(x^{-r})) 
\end{equation*}
in the sector $-\pi/2 \le \arg x<\pi/2$ with 
$c^-x^{-1-2\theta_0-\theta_{\infty}}e^{-x}\ll x^{-r}$,
where $m^1_{11}=e^{\pi i\theta_1},$ $m^1_{21}=\frac{2\pi i e^{-\pi i(\theta_0
+\theta_{\infty})}}{\Gamma(\nu)\Gamma(1-\theta_1+\nu)}c_x^{-1}$,
$c^-=\frac{\Gamma(\theta_0)\Gamma(\nu)\Gamma(1-\theta_1+\nu)}{(2\pi i)^2}
m^0_{12}m^1_{21}$
if $\theta_0+\theta_1+\theta_{\infty}=2\nu,$ and
$m^1_{11}=e^{-\pi i\theta_1},$ $m^1_{21}=\frac{2\pi i e^{-\pi i(\theta_0
+\theta_{\infty})}}{\Gamma(\nu)\Gamma(\theta_1+\nu-1)} c_x^{-1}$,
$c^-=\frac{\Gamma(\theta_0)\Gamma(\nu)\Gamma(\theta_1+\nu-1)}{(2\pi i)^2}
m^0_{12}m^1_{21}$
if $\theta_0-\theta_1+\theta_{\infty}=2(\nu-1).$ 
\end{thm}
\begin{rem}
Each assertion of Theorem \ref{thm2.9} remains valid under a weaker condition.
In the assertion (1), the condition $\theta_0\not\in \mathbb{Z}$ of the case
$\theta_0-\theta_1-\theta_{\infty}=2(\nu-1)$ (respectively, $\theta_0+
\theta_1+\theta_{\infty}=-2(\nu-1)$) may be replaced with $\theta_0-\nu\not\in
\mathbb{N}\cup\{0\}$ (respectively, $\theta_0+\nu \not\in -\mathbb{N}\cup\{0\}$).
\end{rem}
\begin{rem}\label{rem2.7}
When $\theta_1=-(\nu-1) \in -\mathbb{N}\cup\{0\}$, the condition of Theorem 
\ref{thm2.8},(1) is violated. Then the solutions 
$\mathrm{trunc}^0_{\infty}(x,0,(-1)^{\nu-1}) =x(-\frac 12
(\nu-1+\theta_0+\theta_{\infty})+O(x^{-1})
-c^-x^{2\nu+\theta_{\infty}-1}
e^{-x}(1+O(x^{-r})))^{-1}$ constitute a one-parameter family, each of which 
corresponds to the same monodromy data $(M^0,M^1)=(M^0,(-1)^{\nu-1}I)$   
up to $\sim$, not depending on $c^-$ 
(cf.~Proposition \ref{prop3.1}, \cite[Remark 2.1 and Proposition 2.1]
{Andreev-Kitaev}). If $\theta_1=0$ there exists a one-parameter family
of classical solutions $-(\theta_0+\theta_{\infty})^{-1}(2xw'(x)w(x)^{-1}
-\theta_{\infty}-1+x)$ with $w(x)$ solving the Whittaker 
equation $w''-(\frac 14 -\frac 12(1+\theta_{\infty})x^{-1}+\frac 14(\theta_0^2
-1)x^{-2})w=0$, and then $M^1=I.$ From the remaining truncated solutions 
of Theorems \ref{thm2.6}, \ref{thm2.8} and \ref{thm2.9}
with integers $\theta_0$ or $\theta_1$ we get one-parameter families
having analogous properties.
\end{rem}
\subsection{Riemann-Hilbert correspondence to the monodromy manifold}
\label{ssc2.4}
The solutions exhibited above are labelled with the monodromy data by the
Riemann-Hilbert correspondence.
Recall the monodromy manifold 
\begin{align*}
\mathcal{M}(\theta_0,\theta_1,\theta_{\infty})
=& \{(M^0,M^1)\in \mathrm{SL}_2(\mathbb{C})^2 \,\,| 
\\
& m^0_{11}m^1_{11}+m^0_{21}m^1_{12}=e^{-\pi i\theta_{\infty}}, 
\,\, \mathrm{tr}M^0=2\cos \pi\theta_0,\,\, \mathrm{tr}M^1=2\cos \pi\theta_1 \}
/\sim
\end{align*}
with the gauge equivalence relation 
$c^{\sigma_3}(M^0,M^1)c^{-\sigma_3}\sim (M^0,M^1)$.
In accordance with \cite[Section 2]{Andreev-Kitaev}, adding some 
modifications, we express this by the disjoint union  
$$
\mathcal{M}(\theta_0,\theta_1,\theta_{\infty})=\bigcup_{j=1}^5
\mathcal{R}_j \quad \text{with 
$\mathcal{R}_2=\mathcal{R}_2^0 \cup \mathcal{R}_2^1
\cup \mathcal{R}_2^{01},$} 
$$
where the regions $\mathcal{R}_j$ and $\mathcal{R}_2^{\square}$ are given 
by the following:
\begin{align*}
& \text{Region} &  \quad &  \text{Proper coordinates}
\\
&\mathcal{R}_1:\,\,\, m^0_{11}\not=0,\,\, m^1_{11}\not=0,\,\, 
m^0_{21} m^1_{12} \not=0,
 &\quad & 
\{m^0_{11},\, m^0_{21}m^1_{12}\}\,\,\text{or} \,\,
\{m^1_{11},\, m^0_{21}m^1_{12}\},
\\
&\mathcal{R}_2^0:\,\,\,m^0_{11}\not=0,\,\, m^1_{11}=0,\,\, 
m^0_{21}  m^1_{12} \not=0, 
&\quad &
\{m^0_{11}\}\,\, \text{with $m^0_{21} m^1_{12}
= e^{-\pi i\theta_{\infty}}$},
\\
&\mathcal{R}_2^1:\,\,\,m^0_{11}=0,\,\, m^1_{11}\not=0,\,\, 
m^0_{21}  m^1_{12} \not=0, 
&\quad &
\{m^1_{11}\}\,\, \text{with $m^0_{21} m^1_{12}
= e^{-\pi i\theta_{\infty}}$},
\\
&\mathcal{R}_2^{01}:\,\,\,m^0_{11}= m^1_{11}=0,\,\, 
m^0_{21} m^1_{12} \not=0, 
&\quad &
\{m^0_{11}=m^1_{11}=0\}\,\,\text{with $m^0_{21} m^1_{12}
= e^{-\pi i\theta_{\infty}} $},
\\
&\mathcal{R}_3:\,\,\,  m^1_{11}\not=0,\,\, m^0_{21}\not=0,\,\,  m^1_{12} =0, 
&\quad &
\{m^1_{21}/m^0_{21} \} \,\, \text{with $m^1_{11}=e^{\pm \pi i\theta_1}$,} 
\\
&\mathcal{R}_4:\,\,\,  m^1_{11}\not=0,\,\, m^0_{21}=0,\,\,  m^1_{12}\not=0, 
&\quad &
\{m^0_{12}/ m^1_{12} \} \,\, \text{with $m^0_{11}=e^{\pm \pi i\theta_0}$,} 
\\
&\mathcal{R}_5:\,\,\,  m^1_{11}\not=0,\,\, m^0_{21}= m^1_{12}=0, 
&\quad &
\{m^0_{12} m^1_{21} \} \,\, \text{with $m^1_{11}=e^{\pm\pi i\theta_1}$.} 
\end{align*}
The region
$\mathcal{R}_3$ (respectively, $\mathcal{R}_4$) is written in 
the form
\begin{align*}
&\mathcal{R}_3=\mathcal{R}_{3+}\cup \mathcal{R}_{3-}, \quad
\mathcal{R}_{3\pm}=\{(M^0,M^1)\in \mathcal{R}_3 \,|\, m^1_{11}
=e^{\pm \pi i\theta_1}\}, 
\\
(\text{respectively,} \,\,\,
& \mathcal{R}_4=\mathcal{R}_{4+}\cup\mathcal{R}_{4-},  \quad
\mathcal{R}_{4\pm}=\{(M^0,M^1)\in \mathcal{R}_4 \,|\, m^0_{11}
=e^{\pm \pi i\theta_0}\} ),
\end{align*}
in which $\mathcal{R}_{3\pm}=\mathcal{R}_3$ (respectively, $\mathcal{R}_{4\pm}
=\mathcal{R}_4$) 
if $\theta_1 \in\mathbb{Z}$ (respectively, $\theta_0\in\mathbb{Z}$),
and $\mathcal{R}_{3+}\cap\mathcal{R}_{3-}=\emptyset$
(respectively, $\mathcal{R}_{4+}\cap\mathcal{R}_{4-}=\emptyset$) otherwise
\cite[Proposition 2.2]{Andreev-Kitaev}. 
\begin{rem}\label{2.8}
Our main concern is the Riemann-Hilbert correspondence of 
$(y(x),\mathfrak{z}(x))$ to the monodromy data 
in $\mathcal{M}(\theta_0,\theta_1,\theta_{\infty})$,
and by Remark \ref{rem2.4} it is sufficient to consider $y(x)$ only. 
\end{rem}
\begin{rem}\label{rem2.9}
To treat the solutions in the sector $|\arg x|<\pi/2$, the regions $\mathcal{R}
_1$ and $\mathcal{R}_2$ of \cite{Andreev-Kitaev} are subdivided, and are
denoted here by $\mathcal{R}_1\cup\mathcal{R}_2^{1}$ and $\mathcal{R}_2^0 \cup
\mathcal{R}_2^{01}$, respectively.
Since the monodromy data are considered up to the gauge equivalence,
the proper coordinate $\{m^0_{21}m^1_{12}\}$ is adopted in place of
$\{m^0_{21}, m^1_{12}\}$ of \cite{Andreev-Kitaev} if $m^0_{21} m^1_{12}\not=0.$ 
For the same reason so are the coordinates $\{m^1_{21}/m^0_{21}\}$ and
$\{m^0_{12}/m^1_{12}\}$ of $\mathcal{R}_3$ and $\mathcal{R}_4$, respectively. 
\end{rem}
In \cite{Andreev-Kitaev}, all the asymptotics of solutions  
on the positive real axis corresponding 
to $\mathcal{R}_1 \cup\mathcal{R}_2$ are completely characterised, and 
about $\bigcup_{j=3,4,5}\mathcal{R}_j$, truncated solutions of \cite{Andreev} 
are referred to. Let us begin with the solutions for $|\arg x|<\pi/2$ 
related to $\mathcal{R}_1$ and $\mathcal{R}_2$.
Define families of solutions (cf.~Theorems \ref{thm2.1} -- 
\ref{thm2.4} and Remark \ref{rem2.3}):  
\begin{align*}
 \mathrm{El}_{\,(-\frac{\pi}2,0)\cup(0,\frac{\pi}2)}
:&=\{\mathrm{ellip}(x,m^0_{21}m^1_{12},m^0_{11},
m^1_{11})\,|\,\, m^0_{21}m^1_{12}m^0_{11}m^1_{11}\not=0\},
\\
 \mathrm{El}_{\,(-\frac{\pi}2,0)}
:&=\{\mathrm{ellip}(x,m^0_{21}m^1_{12},m^0_{11},
0)\,|\,\, m^0_{21}m^1_{12}=e^{-\pi i\theta_{\infty}},\, m^0_{11}\not=0\},
\\
 \mathrm{El}_{\,(0,\frac{\pi}2)}
:&=\{\mathrm{ellip}(x,m^0_{21}m^1_{12},0, m^1_{11})
\,|\,\, m^0_{21}m^1_{12} =e^{-\pi i\theta_{\infty}},\, m^1_{11}\not=0\},
\\
 \mathrm{Trg}_{\,[0]}
:&=\{\mathrm{trig}(x,m^0_{21}m^1_{12},m^0_{11}, m^1_{11})
\,|\,\, m^0_{21}m^1_{12}m^0_{11}m^1_{11}\not=0\},
\\
 \mathrm{Tc}_{\,(-\pi, 0]}
:&=\{\mathrm{trunc}(x,0, m^1_{11})
\,|\,\, m^0_{21}m^1_{12}=e^{-\pi i\theta_{\infty}}, m^1_{11}\not=0\},
\\
 \mathrm{Tc}_{\,[0,\pi)}
:&=\{\mathrm{trunc}(x, m^0_{11},0)
\,|\,\, m^0_{21}m^1_{12}=e^{-\pi i\theta_{\infty}}, m^0_{11}\not=0\},
\\
 \mathrm{dTc}_{\,(-\pi,\pi)}
:&=\{\mathrm{trunc}(x,0, 0)
\,|\,\, m^0_{21}m^1_{12}=e^{-\pi i\theta_{\infty}}, m^0_{11}=m^1_{11}=0\}.
\end{align*}
Then we have the Riemann-Hilbert correspondence to 
$\mathcal{R}_1\cup\mathcal{R}_2$.
\begin{thm}\label{thm2.10}
There exist canonical bijections 
$\psi_{\square}$ described as follows, which map each solution to the proper 
coordinates of $\mathcal{R}_1$ and $\mathcal{R}_2$. 
\par
$(1)$ Along the positive real axis $\arg x=0$,
\begin{align*}
&\psi_0 : \mathrm{Trg}_{\,[0]} \,\, \to \mathcal{R}_1, \quad
\\
&\psi_0 : \mathrm{Tc}_{\,(-\pi, 0]}  \to \mathcal{R}^1_2, \quad
\psi_0 : \mathrm{Tc}_{\,[0,\pi)}  \to \mathcal{R}^0_2, \quad
\psi_0 : \mathrm{dTc}_{\,(-\pi,\pi)}  \to \mathcal{R}^{01}_2. 
\end{align*}
\par
$(2)$ In the sector $-\pi/2<\arg x<0$,
\begin{align*}
&\psi_- : \mathrm{El}_{\,(-\frac{\pi}2,0)\cup(0,\frac{\pi}2)} \to \mathcal{R}_1,
\\
&\psi_- : \mathrm{El}_{\,(-\frac{\pi}2,0)} \to \mathcal{R}^0_2,
\quad
\psi_- : \mathrm{Tc}_{\,(-\pi, 0]}  \to \mathcal{R}^1_2, \quad
\quad
\psi_- : \mathrm{dTc}_{\,(-\pi,\pi)}  \to \mathcal{R}^{01}_2. 
\end{align*}
\par
$(3)$ In the sector $0<\arg x<\pi/2,$
\begin{align*}
&\psi_+ : \mathrm{El}_{\,(-\frac{\pi}2,0)\cup(0,\frac{\pi}2)} \to \mathcal{R}_1,
\\
&\psi_+ : \mathrm{El}_{\,(0,\frac{\pi}2)} \to \mathcal{R}^1_2,
\quad
\psi_+ : \mathrm{Tc}_{\,[0,\pi)}  \to \mathcal{R}^0_2, \quad
\psi_+ : \mathrm{dTc}_{\,(-\pi,\pi)}  \to \mathcal{R}^{01}_2. 
\end{align*}
\end{thm}
To observe the correspondence of solutions to $\mathcal{R}_3$ and 
$\mathcal{R}_4$ let us define the families of truncated solutions of 
Theorems \ref{thm2.6} and \ref{thm2.8}:
\begin{align*}
 \mathrm{Tc}^0_{0\,(-\frac{\pi}2, \frac{\pi}2]}
:&=\{\mathrm{trunc}^0_0(x,m^1_{21}/m^0_{21},e^{\pi i\theta_1} )
\,|\,\, m^1_{21}/m^0_{21}\not=0, m^1_{11}=e^{\pi i\theta_1}\},
\\
 \mathrm{dTc}^0_{0\,(-\frac{\pi}2,\frac{3\pi}2) }
:&=\{\mathrm{trunc}^0_0(x,0,e^{\pi i\theta_1} )
\,|\,\, m^1_{21}=0, m^1_{11}=e^{\pi i\theta_1}\},
\\
 \mathrm{Tc}^1_{0\,[-\frac{\pi}2 , \frac{\pi}2)}
:&=\{\mathrm{trunc}^1_0(x,m^0_{12}/m^1_{12},e^{-\pi i\theta_0} )
\,|\,\, m^0_{12}/m^1_{12}\not=0, m^0_{11}=e^{-\pi i\theta_0}\},
\\
 \mathrm{dTc}^1_{0\,(-\frac{3\pi}2,\frac{\pi}2)}
:&=\{\mathrm{trunc}^1_0(x,0,e^{-\pi i\theta_0} )
\,|\,\, m^0_{12}=0, m^0_{11}=e^{-\pi i\theta_0}\},
\\
 \mathrm{Tc}^0_{\infty\,(-\frac{\pi}2, \frac{\pi}2]}
:&=\{\mathrm{trunc}^0_{\infty}(x,m^1_{21}/m^0_{21},e^{-\pi i\theta_1} )
\,|\,\, m^1_{21}/m^0_{21}\not=0, m^1_{11}=e^{-\pi i\theta_1}\},
\\
 \mathrm{dTc}^0_{\infty\,(-\frac{\pi}2, \frac{3\pi}2)}
:&=\{\mathrm{trunc}^0_{\infty}(x,0,e^{-\pi i\theta_1} )
\,|\,\, m^1_{21}=0, m^1_{11}=e^{-\pi i\theta_1}\},
\\
 \mathrm{Tc}^1_{\infty\,[-\frac{\pi}2 , \frac{\pi}2)}
:&=\{\mathrm{trunc}^1_{\infty}(x,m^0_{12}/m^1_{12},e^{\pi i\theta_0} )
\,|\,\, m^0_{12}/m^1_{12}\not=0, m^0_{11}=e^{\pi i\theta_0}\},
\\
 \mathrm{dTc}^1_{\infty\,(-\frac{3\pi}2,\frac{\pi}2)}
:&=\{\mathrm{trunc}^1_{\infty}(x,0,e^{\pi i\theta_0} )
\,|\,\, m^0_{12}=0, m^0_{11}=e^{\pi i\theta_0} \}.
\end{align*}
Then the correspondence of these solutions to $\mathcal{R}_3\cup\mathcal{R}_4$
is given in the theorem below, in which $(\Theta_j)$'s denote conditions as
follows: 
\par
$(\Theta_1)$ $ \theta_0-\theta_1-\theta_{\infty} \not\in 2\mathbb{N}\cup\{0\}$
and $\theta_0+\theta_1+\theta_{\infty} \not\in -2\mathbb{N}\cup\{0\},$
\par 
$(\Theta_2)$ $ \theta_0-\theta_1-\theta_{\infty}\not \in -2\mathbb{N}\cup\{0\}$
and $ \theta_0+\theta_1-\theta_{\infty} \not \in -2\mathbb{N}\cup\{0\},$
\par
$(\Theta_3)$ $ \theta_0+\theta_1-\theta_{\infty} \not\in 2\mathbb{N}$
and $\theta_0-\theta_1+\theta_{\infty} \not\in -2\mathbb{N}\cup\{0\},$
\par 
$(\Theta_4)$ $ \theta_0+\theta_1+\theta_{\infty}\not \in 2\mathbb{N}$
and $ \theta_0-\theta_1+\theta_{\infty}  \not\in 2\mathbb{N}\cup\{0\}.$
\par\noindent
\begin{thm}\label{thm2.11}
There exist canonical bijections $\psi_{\square}^{\square}$ given as follows,
which map each solution in the sector $|\arg x|<\pi/2$ to the proper coordinates 
of $\mathcal{R}_3$ or $\mathcal{R}_4$.
\par
$(1)$ Under $(\Theta_1)$ with $\theta_1\not \in \mathbb{N}\cup\{0\}$,
\,\, $\psi_{+}^0: \mathrm{Tc}^0_{0\,(-\frac{\pi}2, \frac{\pi}2]}
\cup \mathrm{dTc}^0_{0\,(-\frac{\pi}2,\frac{3\pi}2) } \to \mathcal{R}_{3+},$
\par
$(2)$ Under $(\Theta_2)$ with $\theta_0\not \in \mathbb{N}\cup\{0\}$,
\,\, $\psi_{-}^0: \mathrm{Tc}^1_{0\,[-\frac{\pi}2 , \frac{\pi}2)}
 \cup\mathrm{dTc}^1_{0\,(-\frac{3\pi}2,\frac{\pi}2)} \to \mathcal{R}_{4-},$
\par
$(3)$ Under $(\Theta_3)$ with $\theta_1\not \in -\mathbb{N} \cup\{0\}$,
\,\, $\psi_{+}^{\infty}: \mathrm{Tc}^0_{\infty\,(-\frac{\pi}2, \frac{\pi}2]}
\cup\mathrm{dTc}^0_{\infty\,(-\frac{\pi}2, \frac{3\pi}2)} \to \mathcal{R}_{3-},$
\par
$(4)$ Under $(\Theta_4)$ with $\theta_0\not\in -\mathbb{N}\cup\{0\}$,
\,\, $\psi_{-}^{\infty}:\mathrm{Tc}^1_{\infty\,[-\frac{\pi}2 , \frac{\pi}2)}
\cup \mathrm{dTc}^1_{\infty\,(-\frac{3\pi}2,\frac{\pi}2)}\to\mathcal{R}_{4+}.$
\end{thm}
\begin{rem}\label{rem2.10}
In Theorem \ref{thm2.11}, if the condition on $\theta_0$ or $\theta_1$ is 
violated, the map $\psi_{\pm}^0$ or $\psi_{\pm}^{\infty}$ is neither 
one-to-one nor surjective. Say, in (1), if $\theta_1 \in \mathbb{N}$, 
then every $y(x) \in
\mathcal{T}_1:= \mathrm{Tc}^0_{0\,(-\frac{\pi}2, \frac{\pi}2]}
\cup \mathrm{dTc}^0_{0\,(-\frac{\pi}2,\frac{3\pi}2) }$ corresponds to
the monodromy data $(M^0,(-1)^{\theta_1}I)$ with $M_0$ as in Theorem \ref
{thm2.6},(1), and the image of $\mathcal{T}_1$ shrinks to the one point 
$\psi_{+}^0(\mathcal{T}_1)=\{m^1_{21}/m^0_{21}=0\} \in \mathcal{R}_{3+}$.  
\end{rem}
Note that $\mathcal{R}_5\not= \emptyset$ if and only if 
$\epsilon_0\theta_0 +\epsilon_1\theta_1 +\theta_{\infty} \in 2\mathbb{Z}$
with $\epsilon_{0}, \epsilon_1\in \{-1,1\}$ 
\cite[Proposition 2.2]{Andreev-Kitaev}. 
We get a bijection to the monodromy manifold.
\begin{thm}\label{thm2.12}
Let $\mathcal{V}_{\phi}$ be the family of solutions of 
$(\mathrm{P}_{\mathrm{V}})$
in a strip containing the direction $\arg x=\phi$. 
Under the conditions $(\Theta_1),$ $(\Theta_2),$ $(\Theta_3),$ $(\Theta_4)$
with $\theta_0,\theta_1\not\in \mathbb{Z},$ for each direction 
$\arg x=\phi$ such that $|\phi|<\pi/2$, there exists the bijection 
$$
\psi_{\phi}:  \mathcal{V}_{\phi} 
\to \mathcal{M}(\theta_0,\theta_1,\theta_{\infty})
=\bigcup_{j=1}^4 \mathcal{R}_j, 
$$
that is, for each $(M^0,M^1)\in \mathcal{M}(\theta_0,\theta_1,\theta_{\infty})$ 
an explicit pair $(\psi_{\phi},y(x))$ with $y(x) \in\mathcal{V}_{\phi}$ 
satisfying $\psi_{\phi}:y(x)\mapsto (M^0,M^1)$ may be found uniquely in 
Theorem $\ref{thm2.10}$ or $\ref{thm2.11}$. 
\end{thm}
\begin{rem}\label{rem2.11}
In the theorem above, suppose that one condition alone is violated,
say, $\theta_0+\theta_1+\theta_{\infty} \in
-2\mathbb{N}\cup\{0\}$ and $\theta_0,$ $\theta_1\not\in \mathbb{Z}$. 
This situation affects Theorem \ref{thm2.6},(1) only, in which entries of
$(M^0,M^1)$ become $m^0_{11}=e^{-\pi i(\theta_1+\theta_{\infty})}=e^{\pi i
\theta_0},$ $m^1_{11}=e^{\pi i\theta_1}$ and $m^0_{21}=0,$ implying that
such $(M^0,M^1)$ does not fall on $\mathcal{R}_{3+}$, and $\mathcal{R}_{3+}
=\emptyset$ by Theorem \ref{thm2.11}.
In this case, by Proposition \ref{prop3.2},
we have $m^0_{12}\not=0$ since $m^1_{12}=0,$ and Theorem \ref{thm2.9},(1)
provides solutions mapped to $\mathcal{R}_5.$ 
Thus we have the bijection $\mathcal{V}_{\phi}\to
\mathcal{M}(\theta_0,\theta_1,\theta_{\infty})
=\mathcal{R}_1\cup\mathcal{R}_2\cup \mathcal{R}_{3-}\cup\mathcal{R}_4
\cup\mathcal{R}_5.$  
\end{rem}
In the case where the condition $(\Theta_j)$ $(j=2,3,4)$ alone is violated 
as well, using Proposition \ref{prop3.2} and Remark \ref{rem3.1} 
by the same argument as in Remark \ref{rem2.11}, we may verify the existence
of an analogous bijection to $\mathcal{M}(\theta_0,\theta_1,\theta_{\infty})$.  
\begin{thm}\label{thm2.13} 
Suppose that $\theta_0,$ $\theta_1 \not\in \mathbb{Z}.$ In the case where
$(\Theta_j)$ alone is violated, we have the following:
\par
$(1)$ $\mathcal{R}_{3+}=\emptyset$ if 
$\theta_0+\theta_1+\theta_{\infty}\in -2\mathbb{N}\cup\{0\}$ or 
$\theta_0-\theta_1-\theta_{\infty}\in 2\mathbb{N}\cup\{0\};$
\par
$(2)$ $\mathcal{R}_{3-}=\emptyset$ if 
$\theta_0-\theta_1+\theta_{\infty}\in -2\mathbb{N}\cup\{0\}$ or 
$\theta_0+\theta_1-\theta_{\infty}\in 2\mathbb{N};$
\par
$(3)$ $\mathcal{R}_{4-}=\emptyset$ 
if $\theta_0-\theta_1-\theta_{\infty}\in -2\mathbb{N}\cup\{0\}$ or 
$\theta_0+\theta_1-\theta_{\infty}\in -2\mathbb{N}\cup\{0\};$
\par
$(4)$ $\mathcal{R}_{4+}=\emptyset$ 
if $\theta_0+\theta_1+\theta_{\infty}\in 2\mathbb{N}$ or 
$\theta_0-\theta_1+\theta_{\infty}\in 2\mathbb{N}\cup\{0\}.$
\par\noindent
In each case $\mathcal{M}(\theta_0,\theta_1,\theta_{\infty})$ contains 
$\mathcal{R}_5$ in place of the empty region, and the corresponding 
solutions in Theorem $\ref{thm2.9}$ are bijectively mapped to $\mathcal{R}_5$.
\end{thm}
\subsection{Nonlinear monodromy-Stokes structure and the character variety}
\label{ssc2.5}
The solution $y(x,M^0,M^1)$ defined in Subsection \ref{ssc2.1} is analytic on the
universal covering of $\mathbb{C}\setminus \{0\}$.  
To examine the analytic continuation of $y(x,M^0,M^1)$ it is necessary 
to know nonlinear monodromy and Stokes structure.
The nonlinear monodromy of $y(x,M^0,M^1)$ is represented as follows:
\begin{thm}\label{thm2.14}
For every $p\in \mathbb{Z}$ we have
$$
y(e^{2\pi i p}x,M^0,M^1)=y(x,M_{(p)}^0,M_{(p)}^1),
$$
where $M^{0,\,1}_{(p)}:= (M^1M^0)^pM^{0,\,1} (M^1M^0)^{-p}.$
\end{thm}
Let $(\mathrm{P}_{\mathrm{V}})_{(\theta_0, \theta_1, -\theta_{\infty})}$ be
the Painlev\'e equation $(\mathrm{P}_{\mathrm{V}})$ with the parameters 
$(\theta_0, \theta_1,-\theta_{\infty})$, and 
$\mathcal{M}(\theta_0,\theta_1,-\theta_{\infty})$    
the monodromy manifold consisting of monodromy data labelling the solutions
of $(\mathrm{P}_{\mathrm{V}})_{(\theta_0, \theta_1, -\theta_{\infty})}$. 
Let 
$y_{-\theta_{\infty}}(x,M^{0}_-,M^{1}_-)$ denote the solution of
$(\mathrm{P}_{\mathrm{V}})_{(\theta_0, \theta_1, -\theta_{\infty})}$ 
labelled with monodromy data 
$(M^{0}_-,M^{1}_-)\in \mathcal{M}(\theta_0,\theta_1,-\theta_{\infty})$.    
The nonlinear Stokes relation for $y(x,M^0,M^1)$ with $(M^0,M^1)\in 
\mathcal{M}(\theta_0,\theta_1,\theta_{\infty})$ is given as follows, 
which is based on the transformation of \cite[(2.9)]{Andreev-Kitaev-2}, 
\cite[Section 9.3]{Andreev-Kitaev-2019}.   
\begin{thm}\label{thm2.15}
Write $[M]_{\sigma_1}:=\sigma_1 M \sigma_1.$
We have
\begin{align*}
&y(e^{\pi i}x,M^0,M^1)=
1/y_{-\theta_{\infty}}(x,[\hat{M}^0]_{\sigma_1},[\hat{M}^1]_{\sigma_1}),
\\
&y(e^{-\pi i}x,M^0,M^1)=
1/y_{-\theta_{\infty}}(x,[\check{M}^0]_{\sigma_1},[\check{M}^1]_{\sigma_1}),
\phantom{----}
\end{align*}
where 
\begin{align*}
&\hat{M}^{0,1}:=e^{-\frac{\pi i}2\theta_{\infty}\sigma_3} S^{-1}_2 M^{0,1}
S_2 e^{\frac{\pi i}2\theta_{\infty}\sigma_3},
\\
&\check{M}^{0,1}:=e^{\frac{\pi i}2\theta_{\infty}\sigma_3} S_1 M^{0,1}
S_1^{-1} e^{-\frac{\pi i}2\theta_{\infty}\sigma_3}.
\phantom{----}
\end{align*}
\end{thm}
\begin{rem}\label{rem2.110}
The matrices $[\hat{M}^{0,\,1}]_{\sigma_1}$ and $[\check{M}^{0,\,1}]_{\sigma_1}$
coincide with \cite[(2.11), (2.12)]{Andreev-Kitaev-2} up to gauge multipliers. 
Since 
$[\hat{M}^{1}]_{\sigma_1}[\hat{M}^{0}]_{\sigma_1}
=[\check{M}^{1}]_{\sigma_1}[\check{M}^{0}]_{\sigma_1}
= e^{\frac{\pi i}2 
\theta_{\infty}\sigma_3}[S_2^{-1}S_1^{-1}]_{\sigma_1}e^{\frac{\pi i}2
\theta_{\infty}\sigma_3}$, we may consider
$[\hat{M}^{0,\,1}]_{\sigma_1},$ $[\check{M}^{0,\,1}]_{\sigma_1} \in \mathcal{M}
(\theta_0,\theta_1,-\theta_{\infty}).$
\end{rem}
Proofs of Theorems \ref{thm2.14} and \ref{thm2.15} are given in Section 
\ref{sc4}. 
\par
The monodromy and Stokes relations above may be regarded as operators 
acting on the monodromy data assigned to each solution. From such a view point
these operators are formulated as follows.
\par
For fixed $(M^0,M^1)\in \mathcal{M}(\theta_0,\theta_1,\theta_{\infty})$ 
define families by
\begin{align*}
\mathfrak{M}(M^0,M^1):&=\{(M^0_{(p)},M^1_{(p)})\,|\,\, M^{0,1}_{(p)}
=(M^1M^0)^p M^{0,1}(M^1M^0)^{-p},\,\, p\in \mathbb{Z} \},
\\
\hat{\mathfrak{M}}(M^0,M^1):&
=\{([\hat{M}^0_{(p)}]_{\sigma_1},[\hat{M}^1_{(p)}]_{\sigma_1})\,|\,\,
 \hat{M}^{0,1}_{(p)}
=e^{-\frac{\pi i}2\theta_{\infty}\sigma_3}
 S^{-1}_2 M_{(p)}^{0,1} S_2 e^{\frac{\pi i}2\theta_{\infty}\sigma_3}
 ,\,\, p\in \mathbb{Z} \}
\\
&=\sigma_1 e^{-\frac{\pi i}2\theta_{\infty}\sigma_3}
S^{-1}_2\mathfrak{M}(M^0,M^1)S_2
e^{\frac{\pi i}2\theta_{\infty}\sigma_3}\sigma_1 
\end{align*}
such that $\mathfrak{M}(M^0,M^1)\subset \mathcal{M}(\theta_0,\theta_1,
\theta_{\infty})$ and 
$\hat{\mathfrak{M}}(M^0,M^1)\subset 
\mathcal{M}(\theta_0,\theta_1,-\theta_{\infty})$, the last inclusion being
due to 
$[\hat{M}^1]_{\sigma_1}[\hat{M}^0]_{\sigma_1}= e^{\frac{\pi i}2\theta_{\infty}
\sigma_3}[S_2^{-1}S_1^{-1}]_{\sigma_1}e^{\frac{\pi i}2\theta_{\infty}\sigma_3}$. 
\begin{defi}\label{defi2.1}
Define the monodromy operator $\mathfrak{m}$ and $\hat{\mathfrak{m}}$ by
\begin{align*}
\mathfrak{m} &:\,\, \mathfrak{M}(M^0,M^1)
\to \mathfrak{M}(M^0,M^1),\quad 
 \mathfrak{m}(M^0_{(p)},M^1_{(p)})=(M^0_{(p+1)},M^1_{(p+1)}),
\\
\hat{\mathfrak{m}} &:\,\, \hat{\mathfrak{M}}(M^0,M^1)
\to \hat{\mathfrak{M}}(M^0,M^1),\quad 
 \hat{\mathfrak{m}}([\hat{M}^0_{(p)}]_{\sigma_1},[\hat{M}^1_{(p)}]_{\sigma_1})
=([\hat{M}^0_{(p+1)}]_{\sigma_1},[\hat{M}^1_{(p+1)}]_{\sigma_1}),
\end{align*}
and the Stokes operators  $\mathfrak{s}_{0}$, $\mathfrak{s}_{-1}$ and
$\hat{\mathfrak{s}}_{0}$, $\hat{\mathfrak{s}}_{1}$ by
\begin{align*}
\mathfrak{s}_{0,\, -1} &:\,\, \mathfrak{M}(M^0,M^1) \to 
\hat{\mathfrak{M}}(M^0,M^1),
\\
& \mathfrak{s}_0(M^0_{(p)},M^1_{(p)})=([\hat{M}^0_{(p)}]_{\sigma_1},
[\hat{M}^1_{(p)}]_{\sigma_1}), \quad
\quad  \mathfrak{s}_{-1}(M^0_{(p)},M^1_{(p)})=([\hat{M}^0_{(p-1)}]_{\sigma_1},
[\hat{M}^1_{(p-1)}]_{\sigma_1}), 
\\
\hat{\mathfrak{s}}_{0,\,1} &:\,\, \hat{\mathfrak{M}}(M^0,M^1) \to 
\mathfrak{M}(M^0,M^1),
\\
&\hat{\mathfrak{s}}_0([\hat{M}^0_{(p)}]_{\sigma_1},
[\hat{M}^1_{(p)}]_{\sigma_1})=(M^0_{(p)},M^1_{(p)}), \quad
\quad  \hat{\mathfrak{s}}_1([\hat{M}^0_{(p)}]_{\sigma_1},
[\hat{M}^1_{(p)}]_{\sigma_1})=(M^0_{(p+1)},M^1_{(p+1)}). 
\end{align*}
\end{defi}  
\begin{rem}\label{rem2.12}
The operations $\mathfrak{s}_{0,\,-1}$ and $\hat{\mathfrak{s}}_{0,\,1}$ are also
written as
\begin{align*}
&\mathfrak{s}_0(M)=\sigma_1 e^{-\frac{\pi i}2\theta_{\infty}\sigma_3}S_2^{-1} M
S_2 e^{\frac{\pi i}2\theta_{\infty}\sigma_3}\sigma_1,
&\quad & \mathfrak{s}_{-1}(M)= \sigma_1e^{\frac{\pi i}2\theta_{\infty}\sigma_3}
S_1 M S_1^{-1}e^{-\frac{\pi i}2\theta_{\infty}\sigma_3}\sigma_1, 
\\
& \hat{\mathfrak{s}}_0([\hat{M}]_{\sigma_1})=
  S_2 e^{\frac{\pi i}2\theta_{\infty}\sigma_3}\hat{M}
   e^{-\frac{\pi i}2\theta_{\infty}\sigma_3}S_2^{-1},
&\quad 
&\hat{\mathfrak{s}}_1([\hat{M}]_{\sigma_1})= 
S_1^{-1}e^{-\frac{\pi i}2\theta_{\infty}\sigma_3} 
 \hat{M} e^{\frac{\pi i}2\theta_{\infty}\sigma_3}S_1 
\end{align*}
for each entry $M$ (respectively, $[\hat{M}]_{\sigma_1}$) of pairs in 
$\mathfrak{M}(M^0,M^1)$ (respectively, $\hat{\mathfrak{M}}(M^0,M^1)$). 
\end{rem}
\begin{prop}\label{prop2.16}
We have \,$\hat{\mathfrak{s}}_1 \circ \mathfrak{s}_0 =\mathfrak{m},$ \, 
$\hat{\mathfrak{s}}_0 \circ \mathfrak{s}_0 =\mathrm{id},$ \,
$\mathfrak{s}_0 \circ \hat{\mathfrak{s}}_1 =\hat{\mathfrak{m}}$ \, and \, 
$\mathfrak{s}_{-1} \circ \hat{\mathfrak{s}}_1 =\mathrm{id}.$ 
\end{prop}
For a solution $y(x, M^0,M^1)$ with $(M^0,M^1) \in 
\mathcal{M}(\theta_0,\theta_1,\theta_{\infty})$
the nonlinear monodromy-Stokes structure is given by 
actions of the operators 
on the families of monodromy data as above,
symbolically denoted by
$\langle \mathfrak{m},\,\, \mathfrak{s}_0,\,\,  \mathfrak{s}_{-1},\,\, 
\hat{\mathfrak{s}}_0,\,\,\hat{\mathfrak{s}}_1;\,\,
\mathfrak{M}(M^0,M^1),\,\, \hat{\mathfrak{M}}(M^0,M^1) \rangle$. 
\par From Theorems \ref{thm2.14} and \ref{thm2.15} we drive the following.
\begin{thm}\label{thm2.17}
Let $p\in \mathbb{Z}.$
For $(M^0_{(p)},M^1_{(p)})\in \mathfrak{M}(M^0,M^1)$
and $([\hat{M}^0_{(p)}]_{\sigma_1},[\hat{M}^1_{(p)}]_{\sigma_1})
\in \hat{\mathfrak{M}}(M^0,M^1)$, we have
\begin{align*}
&y(e^{2\pi i}x, M^0_{(p)},M^1_{(p)})= y(x, \mathfrak{m}(M^0_{(p)},M^1_{(p)}) ),
\\
&y(e^{\pi i}x, M^0_{(p)},M^1_{(p)})=1/ y_{-\theta_{\infty}}(x, \mathfrak{s}_0
(M^0_{(p)},M^1_{(p)})),
\\
&y(e^{-\pi i}x, M^0_{(p)},M^1_{(p)})=1/ y_{-\theta_{\infty}}(x,\mathfrak{s}_{-1}
(M^0_{(p)},M^1_{(p)})),
\\
&y_{-\theta_{\infty}}(e^{2\pi i}x, [\hat{M}^0_{(p)}]_{\sigma_1},
[\hat{M}^1_{(p)}]_{\sigma_1})= y_{-\theta_{\infty}}(x, \hat{\mathfrak{m}}
([\hat{M}^0_{(p)}]_{\sigma_1},[\hat{M}^1_{(p)}]_{\sigma_1}) ),
\\
&y_{-\theta_{\infty}}(e^{\pi i}x , [\hat{M}^0_{(p)}]_{\sigma_1}, 
[\hat{M}^1_{(p)}]_{\sigma_1})
=1/y(x,\hat{\mathfrak{s}}_1([\hat{M}^0_{(p)}]_{\sigma_1},
[\hat{M}^1_{(p)}]_{\sigma_1})),
\\
&y_{-\theta_{\infty}}(e^{-\pi i}x , [\hat{M}^0_{(p)}]_{\sigma_1}, 
[\hat{M}^1_{(p)}]_{\sigma_1})
=1/y(x,\hat{\mathfrak{s}}_0([\hat{M}^0_{(p)}]_{\sigma_1},
[\hat{M}^1_{(p)}]_{\sigma_1})).
\end{align*}
\end{thm}
For $(M^0,M^1)\in \mathcal{M}(\theta_0,\theta_1,\theta
_{\infty})$, the points $(x_0,x_1,x_2)
=(m^0_{11},m^1_{11},\mathrm{tr}(M^1M^0))$ constitute the character variety 
represented by the Fricke relation \cite{Put}
\begin{equation*}
\mathbf{V}_{(\theta_0,\theta_1,\theta_{\infty})}: \quad  x_0x_1x_2 +x_0^2+x_1^2
             -\mu_0(\theta_{\infty})x_0 -\mu_1(\theta_{\infty})x_1
-e^{-\pi i\theta_{\infty}}x_2+ \kappa(\theta_{\infty})=0,
\end{equation*}
where $\mu_0(\theta_{\infty})=\mathrm{tr}\,M^0+e^{-\pi i\theta_{\infty}}
\mathrm{tr}\,M^1,$ $\mu_1(\theta_{\infty})=\mathrm{tr}\,M^1+e^{-\pi i\theta_
{\infty}}\mathrm{tr}\,M^0$ and $\kappa(\theta_{\infty})=
e^{-2\pi i\theta_{\infty}}+e^{-\pi i\theta_{\infty}}
\mathrm{tr}\,M^0\,\mathrm{tr}\,M^1+1.$
Indeed, by \eqref{2.5}, 
\begin{align*}
x_2=&\mathrm{tr}(M^1M^0)=m^1_{11}m^0_{11}+m^1_{12}m^0_{21}+m^1_{21}m^0_{12}
+m^1_{22}m^0_{22}
\\
=&e^{-\pi i\theta_{\infty}}+\frac{(x_1(\mathrm{tr}\,M^1-x_1)-1)
(x_0(\mathrm{tr}\,M^0-x_0)-1)}{e^{-\pi i\theta_{\infty}}-x_0x_1}
+(\mathrm{tr}\,M^1-x_1)(\mathrm{tr}\,M^0-x_0),
\end{align*}
from which the equation of $\mathbf{V}_{(\theta_0,\theta_1,\theta_{\infty})}$ 
immediately follows.
\begin{thm}\label{thm2.17a}
For $(M^0,M^1)\in \mathcal{M}(\theta_0,\theta_1,\theta_{\infty})$, let
$x_0,$ $x_1$, $x_2$ be as above. 
\par
$(1)$ For $(\hat{M}^0,\hat{M}^1)=e^{-\frac{\pi i}2\theta_{\infty}\sigma_3}
S_2^{-1}(M^0,M^1)S_2e^{\frac{\pi i}2\theta_{\infty}\sigma_3}$ with
$\hat{M}^{0,\,1}=(\hat{m}^{0,\,1}_{ij})$, write
\begin{equation*}
\hat{x}_0=\hat{m}^0_{22}, \quad \hat{x}_1=\hat{m}^1_{22},\quad
\hat{x}_2=\mathrm{tr}(\hat{M}^1\hat{M}^0).
\end{equation*}
Then $(\hat{x}_0,\hat{x}_1,\hat{x}_2) \in \mathbf{V}_{(\theta_0,\theta_1,
-\theta_{\infty})}$, 
$(e^{-\pi i\theta_{\infty}}\hat{x}_1, e^{-\pi i\theta_{\infty}}\hat{x}_0,
\hat{x}_2) \in \mathbf{V}_{(\theta_0,\theta_1,\theta_{\infty})}$, and
$$
\hat{x}_0=e^{\pi i\theta_{\infty}}x_1, \quad
\hat{x}_1=\mathrm{tr}\,M^1-e^{\pi i\theta_{\infty}}(x_1x_2+x_0 
-\mathrm{tr}\,M^0), \quad \hat{x}_2=x_2.
$$
\par
$(2)$ For $(\check{M}^0,\check{M}^1)=e^{\frac{\pi i}2\theta_{\infty}\sigma_3}
S_1(M^0,M^1)S_1^{-1}e^{-\frac{\pi i}2\theta_{\infty}\sigma_3}$ with
$\check{M}^{0,\,1}=(\check{m}^{0,\,1}_{ij})$, write
\begin{equation*}
\check{x}_0=\check{m}^0_{22}, \quad \check{x}_1=\check{m}^1_{22},\quad
\check{x}_2=\mathrm{tr}(\check{M}^1\check{M}^0).
\end{equation*}
Then $(\check{x}_0,\check{x}_1,\check{x}_2) \in \mathbf{V}_{(\theta_0,
\theta_1,-\theta_{\infty})}$, 
$(e^{-\pi i\theta_{\infty}}\check{x}_1, e^{-\pi i\theta_{\infty}}\check{x}_0,
\check{x}_2) \in \mathbf{V}_{(\theta_0,\theta_1,\theta_{\infty})}$, and 
$$
\check{x}_0=\mathrm{tr}\,M^0-e^{\pi i\theta_{\infty}}(x_0x_2+x_1 
-\mathrm{tr}\,M^1), \quad
\check{x}_1=e^{\pi i\theta_{\infty}}x_0, \quad \check{x}_2=x_2.
$$
\end{thm}
\begin{proof}
For $(M^0,M^1) \in \mathcal{M}(\theta_0,\theta_1,\theta_{\infty})$, consider
$([\hat{M}^0]_{\sigma_1}, [\hat{M}^1]_{\sigma_1})$ and write
$[\hat{M}^{0,\,1}]_{\sigma_1}=([\hat{m}^{0,\,1}]_{ij}).$ Note that
$\hat{M}^1\hat{M}^0=e^{-\frac{\pi i}2\theta_{\infty}\sigma_3}S_2^{-1}S_1^{-1}
e^{-\frac{\pi i}2\theta_{\infty}\sigma_3}$ and that
$$
([\hat{M}^1]_{\sigma_1} [\hat{M}^0]_{\sigma_1})_{11}=
[\hat{m}^1]_{11}[\hat{m}^0]_{11}+[\hat{m}^1]_{12}[\hat{m}^0]_{21}=e^{\pi i
\theta_{\infty}}.
$$
Setting $\hat{x}_0=[\hat{m}^0]_{11}=\hat{m}^0_{22},$ 
$\hat{x}_1=[\hat{m}^1]_{11}=\hat{m}^1_{22},$ and $\hat{x}_2=\mathrm{tr}([\hat{M}
^1]_{\sigma_1}[\hat{M}^0]_{\sigma_1}) =\mathrm{tr}(\hat{M}^1\hat{M}^0)$,
and calculating the Fricke relation with respect to  
$([\hat{M}^0]_{\sigma_1}, [\hat{M}^1]_{\sigma_1})$ or $(\hat{M}^0, \hat{M}^1)$,
we have $(\hat{x}_0,\hat{x}_1,\hat{x}_2)\in \mathbf{V}_{(\theta_0,
\theta_1,-\theta_{\infty})},$
from which $(e^{-\pi i\theta_{\infty}}\hat{x}_1,e^{-\pi i\theta_{\infty}}
\hat{x}_0,\hat{x}_2)\in \mathbf{V}_{(\theta_0,\theta_1,\theta_{\infty})}$ 
immediately follows.
The relations between $(x_0,x_1,x_2)$ and $(\hat{x}_0,\hat{x}_1,\hat{x}_2)$
are obtained by using \eqref{2.5}.
\end{proof}
Recall the operators $\mathfrak{s}_0$ and $\hat{\mathfrak{s}}_1 
=(\mathfrak{s}_{-1})^{-1}$ above.
By Proposition \ref{prop2.16},
$\hat{\mathfrak{s}}_1\circ \mathfrak{s}_0(M^0,M^1)= \mathfrak{m}(M^0,M^1)
=(M^0_{(1)},M^1_{(1)}),$ that is, 
$\mathfrak{s}_{-1}(M^0_{(1)},M^1_{(1)})=([\hat{M}^0]_{\sigma_1},
[\hat{M}^1]_{\sigma_1})$. The last relation is equivalent to 
$ e^{\frac{\pi i}2\theta_{\infty}\sigma_3}
S_1(M_{(1)}^0,M_{(1)}^1)S_1^{-1}e^{-\frac{\pi i}2\theta_{\infty}\sigma_3}
=(\hat{M}^0,\hat{M}^1)=(\check{M}^0_{(1)},\check{M}^1_{(1)})$, which is
the condition of Theorem \ref{thm2.17a},(2) with $(M^0,M^1)\mapsto
(M^0_{(1)},M^1_{(1)})$. Thus we have the following corollary, in which
$(x^{(1)}_0,x^{(1)}_1,x^{(1)}_2)$ is obtained from
\begin{align*}
&x^{(1)}_0=e^{-\pi i\theta_{\infty}}\check{x}_1,\quad
\mathrm{tr}\,M^1-x^{(1)}_1=e^{-\pi i\theta_{\infty}}(\check{x}_2\check{x}_1
+\check{x}_0-\mathrm{tr}\,M^0),\quad x^{(1)}_2=\check{x}_2
\\
&\check{x}_0=\hat{x}_0=e^{\pi i\theta_{\infty}}x_1, \quad
\check{x}_1=\hat{x}_1=\mathrm{tr}\,M^1-e^{\pi i\theta_{\infty}}(x_1x_2+x_0
-\mathrm{tr}\,M^0), \quad \check{x}_2=\hat{x}_2=x_2.
\end{align*}
\begin{cor}\label{cor2.17b}
For $(M^0,M^1)\in\mathcal{M}(\theta_0,\theta_1,\theta_{\infty})$, let
$\mathfrak{m}(M^0,M^1)=(M^0_{(1)},M^1_{(1)})=M^1M^0(M^0,M^1)(M^1M^0)^{-1}$
with $M^{0,\,1}_{(1)}=(m^{0,\,1}_{(1),ij})$ and write
$x^{(1)}_0=m^0_{(1),11},$ $x^{(1)}_1=m^1_{(1),11},$ $x^{(1)}_2=\mathrm{tr}
(M^1_{(1)}M^0_{(1)}).$ Then $(x^{(1)}_0,x^{(1)}_1,x^{(1)}_2) \in \mathbf{V}
_{(\theta_0,\theta_1,\theta_{\infty})}$, and
\begin{align*}
&x^{(1)}_0= -x_1x_2-x_0 +\mathrm{tr}\,M^0+e^{-\pi i\theta_{\infty}}\mathrm{tr}
\,M^1,
\\
&x^{(1)}_1= x_1x_2^2+x_0x_2 -x_1-( \mathrm{tr}\,M^0+
e^{-\pi i\theta_{\infty}}\mathrm{tr}\,M^1)x_2+
 \mathrm{tr}\,M^1+e^{-\pi i\theta_{\infty}}\mathrm{tr}\,M^0,
\\
& x^{(1)}_2=x_2.
\end{align*} 
\end{cor}
\begin{rem}\label{rem2.13aaa}
The character variety $\mathbf{V}_{(\theta_0,\theta_1,\theta_{\infty})}$
has singular points if and only if $(\theta_0,\theta_1,\theta_{\infty})
\in \left(\bigcup_{n\in\mathbb{Z}}{P}_{0,n}\right)
\cup \left(\bigcup_{n\in\mathbb{Z}}{P}_{1,n}\right)
\cup \left(\bigcup_{n\in\mathbb{Z}}{P}_{\infty,n,+}\right) 
\cup \left(\bigcup_{n\in\mathbb{Z}}{P}_{\infty,n,-}\right),$ 
where $P_{0,n},$ $P_{1,n},$ $P_{\infty,n,\pm}$ are complex planes defined by 
$\theta_0 =n,$ $\theta_1 =n$, $(\theta_0+\theta_1)^2-(\theta_{\infty}-2n)^2=0,$ 
$(\theta_0-\theta_1)^2-(\theta_{\infty}-2n)^2=0,$ 
respectively \cite{Put}. 
\end{rem}
\begin{rem}\label{rem2.13aa}
The coordinates $(x_0,x_1,x_2)$ of the character variety are birationally
mapped to $(m^0_{11},m^1_{11},m^0_{21}m^1_{12})$, which are the coordinates
of the monodromy manifold.
\end{rem}
\begin{rem}\label{rem2.13a}
The relations of the corollary above represent the nonlinear monodromy action 
on the character variety $\mathbf{V}_{(\theta_0,\theta_1,\theta_{\infty})}$ 
caused by $\mathfrak{m}(M^0,M^1)$, which agrees with the monodromy operator 
given in \cite[Proposition 5.5]{Kli}. The operators $\mathfrak{s}_{0,\,-1}$ and
$\hat{\mathfrak{s}}_{0,\,1}$ are related to the analytic continuation, say, for
$x\to e^{\pm\pi i}x$ of Theorem \ref{thm2.2}.  
Our results suggest some possible relations between these
operators and 
the wild monodromy actions of \cite[Theorem 6.5]{Kli}, \cite{PR}.
\end{rem}
\subsection{Applications and examples}\label{ssc2.6}
By Theorem \ref{thm2.15} or \ref{thm2.17}, the elliptic expression of
Theorem \ref{thm2.2} is immediately
derived from Theorem \ref{thm2.1}. Let $y(x,M^0,M^1)$ be the solution 
in Theorem \ref{thm2.1} for $-\pi/2<\phi<0.$
For $\pi/2<\phi<\pi$ we have $y(x,M^0,M^1)= 1/y_{-\theta_{\infty}}
(e^{-\pi i}x, [\hat{M}^0]_{\sigma_1}, [\hat{M}^1]_{\sigma_1}),$ and then
\begin{align*}
\frac{y(x,M^0,M^1)+1}{y(x,M^0,M^1)-1}&
= -\frac{y_{-\theta_{\infty}}(e^{-\pi i}x, [\hat{M}^0]_{\sigma_1}, 
[\hat{M}^1]_{\sigma_1})+1} 
{y_{-\theta_{\infty}}(e^{-\pi i}x, [\hat{M}^0]_{\sigma_1}, 
[\hat{M}^1]_{\sigma_1})-1}
\\
&=-A_{\phi-\pi}^{1/2}\mathrm{sn}(\tfrac 12(e^{-\pi i}x -{x}^-_0);
A^{1/2}_{\phi-\pi})+O(x^{-1})
\\
&=A_{\phi}^{1/2}\mathrm{sn}(\tfrac 12(x+{x}_0^-);A^{1/2}_{\phi})+O(x^{-1}),
\end{align*}
in which
\begin{align*}
-x^-_0 &= -x_0(-\theta_{\infty}, [\hat{M}^0]_{\sigma_1},[\hat{M}^1]_{\sigma
_1}) 
\\
 &= -x_0(-\theta_{\infty}, [e^{-\frac{\pi i}2\theta_{\infty}\sigma_3}
\breve{M}^0 e^{\frac{\pi i}2\theta_{\infty}\sigma_3}]_{\sigma_1},
[e^{-\frac{\pi i}2\theta_{\infty}\sigma_3}\breve{M}^1
e^{\frac{\pi i}2\theta_{\infty}\sigma_3} ]_{\sigma_1}) 
\\
&= (\pi i)^{-1}\left(\Omega_{\mathbf{b}}\ln (e^{-\pi i\theta_{\infty}}
\breve{m}^0_{12}\breve{m}^1_{21})+\Omega_{\mathbf{a}}
 \ln (e^{-\frac{\pi i}2\theta_{\infty}}\breve{m}^0_{22})\right)
+\Omega_{\mathbf{a}}+\Omega_{\mathbf{b}}
\\
&\equiv -(\pi i)^{-1}\left(\Omega_{\mathbf{b}}\ln (e^{\pi i\theta_{\infty}}
(\breve{m}^0_{12}\breve{m}^1_{21})^{-1})
+\Omega_{\mathbf{a}} \ln (e^{\frac{\pi i}2\theta_{\infty}}
(\breve{m}^0_{22})^{-1}) \right)+\Omega_{\mathbf{a}}+\Omega_{\mathbf{b}}.
\end{align*}
This coincides with $\breve{x}_0 \!\! \mod 2\Omega_{\mathbf{a}}\mathbb{Z}
+2\Omega_{\mathbf{b}}\mathbb{Z}$, implying Theorem \ref{thm2.2}.
\par
In a general direction we also have the elliptic expression by using
$y(e^{2\pi ip}x, M^0,M^1)=y(x,M^0_{(p)},M^1_{(p)})$ and 
$y(e^{\pi i}x, M^0_{(p)},M^1_{(p)})
=1/y_{-\theta_{\infty}}(x,\mathfrak{s}_0 (M^0_{(p)},M^1_{(p)}))$ of Theorem 
\ref{thm2.17}.
Denote by $x_0=x_0(\theta_{\infty},M^0,M^1)$ the phase shift as in 
Theorem \ref{thm2.1}.   
\begin{thm}\label{thm2.18}
For each $p\in \mathbb{Z}$ let $M^{0,1}_{(p)}=((m^{0,1}_{p})_{ij})$
and $\hat{M}^{0,1}_{(p)}=((\hat{m}^{0,1}_{p})_{ij})$ be such that 
$(M^0_{(p)},M^1_{(p)})
\in \mathfrak{M}(M^0,M^1)$ and $([\hat{M}_{(p)}^0]_{\sigma_1},
[\hat{M}_{(p)}^1]_{\sigma_1})\in \hat{\mathfrak{M}}(M^0,M^1).$
\par
$(1)$ Suppose that $(m^0_{p})_{11}(m^0_{p})_{21}(m^1_{p})_{12}\not=0$ 
(respectively, $(m^1_{p})_{11}(m^0_{p})_{21}(m^1_{p})_{12}\not=0$). 
Then, for $-\pi/2<\phi-2\pi p<0$ (respectively, $0<\phi-2\pi p<\pi/2$),  
$y(x,M^0,M^1)$ admits the elliptic representation 
with the phase shift $x_0(\theta_{\infty},M^0_{(p)}, M^1_{(p)}).$ 
\par
$(2)$ Suppose that $(\hat{m}^0_{p})_{22}(\hat{m}^0_{p})_{12}
(\hat{m}^1_{p})_{21}\not=0$ 
(respectively, $(\hat{m}^1_{p})_{22}(\hat{m}^0_{p})_{12}
(\hat{m}^1_{p})_{21}\not=0$). 
Then, for $\pi/2<\phi-2\pi p<\pi$ (respectively, $\pi <\phi-2\pi p<3\pi/2$), 
$y(x,M^0,M^1)$ admits the elliptic representation with the phase shift 
$-x_0(-\theta_{\infty},[\hat{M}^0_{(p)}]_{\sigma_1},
[\hat{M}^1_{(p)}]_{\sigma_1}).$ 
\end{thm}
\begin{rem}\label{rem2.13}
The transitions of the phase shift shown in  
Theorems \ref{thm2.1}, \ref{thm2.2} and \ref{thm2.18}
may be regarded as explicit representations of 
the nonlinear Stokes relation. 
\end{rem}
\begin{rem}\label{rem2.14}
Since $\hat{M}^{0,1}_{(p)} =e^{-\frac{\pi i}2\theta_{\infty}\sigma_3}
 \breve{M}^{0,1}_{(p)}e^{\frac{\pi i}2\theta_{\infty}\sigma_3}$
with $\breve{M}^{0,1}_{(p)}=S_2^{-1}M^{0,1}_{(p)} S_2$, the second part of
Theorem \ref{thm2.18} may be written by using $\breve{M}^{0,1}_{(p)}
=((\breve{m}_p^{0,1})_{ij})$. 
\end{rem}
\begin{exa}\label{exa2.1}
Recall the doubly truncated solution
$$
y_0(x)= y(x,M^0,M^1)=\mathrm{trunc}(x,0,0)
=-1+4(\theta_0+\theta_1-1)x^{-1}+O(x^{-2})
$$ 
for $|\arg x|<\pi$ with $m^0_{11}=m^1_{11}=0,$ 
$m^0_{21}m^1_{12}=e^{-\pi i\theta_{\infty}},$ $m^0_{22}=2\cos\pi\theta_0$ 
and $m^1_{22}=2\cos\pi\theta_1$ (cf.~Theorem \ref{thm2.4}). 
The analytic continuation to the sector $0<\arg x\le \pi$ is
$$
y_0(x)=-1+4(\theta_0+\theta_1-1)x^{-1}-2\pi^{-1/2}e^{\pi i/4}e^{-\frac{\pi i}2
\theta_{\infty}}\hat{m}^1_{22} x^{-1/2} e^{ix/2}(1 +O(x^{-1/2}))
$$
with $\hat{m}^1_{22}=2(\cos\pi\theta_1+
e^{\pi i\theta_{\infty}}\cos\pi\theta_0),$ which agrees with \cite[(3.7),
$k=1$]{Andreev-Kitaev-2}.  
In the sector $\pi<\arg x<3\pi/2$,
\begin{align*}
&\frac{y_0(x)+1}{y_0(x)-1}=A^{1/2}_{\phi}\mathrm{sn}(\tfrac 12
(x-\hat{x}_0^*);A^{1/2}_{\phi})+O(x^{-1}),
\\
&\hat{x}_0^*= 
-(\pi i)^{-1} \Omega_{\mathbf{a}} \ln(2e^{-\frac{\pi i}2\theta_{\infty}}
(\cos\pi\theta_1+e^{\pi i\theta_{\infty}}\cos\pi\theta_0))
-\Omega_{\mathbf{a}}-\Omega_{\mathbf{b}} 
\end{align*}
if $\hat{m}^1_{22}=2(\cos\pi\theta_1+e^{\pi i\theta_{\infty}}\cos\pi\theta_0) 
\not=0,$ and for $3\pi/2<\arg x <2\pi$, 
$y_0(x)=y(x,M^0,M^1)=y(e^{-2\pi i}x, M_{(1)}^0,M_{(1)}^1)$ 
admits an elliptic expression
if $(m^0_1)_{21}(m^1_1)_{12}(m^0_1)_{11}\not=0.$ 
\begin{rem}\label{rem2.18}
If $\cos\pi\theta_1+e^{\pi i\theta_{\infty}}\cos\pi\theta_0=0$, then
$y_0(x)=-1+4(\theta_0+\theta_1-1)x^{-1}+O(x^{-2})$ for $-\pi<\arg x<+\infty$,
and if $\cos\pi\theta_1+e^{\pm\pi i\theta_{\infty}}\cos\pi\theta_0=0$, then
$y_0(x)$ is a rational function \cite[Corollary 1]{Andreev-Kitaev-2}.
\end{rem}
\par
{\it Derivation}.
Note that $y_0(x)=y(x,M^0,M^1)=1/y_{-\theta_{\infty}}(e^{-\pi i}x,
[\hat{M}^0]_{\sigma_1},[\hat{M}^1]_{\sigma_1})$ with 
$\hat{M}^{0,\,1}=(\hat{m}^{0,\,1}_{ij})$ in the sector $|\arg x-\pi|<\pi$, 
and that 
$$
[\hat{M}^0]_{\sigma_1} =\begin{pmatrix} 0 & \hat{m}^0_{21} \\
\hat{m}^0_{12} & 2\cos \pi\theta_0   \end{pmatrix},\quad
[\hat{M}^1]_{\sigma_1} =\begin{pmatrix} 2(\cos\pi\theta_1+ e^{\pi i\theta
_{\infty}}\cos\pi\theta_0) & \hat{m}^1_{21} \\
\hat{m}^1_{12} & -2e^{\pi i\theta_{\infty}}\cos \pi\theta_0 \end{pmatrix}
$$
with $\hat{m}^0_{12}\hat{m}^1_{21}=e^{\pi i\theta_{\infty}}.$ 
To find $y_{-\theta_{\infty}}(e^{-\pi i}x, [\hat{M}^0]_{\sigma_1},[\hat{M}^1]
_{\sigma_1})$, let us consider, as a candidate yielding $y_{-\theta_{\infty}}$, 
the truncated solution in the sector 
$-\pi<\arg x\le 0$ as in Theorem \ref{thm2.4}
\begin{align*}
\mathrm{trunc}&(x,0,{m}^1_{*,11})
=-1+4(\theta_0+\theta_1-1)x^{-1}
+2\sqrt{2}e^{\pi i/4}\hat{v}x^{-1/2}e^{-ix/2}(1+O(x^{-1/2}))
\\
&=-1+4(\theta_0+\theta_1-1)x^{-1}
-2i\pi^{-1/2}e^{\pi i/4}e^{\frac{\pi i}2\theta_{\infty}}m^1_{*,11}
x^{-1/2}e^{-ix/2}(1+O(x^{-1/2}))
\end{align*}
labelled with the monodromy data $({M}^0_*,
{M}^1_*)=((m^0_{*,ij}),(m^1_{*,ij}))\in \mathcal{M}(\theta_0,\theta_1,\theta
_{\infty})$ such that ${m}^0_{*,11}=0,$ ${m}^0_{*,21}{m}^1_{*,12}
=e^{-\pi i\theta_{\infty}},$ ${m}^1_{*,11}=i\sqrt{2\pi}e^{-\frac{\pi i}2
\theta_{\infty}}\hat{v}$. Then by
$\hat{m}^0_{12}\hat{m}^1_{21}=e^{\pi i\theta_{\infty}}$ and 
$y_{-\theta_{\infty}}(e^{-\pi i}x,[\hat{M}^0]_{\sigma_1},
[\hat{M}^1]_{\sigma_1})=\mathrm{trunc}_{-\theta_{\infty}}(e^{-\pi i}x,
0,\hat{m}^1_{22})$, we have the expression of $y_0(x)$  
for $0<\arg x\le \pi$ 
with $\hat{m}^1_{22}=2(\cos\pi\theta_1+
e^{\pi i\theta_{\infty}}\cos\pi\theta_0).$ Theorem \ref{thm2.2} provides 
elliptic asymptotics in the sectors $\pi<\arg x<3\pi/2$ and 
$3\pi/2<\arg x <2\pi$. 
\end{exa}
\begin{exa}\label{exa2.2}
The truncated solution $\mathrm{trunc}^1_0(x,m^0_{12}/m^1_{12},e^{-\pi i\theta_0}
)$ of Theorem \ref{thm2.6},(2) is expressed by 
\begin{equation*}
x y(x,M^0,M^1)=-\frac 12(\theta_0-\theta_1-\theta_{\infty})(1+O(x^{-1}))
-c^-x^{2\theta_0-\theta_{\infty}-1}e^{-x}(1+O(x^{-r}))
\end{equation*}
as $x\to \infty$ through $-\pi/2\le \arg x<\pi/2$ with 
$c^-x^{2\theta_0-\theta_{\infty}-1} e^{-x} \ll x^{-r},$ where 
$$
M^0=\begin{pmatrix}
e^{-\pi i\theta_0}  & \frac{2\pi i e^{\pi i(\theta_{\infty}
-\theta_0)} } {\Gamma(1-\theta_0)}c_0
 \\  0  & e^{\pi i\theta_0}  
\end{pmatrix}, \quad
M^1=\begin{pmatrix}
 e^{\pi i(\theta_0-\theta_{\infty})} &    \frac{2\pi i }
{\Gamma(\frac12 (\theta_0-\theta_1-\theta_{\infty}) )
\Gamma(\frac12 (\theta_0+\theta_1-\theta_{\infty}) ) } c_x \\
m^1_{21} & m^1_{22}
\end{pmatrix}
$$
and 
\begin{equation*}
c^-=\frac{m^0_{12}}{m^1_{12}}\cdot\frac{ e^{\pi i(\theta_0-\theta_{\infty})}
\Gamma(1-\theta_0)}{\Gamma(\frac {\theta_0-\theta_1-\theta_{\infty}}2)
\Gamma(\frac {\theta_0+\theta_1-\theta_{\infty}}2)}.
\end{equation*}
The expression above is inaccurate near the boundary $\arg x=\pi/2.$
The analytic continuation to the sector $-\pi/2 < \arg x \le \pi/2$, 
in particular, near $\arg x=\pi/2$, is given by
\begin{equation*}
xy(x,M^0,M^1)= -\frac 12(\theta_0-\theta_1-\theta_{\infty}) (1+O(x^{-1}))-
c_{\pi/2}x^{2\theta_0-\theta_{\infty}-1}e^{-x}(1+O(x^{-r})) 
\end{equation*}
with
\begin{equation*}
c_{\pi/2}=c^-+\frac{2\pi i}
{\Gamma(\theta_0)\Gamma(\frac {\theta_0+\theta_1-\theta_{\infty}}2)
\Gamma(\frac {\theta_0-\theta_1-\theta_{\infty}}2)}.
\end{equation*}
\par
{\it Derivation}. Note that
$$
[\hat{M}^0]_{\sigma_1}=\begin{pmatrix}
e^{\pi i\theta_0} & 0 \\ e^{-\pi i\theta_{\infty}}(m^0_{12}+(e^{-\pi i\theta_0}
-e^{\pi i \theta_0})s_2) & e^{-\pi i\theta_0}
\end{pmatrix}, \quad
[\hat{M}^1]_{\sigma_1}=\begin{pmatrix}  \hat{m}^1_{22} & 
e^{\pi i\theta_{\infty}}m^1_{21} \\ \hat{m}^1_{12}  &  \hat{m}^1_{11}
  \end{pmatrix},
$$
where $\hat{M}^{0,\,1}=(\hat{m}^{0,\,1}_{ij})$ and 
$s_2=-e^{\pi i\theta_{\infty}}(m^1_{11}m^0_{12}+m^1_{12}m^0_{22})
=-e^{\pi i\theta_0}(m^0_{12}+e^{\pi i\theta_{\infty}}m^1_{12}).$ 
For our purpose it is necessary to know a solution for $-3\pi/2<\arg x
\le -\pi/2$ in Remark \ref{rem5.3}.
Observe the form of the monodromy data \eqref{5.5} for the solution 
\eqref{5.4} $y(x,M^0_*,M^1_*)$ with $(M^0_*,M^1_*)\in\mathcal{M}(\theta_0,
\theta_1,\theta_{\infty})$ in $-{3\pi}/2 <\arg x\le -{\pi}/2$ 
having the integration constant 
$$
c_*=m_{*,21}^0m_{*,12}^1(2\pi i)^{-2}e^{\pi i\theta_{\infty}}\Gamma(1-\theta_0)
\Gamma(1-\tfrac {\theta_0+\theta_1+\theta_{\infty}}2)
\Gamma(1-\tfrac {\theta_0-\theta_1+\theta_{\infty}}2),
$$
and consider $y_{-\theta_{\infty}}(x,M^0_-, M^1_-)$
with $(M_-^0,M_-^1)\in\mathcal{M}(\theta_0,
\theta_1,-\theta_{\infty}).$ Then we set
\begin{align*}
\frac{e^{-\pi i}x}{y_{-\theta_{\infty}}(e^{-\pi i}x,[\hat{M}^0]_{\sigma_1},
[\hat{M}^1]_{\sigma_1})}
=&\frac 12(\theta_0-\theta_1-\theta_{\infty})(1+O(x^{-1}))
\\
&+e^{-\pi i(2\theta_0-\theta_{\infty}-1)}\hat{c}_{-}e^{-x}x^{2\theta_0-\theta
_{\infty}-1}(1+O(x^{-r})),
\end{align*}
with
$$
\hat{c}_{-}=\hat{m}^0_{12}\hat{m}^1_{21}(2\pi i)^{-2}e^{-\pi i\theta
_{\infty}}\Gamma(1-\theta_0)\Gamma(1-\tfrac {\theta_0+\theta_1-\theta_{\infty}}
2)\Gamma(1-\tfrac {\theta_0-\theta_1-\theta_{\infty}}2).
$$
Here
\begin{align*}
\hat{m}^0_{12}\hat{m}^1_{21}=&
(m^0_{12}+(e^{-\pi i\theta_0}-e^{\pi i\theta_0})s_2)m^1_{21}
\\
=&(m^0_{12}-(e^{-\pi i\theta_0}-e^{\pi i\theta_0})e^{\pi i\theta_0}(m^0_{12}
+e^{\pi i\theta_{\infty}}m^1_{12}))m^1_{21}
\\
=& m^0_{12}m^1_{21}e^{2\pi i\theta_0}-e^{\pi i(\theta_{\infty}+\theta_0)}
(e^{-\pi i\theta_0}-e^{\pi i\theta_0})m^1_{12}m^1_{21}
\\
=&\Bigl(e^{2\pi i\theta_0} \frac{m^0_{12}}{m^1_{12}} -e^{\pi i(\theta_{\infty}
+\theta_0)}(e^{-\pi i\theta_0}-e^{\pi i\theta_0})\Bigr)(m^1_{11}m^1_{22}-1)
\\
=&-\Bigl(e^{2\pi i\theta_0} \frac{m^0_{12}}{m^1_{12}} +2i e^{\pi i(\theta
_{\infty}+\theta_0)}\sin\pi\theta_0\Bigr)(e^{\pi i(\theta_0+\theta_1
-\theta_{\infty})}-1)(e^{\pi i(\theta_0-\theta_1-\theta_{\infty})}-1)
\\
=&-\Bigl(e^{2\pi i\theta_0}\frac
{m^0_{12}}{m^1_{12}}+\frac{2\pi ie^{\pi i(\theta_{\infty}+\theta_0)}}
{\Gamma(\theta_0)\Gamma(1-\theta_0)}\Bigr)
\\
&\times
\frac{(2\pi i)^2e^{\pi i(\theta_0-\theta_{\infty})}}
{\Gamma(\frac {\theta_0+\theta_1-\theta_{\infty}}2)
\Gamma(1-\frac {\theta_0+\theta_1-\theta_{\infty}}2)
\Gamma(\frac {\theta_0-\theta_1-\theta_{\infty}}2)
\Gamma(1-\frac {\theta_0-\theta_1-\theta_{\infty}}2)}.
\end{align*}
Since $m^0_{12}/m^1_{12}=c^- e^{\pi i(\theta_{\infty}-\theta_0)}\frac
{\Gamma(\frac 12(\theta_0-\theta_1-\theta_{\infty}))
\Gamma(\frac 12(\theta_0+\theta_1-\theta_{\infty}))}{\Gamma(1-\theta_0)}$,
we have
\begin{align*}
\hat{c}_{-}&=-\Bigl(e^{2\pi i\theta_0}\frac{m^0_{12}}{m^1_{12}}
+\frac{2\pi ie^{\pi i(\theta_{\infty}+\theta_0)}}
{\Gamma(\theta_0)\Gamma(1-\theta_0)}\Bigr)\frac{e^{\pi i(\theta_0-2\theta
_{\infty})}\Gamma(1-\theta_0)}
{\Gamma(\frac {\theta_0+\theta_1-\theta_{\infty}}2)
\Gamma(\frac {\theta_0-\theta_1-\theta_{\infty}}2)}
\\
&=-e^{\pi i(2\theta_0-\theta_{\infty})}c^- -
\frac{2\pi i e^{\pi i(2\theta_0-\theta_{\infty})}}
{\Gamma(\theta_0)\Gamma(\frac {\theta_0+\theta_1-\theta_{\infty}}2)
\Gamma(\frac {\theta_0-\theta_1-\theta_{\infty}}2)}.
\end{align*}
Then setting $c_{\pi/2}:=-e^{-\pi i(2\theta_0-\theta_{\infty})}\hat{c}_{-}$, 
we obtain the desired analytic continuation
\begin{align*}
xy(x,M^0,M^1)=&\frac{x}{y_{-\theta_{\infty}}(e^{-\pi i}x,[\hat{M}^0]_{\sigma_1},
[\hat{M}^1]_{\sigma_1}))}
\\
=&-\frac 12(\theta_0-\theta_1-\theta_{\infty}) (1+O(x^{-1}))-
c_{\pi/2}x^{2\theta_0-\theta_{\infty}-1}e^{-x}(1+O(x^{-r})) 
\end{align*}
in the sector $-\pi/2<\arg x\le \pi/2$ with $c_{\pi/2}
x^{2\theta_0-\theta_{\infty}-1}e^{-x}\ll x^{-r}.$
\end{exa}
\section{Basic facts}\label{sc3}
\subsection{Another isomonodromy system}\label{ssc3.1}
For every $k\in \mathbb{Z}$ the linear system \eqref{1.1} admits the matrix 
solution
\begin{equation*}
\Xi_k(x,\xi)=(I+O(\xi^{-1}))\exp(\tfrac 12(x\xi-\theta_{\infty}\ln\xi)\sigma_3)
\end{equation*}
as $\xi\to\infty$ through the sector $\Sigma_k:$ 
$|\arg(x\xi)-\pi/2-(k-2)\pi|<\pi$. Then the solution $Y_k(t,\lambda)$
of \eqref{2.1} is given by
\begin{align*}
&Y_k(t,\lambda)=\exp(-\varpi(t,\phi)\sigma_3)\Xi_k(x,\xi),
\\
&x=e^{i\phi}t,\quad \xi=\tfrac 12(e^{-i\phi}\lambda+1),\quad \varpi(t,\phi)
=\tfrac 14e^{i\phi}t+\tfrac 12\theta_{\infty}(i\phi+\ln 2)
\end{align*}
as $\lambda\to\infty$ (or $\xi\to\infty$) through the sector 
$|\arg\lambda-\pi/2-(k-2)\pi|<\pi$ (or $\Sigma_k$), which may be verified
by comparing the asymptotic behaviours on both sides and by the uniqueness
of a canonically behaving solution. This fact implies that 
$(M^0,M^1,S_1,S_2)$ defined by $Y_k(t,\lambda)$ in Section \ref{sc2} 
coincide with the monodromy data defined by $\Xi_k(x,\xi)$, that is, 
$(M^0,M^1,S_1,S_2)$ of \cite{Andreev-Kitaev}, in which $(M^0,M^1)$ is 
given by the analytic continuation of $\Xi_2(x,\xi)$ along 
the loops surrounding $\xi=0$ and $1$. 
Moreover consider another linear system of the form \cite[(1.1)]{S-2018}
\begin{equation}\label{3.1}
\frac{d\Psi}{d\lambda}=\Bigl(\frac{\sigma_3}2+\frac{A_0}{\lambda}+\frac{A_1}
{\lambda-x}\Bigr)\Psi
\end{equation} 
with $A_{0,\, 1}=x^{-\frac12 {\theta_{\infty}}\sigma_3}\mathcal{A}_{0,\,1} 
x^{\frac12{\theta_{\infty}}\sigma_3},$ in other words 
$A_{0,\, 1}=\mathcal{A}_{0,\,1}|_{u \mapsto 
u x^{-\theta_{\infty}}}.$ System \eqref{3.1} admits 
the canonical solution
$$
\Psi_k(x,\lambda)=(I+O(\lambda^{-1}))\exp(\tfrac12 (\lambda-\theta_{\infty}
\ln \lambda)\sigma_3)
$$
in the sector $|\arg \lambda-\pi/2-(k-2)\pi|<\pi$. Note that
${\Xi}_k^*(x,\xi):=x^{-\frac12{\theta_{\infty}}\sigma_3}\Xi_k(x,\xi)
x^{\frac12{\theta_{\infty}}\sigma_3}$ solves 
$$
\frac{d\Xi^*}{d\xi}=\Bigl(\frac x2\sigma_3+\frac{A_0}{\xi}+\frac{A_1}
{\xi-1}\Bigr)\Xi^*=x^{-\frac12{\theta_{\infty}}\sigma_3} \Bigl(\frac x2\sigma_3 
+\frac{\mathcal{A}_0} {\xi} +\frac{\mathcal{A}_0 }{\xi-1}\Bigr)
x^{\frac12{\theta_{\infty}}\sigma_3}\Xi^*,
$$
which is related to system \eqref{3.1} via $x\xi=\lambda$. 
Then we have $\Psi_k(x,\lambda)
=\Xi^*_k(x,\xi)x^{-\frac12{\theta_{\infty}}\sigma_3} 
=x^{-\frac12{\theta_{\infty}}\sigma_3}\Xi_k(x,\xi).$ This implies that
the monodromy data $(M^0,M^1, S_1, S_2)$ is also defined by the canonical
solution $\Psi_k(x,\lambda)$ of \eqref{3.1}, and that
the isomonodromy deformation of \eqref{3.1} is also governed by \eqref{1.2} with 
$(y,\mathfrak{z},u)$. 
Thus we arrive the following fact. 
\begin{prop}\label{prop3.0}
The isomonodromy system \eqref{1.1}, \eqref{2.1} and \eqref{3.1} 
containing, respectively,
$(y,\mathfrak{z},u)$, $(y,\mathfrak{z}, ue^{-2\varpi(t,\phi)})$ and
$(y,\mathfrak{z},ux^{-\theta_{\infty}})$ define the same monodromy data 
$(M^0,M^1,S_1,S_2)$. Here $y,$ $\mathfrak{z}$ and $u$ of each triplet 
solve \eqref{1.2}, which governs the isomonodromy deformation of these systems.
\end{prop}
\begin{rem}\label{rem3.0}
The elliptic asymptotics of Theorems \ref{thm2.1} and \ref{2.2} are derived
from isomonodromy system \eqref{2.1}. The trigonometric asymptotic solutions
on the positive real axis of Theorem \ref{thm2.3} and the related truncated
solutions of Theorem \ref{thm2.4} are based on system \eqref{1.1} 
\cite{Andreev-Kitaev}. The truncated solutions of Theorems \ref{thm2.6} and
\ref{thm2.8} arising from general solutions along the imaginary axis are
obtained by using isomonodromy system \eqref{3.1} in Section \ref{sc5}
\cite{S-2018}, and such general solutions follow also from system
\eqref{1.1} \cite{Andreev-Kitaev-2019}.
\end{rem}
\subsection{Some properties of the monodromy data}\label{ssc3.2}
The following fact is given by \cite[Proposition 2.1]{Andreev-Kitaev} 
for \eqref{1.1} without a proof, and is similarly verified for \eqref{1.1}
and \eqref{3.1}. 
\begin{prop}\label{prop3.1}
If $M^0\not=\pm I$ and $M^1\not=\pm I$, a canonical solution $Y_2(t,\lambda)$ 
of \eqref{2.1} defining $(M^0,M^1)$ is uniquely determined.
\end{prop}
\begin{proof}
Let 
$$
Y_2(\lambda),\,\,\, \tilde{Y}_2(\lambda)=(I+O(\lambda)) \exp(\tfrac 14
(t\lambda - 2\theta_{\infty}\ln\lambda)\sigma_3), \quad |\arg\lambda-\pi/2|<\pi
$$ 
be canonical solutions related to $(y,\mathfrak{z},u)$ and 
$(\tilde{y},\tilde{\mathfrak{z}},\tilde{u})$, respectively, and suppose that 
$Y_2(\lambda)$ and $\tilde{Y}_2(\lambda)$ yield the same monodromy data 
$(M^0,M^1).$ Let
$$
Y_2(\lambda)=G_0(I+O(\lambda_-))\lambda_-^{\frac12 {\theta_0}\sigma_3+L}C_0,
 \,\,\,\, 
\tilde{Y}_2(\lambda)=\tilde{G}_0(I+O(\lambda_-))\lambda_-^{\frac12 {\theta_0}
\sigma_3+\tilde{L}}\tilde{C}_0
$$ 
with $\lambda_-=\lambda+e^{i\phi}$ around $\lambda=-e^{i\phi},$ 
where $G_0, \tilde{G_0}, C_0, \tilde{C_0}\in \mathrm{SL}_2(\mathbb{C})$, 
and $L=(l_{ij})$ and $\tilde{L}=(\tilde{l}_{ij})$ are upper-triangular
(respectively, lower-triangular) matrices such that diagonal entries vanishes. 
Then $M^0=C_0^{-1}e^{\pi i\theta_0\sigma_3}e^{2\pi i L}C_0
=\tilde{C}_0^{-1}e^{\pi i\theta_0\sigma_3}e^{2\pi i \tilde{L}}\tilde{C}_0$.
Note that $M^0=\pm I$
if and only if $\theta_0\in \mathbb{Z}$ and $L=\tilde{L}=0.$ Hence $M^0\not=\pm
I$ implies one of the following: (a) $\theta_0\not\in\mathbb{Z}$ and $L=\tilde
{L}=0;$ (b) $\theta_0 \in \mathbb{Z}_{\ge 0}$ and $l_{12}\tilde{l}_{12}\not=0;$
or (c) $\theta_0 \in \mathbb{Z}_{\le 0}$ and $l_{21}\tilde{l}_{21}\not=0.$ Then
the relation $(C_0\tilde{C}_0^{-1})^{-1}e^{\pi i\theta_0}e^{2\pi iL} 
C_0\tilde{C}_0^{-1}=e^{\pi i\theta_0}e^{2\pi i\tilde{L}}$ leads to 
$C_0\tilde{C}_0^{-1}=\mathrm{diag}[\alpha,\alpha^{-1}]$ or 
$\mathrm{diag}[\alpha,\alpha^{-1}]+L_*$ with some $\alpha\not=0$ and 
some upper- or lower-triangular matrix $L_*$ of the same form as in
(b) or (c). By this fact 
$$
Y_2(\lambda)\tilde{Y}_2(\lambda)^{-1}=G_0(I+O(\lambda_-))\lambda_-^{\frac 12
{\theta_0}\sigma_3}e^{2\pi iL}C_0\tilde{C}_0^{-1}e^{-2\pi i\tilde{L}}\lambda_-
^{-\frac12 {\theta_0}\sigma_3}(I+O(\lambda_-))\tilde{G}_0^{-1}
$$
is holomorphic around $\lambda= -e^{i\phi}.$ Similarly $Y_2(\lambda)\tilde{Y}_2
(\lambda)^{-1}$ is holomorphic around $\lambda=e^{i\phi}$ if $M^1\not=\pm I$. 
Observing that
$Y_2(\lambda)\tilde{Y}_2(\lambda)^{-1}=I+O(\lambda^{-1})$ as $\lambda\to
\infty$ in $|\arg\lambda-\pi/2|<\pi$, and using the Phragm\'en-Lindel\"of
reasoning, we conclude that $Y_2(\lambda)=\tilde{Y}_2(\lambda).$
\end{proof}
\begin{prop}\label{prop3.2}
Suppose that $\theta_0+\theta_1+\theta_{\infty}\in -2\mathbb{N}\cup\{0\}$
(respectively, $\theta_0-\theta_1-\theta_{\infty}\in 2\mathbb{N}\cup\{0\}$) and 
$\theta_0,$ $\theta_1\not\in \mathbb{Z}.$ 
Then, for system \eqref{3.1}, there exist no monodromy data $(M^0,M^1)$ 
such that $m^{0}_{11}=e^{\pi i\theta_0}$ (respectively, 
$m^0_{11}=e^{-\pi i\theta_0}$), $m^{1}_{11}=e^{\pi i\theta_1}$ and
$m^0_{12}=m^1_{12}=0.$
\end{prop}
\begin{proof}
Set $x^{-\theta_{\infty}}u =1$ in \eqref{3.1}.
Let $\Psi_2(\lambda)=(I+O(\lambda^{-1}))e^{\frac12 {\lambda}\sigma_3}
\lambda^{\frac 12(\theta_0+\theta_1)\sigma_3}\lambda^{(N-1)\sigma_3}$ in 
$|\arg\lambda-\pi/2|<\pi$
be the canonical solution of \eqref{3.1} with $\theta_0+\theta_1
+\theta_{\infty}=-2(N-1)$ for $N\in\mathbb{N}$. Suppose that there exist the 
monodromy data $(M^0,M^1)$ as in the proposition given by the analytic
continuation of $\Psi_2(\lambda).$ Write, around $\lambda =0$ and $\lambda=x$,  
$$
\Psi_2(\lambda)=G_0 (I+O(\lambda))\lambda^{\frac12 {\theta_0}
\sigma_3} C_0,
\quad
\Psi_2(\lambda)=G_x (I+O(\lambda_x))\lambda_x^{\frac12 {\theta_1}
\sigma_3} C_x,
$$ 
where $\lambda_{x}=\lambda- x$,
$C_0=(c^0_{ij}),$ $C_x=(c^x_{ij})\in \mathrm{SL}_2(\mathbb{C}),$ 
and $G_0=(g^0_{ij})$, $G_x=(g^x_{ij})\in \mathrm{SL}_2(\mathbb{C})$ 
may be chosen such that $g^0_{12}=1/\theta_0,$ $g^0_{22}=1,$
$g^x_{12}=y/\theta_1,$ $g^x_{22}=1.$ Then,
by $M^0=C_0^{-1}e^{\pi i\theta_0\sigma_3} C_0$ and
$M^1=C_x^{-1}e^{\pi i\theta_1\sigma_3} C_x$ with $\theta_0,
\theta_1 \not\in \mathbb{Z},$ we have $c^0_{12}=c^x_{12}=0.$
Denote by $\psi_{12}(\lambda)$ the (1,2)-entry of $\Psi_2(\lambda)$, which is 
expressed as  
$$
\psi_{12}(\lambda)=(c_0^*\lambda^{-1}+O(\lambda^{-2})) e^{-\lambda/2}
\lambda^{-(\theta_0+\theta_1)/2}\lambda^{-N+1}
$$ 
as $\lambda\to \infty$ in $|\arg \lambda-\pi/2|<\pi$ for some $c_0^*$, and  
$$
\psi_{12}(\lambda)=c^0_{22} ( 1/\theta_0+O(\lambda) )
\lambda^{- {\theta_0}/2}, \quad
\psi_{12}(\lambda)=c^x_{22} (y/\theta_1+O(\lambda_x))
\lambda_x^{- {\theta_1}/2}, 
$$ 
around $\lambda=0$ and $x$. Then the function $\psi_0(\lambda)=\psi_{12}
(\lambda)e^{\lambda/2}\lambda^{ {\theta_0}/2}
\lambda_x^{ {\theta_1}/2}$ is holomorphic around 
$\lambda=0$ and $x$, and is entire on $\mathbb{C}$. 
In $|\arg \lambda-\pi/2|<\pi$ near 
$\lambda=\infty$, we have $\psi_0(\lambda)=O(\lambda^{-N})$. By the 
Phragm\'en-Lindel\"of reasoning, $\psi_0(\lambda)$ is bounded on $\mathbb{C}$,
and hence $\psi_0(\lambda)\equiv 0$, which contradicts  the behaviour of
$\psi_{12}(\lambda)$ around $\lambda=0$ and $x$.
\end{proof}
\begin{rem}\label{rem3.1}
Supposing $\theta_0,$ $\theta_1 \not\in \mathbb{Z},$ we may show the following
in a similar manner:
\par
(1) if $\theta_0-\theta_1+\theta_{\infty}\in -2\mathbb{N}\cup\{0\}$ 
(respectively, $\theta_0+\theta_1-\theta_{\infty}\in 2\mathbb{N}$), then 
the case $m^0_{11}=e^{\pi i\theta_0}$ (respectively, $=e^{-\pi i\theta_0})$, 
$m^1_{11}=e^{-\pi i \theta_1},$ $m^0_{12}=m^1_{12}=0$ does not occur.
\par
(2) if $\theta_0-\theta_1-\theta_{\infty}\in -2\mathbb{N}\cup\{0\}$ 
(respectively, 
$\theta_0+\theta_1-\theta_{\infty}\in -2\mathbb{N}\cup\{0\}$), then 
the case $m^1_{11}=e^{\pi i\theta_1}$ (respectively, $=e^{-\pi i\theta_1})$, 
$m^0_{11}=e^{-\pi i \theta_0},$ $m^0_{21}=m^1_{21}=0$ does not occur.
\par
(3) if $\theta_0+\theta_1+\theta_{\infty}\in 2\mathbb{N}$ (respectively, 
$\theta_0-\theta_1+\theta_{\infty}\in 2\mathbb{N}\cup\{0\}$), then 
the case $m^1_{11}=e^{\pi i\theta_1}$ (respectively, $=e^{-\pi i\theta_1})$, 
$m^0_{11}=e^{\pi i \theta_0},$ $m^0_{21}=m^1_{21}=0$ does not occur.
\end{rem}
\section{Proofs of Theorems \ref{thm2.14} and \ref{thm2.15}}\label{sc4}
\subsection{Proof of Theorem \ref{thm2.14}}\label{sc4.1}
Let 
$$
\mathcal{E}(\lambda,\phi,y(x),\mathfrak{z}(x), \hat{u}(x),Y_k(\lambda), 
\theta_{\infty})=0, \quad x=t e^{i\phi}
\phantom{-------}
$$
denote the isomonodromy system \eqref{2.1} admitting the canonical solution
$Y_k(\lambda)=Y_k(t,\lambda)$ such that
$Y_k(\lambda)=(I+O(\lambda^{-1}))\exp(\tfrac 14(t\lambda-2\theta_{\infty}
\ln\lambda)\sigma_3)$ as $\lambda\to \infty$ in the sector $|\arg\lambda
-(k-2)\pi-\pi/2|<\pi.$ For $(M^0,M^1)$ the solution $y(x)=y(x,M^0,M^1)$ 
is defined by 
\begin{equation}\label{4.0}
\mathcal{E}(\lambda,\phi,y(x),\mathfrak{z}(x), \hat{u}(x),Y_2(\lambda), 
\theta_{\infty})=0, \quad Y_2(\lambda)|_{\hat{l}_{0,1}\ni\lambda}
 =Y_2(\lambda)M^{0,\,1}
\end{equation}
for $ |\phi|<\pi/2,$ $ |\arg\lambda-\pi/2|<\pi$.
Here read $|_{\hat{l}_{0,1}\ni\lambda}$ as the result of analytic 
continuation along each loop of $\hat{l}_0$ and $\hat{l}_1$.  
For $p\in \mathbb{Z}$ the replacement $(\lambda,\phi,x) \mapsto
(e^{2\pi ip}\lambda, 2\pi p+\phi, e^{2\pi ip}x)$ 
with $ |\phi|<\pi/2,$ $|\arg\lambda-\pi/2|<\pi$ leads to
$$
\mathcal{E}(e^{2\pi ip}\lambda,2\pi p+\phi,y(e^{2\pi ip}x),Y_2(e^{2\pi ip}
\lambda))=0, \quad Y_2(e^{2\pi ip}\lambda)|_{
\hat{l}^p_{0,1}\ni e^{2\pi ip}\lambda } =Y_2(e^{2\pi ip}\lambda)M^{0,\,1},
$$
in which
$\hat{l}^p_{0,1}:= e^{2\pi ip}\hat{l}_{0,1}$.  
The entries $\mathfrak{z}(e^{2\pi ip}x),$ $\hat{u}(e^{2\pi ip}x)$ and 
$\theta_{\infty}$ are omitted from the expression of $\mathcal{E}$.
By using $Y_2(e^{2\pi ip}\lambda)
=Y_{2-2p}(\lambda)e^{-\pi i\theta_{\infty}p\sigma_3}$, this is written 
in the form
$$
\mathcal{E}(\lambda,\phi,y(e^{2\pi ip}x),Y_{2-2p}(\lambda))=0, \quad 
Y_{2-2p}(\lambda)|_{\hat{l}_{0,1}\ni\lambda} 
=Y_{2-2p}(\lambda)e^{-\pi i\theta_{\infty}p\sigma_3}M^{0,\,1}
e^{\pi i\theta_{\infty}p\sigma_3}
$$
for $|\phi|<\pi/2,$ $|\arg\lambda-\pi/2|<\pi.$ 
If $p \ge 1$, noticing that $Y_{2-2p}(\lambda)S_{2-2p}S_{3-2p}
\cdots S_1=Y_2(\lambda)$ and setting
$$
M^{0,\,1}_{(p)}:=(S_{2-2p}S_{3-2p}\cdots S_0S_1)^{-1}e^{-\pi i\theta_{\infty}p
\sigma_3}M^{0,\,1}e^{\pi i\theta_{\infty}p\sigma_3}S_{2-2p}S_{3-2p}\cdots S_0S_1,$$
we have, for $|\phi|<\pi/2,$ $|\arg\lambda-\pi/2|<\pi,$ 
$$
\mathcal{E}(\lambda,\phi,y(e^{2\pi ip}x),Y_{2}(\lambda))=0, \quad 
Y_{2}(\lambda)|_{\hat{l}_{0,1}\ni\lambda}
 =Y_{2}(\lambda)M^{0,\,1}_{(p)},
$$
which implies $y(e^{2\pi ip}x, M^0,M^1)=y(x,M^0_{(p)},M^1_{(p)})$.
If $p \le -1$, let
$$
M^{0,\,1}_{(p)}:=S_{2}S_{3}\cdots S_{1-2p}e^{-\pi i\theta_{\infty}p
\sigma_3}M^{0,\,1}e^{\pi i\theta_{\infty}p\sigma_3}(S_2S_3\cdots S_{1-2p})^{-1}.
$$
Then by $Y_2(\lambda)S_2S_3\cdots S_{1-2p}=Y_{2-2p}(\lambda)$ the same relation
follows. By using \eqref{2.3} it is easily verified that  
$M^{0,\,1}_{(p)}=e^{-\pi i\theta_{\infty}p\sigma_3} U_p^{-1}M^{0,\,1}U_p
e^{\pi i\theta_{\infty}p\sigma_3},$ where $U_p$ is such that
$$
U_p :=\begin{cases}  S_2 S_3 \cdots S_{2p}S_{2p+1} \quad & \text{if $p\ge 1$;}
\\
 I \quad & \text{if $p=0$;}
\\
 S_1^{-1}S_0^{-1} \cdots S_{2p+1}^{-1}S_{2p+2}^{-1} \quad &\text{if $p\le -1$.}
\end{cases}
$$  
Combining this expression of $M^{0,\,1}_{(p)}$ with the following fact 
we immediately obtain the conclusion of Theorem \ref{thm2.14}.
\begin{lem}\label{lem4.2}
For every $p\in \mathbb{Z}$, $U_p=(M^1M^0)^{-p}e^{-\pi i
 \theta_{\infty}p\sigma_3}.$
\end{lem}
\begin{proof}
For $p\ge 1$, by \eqref{2.3}, $e^{-\pi i\theta_{\infty}\sigma_3}(S^{-1}_2U_p)
e^{\pi i\theta_{\infty}\sigma_3}=e^{-\pi i\theta_{\infty}\sigma_3}(S_3\cdots
S_{2p+1})e^{\pi i\theta_{\infty}\sigma_3}=S_1S_2\cdots S_{2p-1}=S_1U_{p-1},$
which implies $S^{-1}_2U_p= e^{\pi i\theta_{\infty}\sigma_3}S_1U_{p-1}
e^{-\pi i\theta_{\infty}\sigma_3}.$
Then, by \eqref{2.2}, $S_1^{-1}e^{-\pi i\theta_{\infty}\sigma_3}S^{-1}_2U_p
=M^1M^0 U_p=U_{p-1}e^{-\pi i\theta_{\infty}\sigma_3},$ that is,
$ U_p=(M^1M^0)^{-1}U_{p-1}e^{-\pi i\theta_{\infty}\sigma_3}.$ 
This yields the desired formula.
\end{proof}
\subsection{Proof of Theorem \ref{thm2.15}}\label{ssc4.2}
Recall that $y(x)=y(x,M^0,M^1)$ is defined by \eqref{4.0}. 
For given $({M}^0_-,{M}^1_-)\in\mathcal{M}(\theta_0,\theta_1,-\theta_{\infty})$  
let $y_{-\theta_{\infty}}(x)=y_{-\theta_{\infty}}(x,{M}^0_-,{M}^1_-)$ 
be a solution of 
$(\mathrm{P}_{\mathrm{V}})_{(\theta_0,\theta_1,-\theta_{\infty})}$.
By using some $\mathfrak{z}_{-\theta
_{\infty}}(x)$ and $\hat{u}_{-\theta_{\infty}}(x)$ this fact is expressed as
\begin{align*}
&\mathcal{E}(\lambda,\phi,y_{-\theta_{\infty}}(x),\mathfrak{z}_{-\theta_{\infty}}
(x), \hat{u}_{-\theta_{\infty}}(x),Y_2(\lambda,-\theta_{\infty}), 
-\theta_{\infty})=0,
\\
&Y_2(\lambda,-\theta_{\infty})|_{ \hat{l}_{0,1}\ni \lambda}=Y_2(\lambda,
-\theta_{\infty}){M}_-^{0,\,1}
\end{align*}
for $|\phi|<\pi/2,$ $|\arg\lambda-\pi/2|<\pi.$ Replacement $(\lambda,\phi,x)
\mapsto (e^{-\pi i}\lambda ,\phi-\pi, e^{-\pi i}x)$ yields
\begin{align*}
&\mathcal{E}(e^{-\pi i}\lambda,\phi-\pi,y_{-\theta_{\infty}}(e^{-\pi i}x),
\mathfrak{z}_{-\theta_{\infty}}(e^{-\pi i}x), \hat{u}_{-\theta_{\infty}}
(e^{-\pi i}x),Y_2(e^{-\pi i}\lambda,-\theta_{\infty}), -\theta_{\infty})
\\
=&\mathcal{E}(\lambda,\phi,y_{-\theta_{\infty}}(e^{-\pi i}x),
\mathfrak{z}_{-\theta_{\infty}}(e^{-\pi i}x), \hat{u}_{-\theta_{\infty}}
(e^{-\pi i}x),Y_2(e^{-\pi i}\lambda,-\theta_{\infty}), -\theta_{\infty})=0,
\\
&Y_2(e^{-\pi i}\lambda,-\theta_{\infty})|_{
e^{-\pi i} \hat{l}_{0,1}\ni e^{-\pi i}\lambda}
=Y_2(e^{-\pi i}\lambda,-\theta_{\infty})|_{\hat{l}_{0,1}\ni \lambda}
=Y_2(e^{-\pi i}\lambda,-\theta_{\infty}){M}_-^{0,\,1}
\end{align*}
for $|\phi|<\pi/2,$ $|\arg\lambda-\pi/2|<\pi.$ Multiplication of 
both sides of $\mathcal{E}(\dots)$ by $\sigma_1$ combined with the substitution 
$(\mathfrak{z}_{-\theta_{\infty}}
(e^{-\pi i}x), \hat{u}_{-\theta_{\infty}}(e^{-\pi i}x)) =
(-\mathfrak{z}(x)-\theta_0, \hat{u}(x)^{-1})$ leads to
\begin{align}\notag
&\mathcal{E}(\lambda,\phi,y_{-\theta_{\infty}}(e^{-\pi i}x)^{-1},
\mathfrak{z}(x), \hat{u}(x),
Y_2(\lambda,\theta_{\infty}), \theta_{\infty})=0,
\\
\label{4.1}
&Y_2(\lambda,\theta_{\infty})|_{ \hat{l}_{0,1}\ni \lambda}
=Y_2(\lambda,\theta_{\infty})S_2e^{\frac{\pi i}2\theta_{\infty}\sigma_3}
\sigma_1{M}_-^{0,\,1}\sigma_1e^{-\frac{\pi i}2\theta_{\infty}\sigma_3}S_2^{-1},
\end{align}
because $\sigma_1Y_2(e^{-\pi i}\lambda,-\theta_{\infty})\sigma_1
=(I+O(\lambda^{-1}))
\exp(\tfrac 14(t\lambda-2\theta_{\infty}\ln\lambda)\sigma_3)e^{\frac{\pi i}2
\theta_{\infty}\sigma_3}$ for $|\arg\lambda -3\pi/2|<\pi$ and
$\sigma_1Y_2(e^{-\pi i}\lambda,-\theta_{\infty})\sigma_1
=Y_3(\lambda,\theta_{\infty})e^{\frac{\pi i}2\theta_{\infty}\sigma_3}
=Y_2(\lambda,\theta_{\infty})S_2e^{\frac{\pi i}2\theta_{\infty}\sigma_3}.$
Note that $1/y_{-\theta_{\infty}}(e^{-\pi i}x)$ solves $(\mathrm{P}_{\mathrm{V}}
)_{(\theta_0,\theta_1,\theta_{\infty})}$. For given $(M^0,M^1)\in \mathcal{M}
(\theta_0,\theta_1,\theta_{\infty})$ we may choose  
${M}_-^{0,\,1}$ in such a way that $1/y_{-\theta_{\infty}}
(e^{-\pi i}x,{M}_-^0,{M}_-^1)=y(x,M^0,M^1)$. Indeed, in \eqref{4.1}, setting
$M^{0,\,1}=S_2e^{\frac{\pi i}2\theta_{\infty}\sigma_3}\sigma_1{M}_-^{0,\,1}
\sigma_1e^{-\frac{\pi i}2\theta_{\infty}\sigma_3}S_2^{-1}$, 
that is,
$$
{M}_-^{0,\,1}=\sigma_1 e^{-\frac{\pi i}2\theta_{\infty}\sigma_3}S_2^{-1}
M^{0,\,1}S_2e^{\frac{\pi i}2\theta_{\infty}\sigma_3}\sigma_1=[\hat{M}^{0,\,1}]
_{\sigma_1},
$$
and comparing with \eqref{4.0},
we get $y(x,M^0,M^1)=1/y_{-\theta_{\infty}}(e^{-\pi i}x,
[\hat{M}^0]_{\sigma_1},[\hat{M}^1]_{\sigma_1} )$, which is the desired
formula.  
The second formula of Theorem \ref{thm2.15} is obtained by using
$(\lambda,\phi)\mapsto (e^{\pi i}\lambda,\phi+\pi).$
\section{Another kind of truncated solutions}\label{sc5}
The truncated solutions of Theorems \ref{thm2.6} and \ref{thm2.8} are derived 
from asymptotic solutions along the imaginary axis \cite{Andreev-Kitaev-2019,
S-2018}. Recall the isomonodromy linear system \eqref{3.1} 
\begin{equation*}
\frac{d\Psi}{d\lambda}=\Bigl(\frac{\sigma_3}{2}+\frac{A_0}{\lambda}
+\frac{A_1}{\lambda-x}\Bigr)\Psi, \quad A_{0,\,1}=\mathcal{A}_{0,\,1}|
_{u\mapsto  x^{-\theta_{\infty}}u},
\end{equation*}
which are related to system \eqref{1.1} via $\lambda=x\xi,$ $\Xi=x^{\frac12
{\theta_{\infty}}\sigma_3}\Psi$, and defines the same monodromy data $(M^0,M^1,
S_1,S_2)$ as for \eqref{2.1} and \eqref{1.1}.   
System \eqref{3.1} has the isomonodromy property if and only if
$(A_0,A_1)=(A_0(x),A_1(x))=(A_0,A_1)(x,y(x),\mathfrak{z}(x),x^{-\theta_{\infty}}
u(x))$ solves the Schlesinger-type equation    
\begin{equation}\label{5.00}
x\frac{dA_0}{dx}=[A_1, A_0],\quad x\frac{dA_1}{dx}=[A_0,A_1]+\frac x2[\sigma_3,
A_1],
\end{equation}
which is the compatibility condition of the Lax pair for \eqref{3.1}.
\subsection{General solution along the imaginary axis}\label{ssc5.1}
Equation \eqref{5.00} admits a family of solutions 
$\{(A_0(\mathbf{c},\sigma,x), A_1(\mathbf{c},\sigma,x))\,|\, (\mathbf{c},\sigma)
=(c_0,c_x,\sigma)\in B_0\times B_x \times B_* \}$ \cite[Theorem 2.1]{S-2018}, 
$B_0,$ $B_x \subset \mathbb{C}\setminus \{0\}$ and
$B_*\subset \mathbb{C}$ being given bounded domains.   
To rewrite $(A_0,A_1)$ suitably for our purpose, set
\begin{align}\label{5.b}
\begin{split}
&\rho_-:=\tfrac 14(\sigma+2\theta_0-\theta_{\infty}), \quad
\eta_-:=\tfrac 14(\sigma+2\theta_1+\theta_{\infty}), \quad
\\
&\rho_+:=\tfrac 14(-\sigma+2\theta_1-\theta_{\infty}), \quad
\eta_+:=\tfrac 14(-\sigma+2\theta_0+\theta_{\infty}), \quad
\end{split}
\\
\notag
&c:=c_x/c_0, \quad c^-:=c^{-1}=c_0/c_x,\quad  z_+=e^x x^{\sigma-1}, 
\quad z_-=e^{-x}x^{-\sigma-1}.
\end{align} 
The quantities $\gamma_{\pm}^0$ and $\gamma_{\pm}^x$ in \cite{S-2018} are 
written as
\begin{align*}
&\gamma_+^0:=\rho_-c_0=\tfrac 14 c_0(\sigma+2\theta_0-\theta_{\infty}), \quad
\gamma_-^0:=\eta_+c_0^{-1}= \tfrac 14c_0^{-1}
(-\sigma+2\theta_0+\theta_{\infty}), \quad
\\
&\gamma_+^x:=\rho_+c_x=\tfrac 14c_x(-\sigma+2\theta_1-\theta_{\infty}), \quad
\gamma_-^x:=\eta_-c_x^{-1}=\tfrac 14 c_x^{-1}(\sigma+2\theta_1+\theta_{\infty}),
\end{align*}
and
$\gamma_+^0\gamma_-^x= \rho_-\eta_-c^-,$ $\gamma_-^0 \gamma_+^x=\rho_+\eta_+c,$ 
$\gamma_+^0\gamma_-^0=\rho_-\eta_+,$ $ \gamma_+^x\gamma_-^x=\rho_+\eta_-.$
\par
Let
\begin{align*}
&A_0(\mathbf{c},\sigma,x)=f_0 \sigma_3+ f_+ \sigma_{12} +f_- \sigma_{21}, \quad
A_1(\mathbf{c},\sigma,x)=g_0 \sigma_3+ g_+ \sigma_{12} +g_- \sigma_{21}, \quad
\\
& \sigma_{12}=\tfrac 12(\sigma_1+i\sigma_2)= \begin{pmatrix} 0 & 1 \\
0 & 0 \end{pmatrix}, \quad \sigma_{21}=\tfrac 12(\sigma_1-i\sigma_2)
=\begin{pmatrix} 0 & 0 \\ 1 & 0 \end{pmatrix}.
\end{align*}
Then $f_0,$ $f_{\pm},$ $g_0,$ $g_{\pm}$ are given in \cite[Theorem 2.1]{S-2018},
and are rewritten in terms of $\rho_{\pm}$, $\eta_{\pm}$ as follows:
\begin{align*}
f_0+\tfrac 12 {\theta_0}=&
\rho_-+ \omega_0+\rho_+\eta_+cz_+[h_+]+\rho_-\eta_-c^{-}z_-[h_-],
\\
g_0+\tfrac 12 {\theta_1}=&
\rho_+- \omega_0-\rho_+\eta_+cz_+[h_+]-\rho_-\eta_-c^{-}z_-[h_-],
\\
\omega_0=&[x^{-2}]=-\tfrac 12((\sigma+\theta_{\infty})\rho_-\eta_+
+(\sigma-\theta_{\infty})\rho_+\eta_-)x^{-2}+[x^{-3}],
\\
c_0^{-1}x^{(\sigma+\theta_{\infty})/2}f_+=& \mathcal{F}_+:=
\rho_-(1+[x^{-1}]+2\rho_-\eta_-c^-z_-x^{-1}[h_-])
\\
&\phantom{----}
-\rho_+cz_+(\tfrac 12 (\sigma-\theta_{\infty})+[x^{-1}]+\rho_+\eta_+c
z_+[h_+]),
\\
c_x^{-1}e^{-x}x^{-(\sigma-\theta_{\infty})/2} g_+=&\mathcal{G}_+:=
\rho_+(1+[x^{-1}]+2\rho_+\eta_+cz_+x^{-1}[h_+])
\\
&\phantom{----}
-\rho_-c^{-}z_-(\tfrac 12 (\sigma+\theta_{\infty})+[x^{-1}]+\rho_-\eta_-c^{-}
z_-[h_-]),
\\
c_0x^{-(\sigma+\theta_{\infty})/2}f_-=&\mathcal{F}_-:=
\eta_+(1+[x^{-1}]+2\rho_+\eta_+cz_+x^{-1}[h_+])
\\
&\phantom{----}
-\eta_-c^-z_-(\tfrac 12 (\sigma-\theta_{\infty})+[x^{-1}]+\rho_-\eta_-c^-
z_-[h_-]),
\\
c_xe^{x}x^{(\sigma-\theta_{\infty})/2} g_-=&\mathcal{G}_-:=
\eta_-(1+[x^{-1}]+2\rho_-\eta_-c^-z_-x^{-1}[h_-])
\\
&\phantom{----}
-\eta_+cz_+(\tfrac 12 (\sigma+\theta_{\infty})+[x^{-1}]+\rho_+\eta_+cz_+[h_+])
\end{align*}
as $x\to \infty$ through the sector 
$|\arg x-\pi/2|<\pi/2$ with 
$|c^{\pm}z_{\pm}|=|c^{\pm}e^{\pm x}x^{\pm\sigma-1}|<r_0$. 
Here 
(1) $c_0,$ $c_x$, $c^+=c=c_x/c_0,$ $c^-=c_0/c_x$ and $\sigma$ are integration 
constants, and $r_0$ is a sufficiently small number; 
\par
(2) for $p\in \mathbb{Z}$, $[x^{-p}]$ denotes
a holomorphic function $\psi(x)$ admitting an asymptotic representation
$\psi(x)\sim x^{-{p}_0}\sum_{j\ge 0}\psi_j x^{-j}$ in $|\arg x-\pi/2|<\pi/2$, 
with some ${p}_0\ge p$ and  
$$
\psi_j\in \mathbb{Q}_*:=\mathbb{Q}[\theta_0,\theta_1,\theta_{\infty},\sigma,
\,\rho_-\eta_- c^-, \,\rho_+ \eta_+c, \,\rho_-\eta_+, \,\rho_+\eta_-],
$$
and $[x^{-p}]$ does not necessarily express the same function 
in each appearance; 
\par
(3) $[h_{\pm}]$ denote series of the forms 
$1+[x^{-1}]+\sum_{n=1}^{\infty}[x^{-n}](\rho_{\pm}\eta_{\pm}
c^{\pm}z_{\pm})^n$ convergent for $x$ such that $|\arg x-\pi/2|<\pi/2$ with  
$|c^{\pm}z_{\pm}|<r_0$, and does not necessarily express the same series 
in each appearance.
\begin{rem}\label{rem5.1}
In the expressions $f_0,$ $g_0$, $\mathcal{F}_{\pm}$, $\mathcal{G}_{\pm}$ above 
the coefficients $\psi_j$ of $[x^{-p}]\sim x^{-p}\sum_{j=0}
^{\infty}\psi_j x^{-j}$ are recursively determined by 
\cite[(4.6) and (4.7)]{S-2018}. Careful observation of these
relations leads to the fact that $\psi_j$'s belong to the ring of polynomials 
$\mathbb{Q}_*$ defined above, which are more precise expression than that of 
$\mathbb{Q}_*$ in \cite[Section 1]{S-2018}. 
\end{rem}
\begin{rem}\label{rem5.0}
Equation \eqref{5.00} also admits,
in the sector $|\arg x+\pi/2|<\pi/2,$ a solution 
$(A^{(-1)}_0(\mathbf{c},\sigma,x),A^{(-1)}_x(\mathbf{c},\sigma,x))$, 
whose entries admit the same 
asymptotic expressions as of $f_0,$ $g_0,$ $f_{\pm},$ $g_{\pm}$ 
\cite[Remark 2.6]{S-2018}. In what follows such entries in $|\arg x+\pi/2|<
\pi/2$ are also written in terms of the same symbols $f_0,$ $g_0,$ $f_{\pm},$
$g_{\pm}$, $\mathcal{F}_{\pm},$ $\mathcal{G}_{\pm}.$ 
\end{rem}
Then we have the following proposition, in which related monodromy data
are due to \cite[Corollary 2.13, Remark 2.15 and Section 2.3]{S-2018}. 
The simplified process of deriving monodromy data is described in Section 
\ref{sc7}, in which no conditions are imposed on $\theta_0$, $\theta_1$,
$\theta_{\infty}$ and $\sigma$ (cf. Remark \ref{rem7.1}).
\begin{prop}\label{prop5.1}
The ratio of holomorphic functions
\begin{align}\label{5.1}
y(c,\sigma,x)=&\frac{(f_0+\frac 12\theta_0)g_+}{(g_0+\frac12\theta_1)f_+}
=ce^x x^{\sigma} \frac{(f_0+\frac 12\theta_0)\mathcal{G}_+}{(g_0+\frac12\theta_1)
\mathcal{F}_+}
\\
\label{5.2}
=&\frac{(g_0-\frac 12\theta_1)f_-}{(f_0-\frac12\theta_0)g_-}
=ce^x x^{\sigma}\frac{(g_0-\frac 12\theta_1)\mathcal{F}_-}
{(f_0-\frac12\theta_0)\mathcal{G}_-}
\end{align}
is a general solution of $(\mathrm{P}_{\mathrm{V}})$ as $x\to\infty$ avoiding
the poles, i.e. the zeros of the denominator, through the
sector $|\arg x-\pi/2|<\pi/2$ (respectively, $|\arg x+\pi/2|<\pi/2$) with
$|c^{\pm}z_{\pm}|<r_0$. The monodromy data $(M^0,M^1)$ corresponding to 
$y(c,\sigma,x)$ are such that
\begin{align*}
&m^0_{21}=\frac{2\pi ie^{-\pi i\theta_{\infty}}c_0^{-1}}
{\Gamma(1-\frac {\sigma+2\theta_0-\theta_{\infty}}4)\Gamma(-\frac 
{\sigma-2\theta_0-\theta_{\infty}}4)}, \quad
&m^1_{12}=\frac{2\pi ic_x}
{\Gamma(1-\frac {\sigma+2\theta_1+\theta_{\infty}}4)\Gamma(-\frac 
{\sigma-2\theta_1+\theta_{\infty}}4)}, 
\end{align*}
and that
$m^1_{11}=e^{-\frac{\pi i}2(\sigma+\theta_{\infty})}$ 
(respectively, $m^0_{11}=e^{\frac{\pi i}2(\sigma-\theta_{\infty})}$). 
\end{prop}
Suppose that $\rho_+\rho_-\not=0$. If $\re \sigma <0$, then we have, in the
strip $|\re x|\ll 1$,
\begin{align*}
y=&ce^xx^{\sigma}\frac{\mathcal{G}_+}{g_0+\frac12\theta_1}\cdot 
\frac{f_0+\frac12\theta_0}{\mathcal{F}_+}
\\
=&ce^xx^{\sigma}
\Bigl(1+O(x^{-1})+\rho_-c^-e^{-x}x^{-\sigma-1}(1+O(x^{-\sigma-1}))\Bigr)
\\
&\phantom{---}\times
\Bigl(1+O(x^{-1})+\eta_-c^-e^{-x}x^{-\sigma-1}(1+O(x^{-\sigma-1}))\Bigr)
\end{align*}
and
\begin{equation*}
x^{-\theta_{\infty}} u=-\frac{f_+}{f_0+\frac12\theta_0}
=\frac{-c_0x^{-(\sigma+\theta_{\infty})/2}(1+O(x^{-1}))}
{1+O(x^{-1})+\eta_-c^-e^{-x}x^{-\sigma-1}(1+O(x^{-\sigma-1}))}
\end{equation*}
as $x \to \infty$ through the sector $|\arg x-\pi/2|<\pi/2$ (respectively,
$|\arg x+\pi/2|<\pi/2$), where the multiplier $x^{-\theta_{\infty}}$ is due to
Proposition \ref{prop3.0}. 
\begin{rem}\label{rem5.22}
By WKB analysis for system \eqref{1.1} with the matching procedure 
asymptotic solutions of (P$_{\mathrm{V}}$) along imaginary axes with 
the monodromy data are  
obtained by \cite[Theorems 1, 2 and 3]{Andreev-Kitaev-2019}. As $t\to \infty$
through the strip $|\re t|\ll 1,$ the asymptotic solutions are given by
\begin{align*}
yt=&\delta t^{\nu_1}e^t \Bigl(1-\frac{\phi-\frac{\theta_0}2}{\delta t^{\nu_1}e^t}
\Bigr)\Bigl(1-\frac{\phi-\frac{\theta_1+\theta_{\infty}}2}{\delta t^{\nu_1}e^t}
\Bigr)+O(t^{-3\nu_1-1}\ln t)+O(t^{\nu_1-1}\ln t) 
\\
u_{\mathrm{AK}}&=\frac{\hat{u}}{\delta}\cdot \frac{t^{2\phi}(
1+O(t^{-1}\ln t)+O(t^{-2\nu_1-1}\ln t))}
{1-\frac{\phi-\frac{\theta_1+\theta
_{\infty}}2}{\delta t^{\nu_1}e^t} +O(t^{-\nu_1-1}\ln t)+O(t^{-3\nu_1-1}\ln t)}
\end{align*}
with $\nu_1=1+\theta_{\infty}-4\phi$ such that $-1/2 <\re \nu_1<1$, 
where $\phi$, $\delta$ and $\hat{u}$ are integration constants. 
The corresponding monodromy data $(M^0,M^1)$ are such that
$$
m^1_{12}=\frac{-2\pi i\hat{u}}{\Gamma(1-\frac{\theta_1+\theta_{\infty}-2\phi}2)
\Gamma(\frac{\theta_1-\theta_{\infty}+2\phi}2)}, \quad
m^0_{21}=\frac{-2\pi i \delta e^{-\pi i\theta_{\infty}}}
{ \hat{u}\Gamma(1-\frac{\theta_0-2\phi}2)
\Gamma(\frac{\theta_0+2\phi}2)}, 
$$
and $m^1_{11}=e^{2\pi i\phi-\pi i\theta_{\infty}}$ if $\arg t\to \pi/2,$
$m^0_{11}=e^{-2\pi i\phi}$ if $\arg t\to -\pi/2.$ 
By setting $2\phi=\frac 12(\theta_{\infty}-\sigma)$, $\hat{u}=-c_x,$
$\hat{u}/\delta=-c_0,$ $u_{\mathrm{AK}}=u$ (cf. Proposition \ref{prop3.0}), 
the leading terms of $ty$ and $u_{\mathrm{AK}}$, and the 
monodromy data agree with ours, since $|e^xx^{\sigma-1}|=|e^x x^{\nu_1-2}|
\ll x^{-1}$ if $|\re x|\ll 1.$
\end{rem}
\subsection{Proof of Theorems \ref{thm2.6} and \ref{thm2.8}}\label{ssc5.2}
The substitution $\mathbf{s}$: $\sigma \mapsto -2\theta_1-\theta_{\infty}$  
gives rise to the replacement 
$(\eta_-,\eta_+,\rho_-,\rho_+) \mapsto (0,\eta_+^0,\rho_-^0,\theta_1)$ with
$\eta_+^0:=\tfrac 12(\theta_0+\theta_1+\theta_{\infty})$, 
$ \rho_-^0:=\tfrac 12(\theta_0-\theta_1-\theta_{\infty})$.
Recall $\omega_0$ in the expressions of $f_0$ or $g_0$. 
Let $\omega_0^*$ denote the result of the substitution $\mathbf{s}$ to 
$\omega_0.$  
\begin{lem}\label{lem5.2}
We have $\omega_0^*=\rho_-^0\eta_+^0
(-\frac12(\sigma+\theta_{\infty})x^2+[x^{-3}]).$\footnote[1]{In fact
$\omega_0=\rho_-\eta_+(-\frac12(\sigma+\theta_{\infty})x^2+[x^{-3}])
+\rho_+\eta_-(-\frac12(\sigma-\theta_{\infty})x^2+[x^{-3}])$. This follows
from \cite[Proposition 5.4]{S-2018} combined with 
$\omega_0=\gamma^0_+\gamma^0_-[1]+\gamma^x_+\gamma^x_-[1]$,
which may be additionally shown in \cite[Proposition 5.5]{S-2018} with
respect to $p_0(x)=\omega_0$ for $\varphi^j$.}
\end{lem}
\begin{proof}
The asymptotic series part free from $e^{\pm x}$ factors of $\theta_{\infty}\z
+\tfrac 14((\theta_0+\theta_{\infty})^2-\theta_1^2)=f_+f_--g_+g_-$
is written in the form $\rho_-\eta_+(1+[x^{-1}])-\rho_+\eta_-(1+[x^{-1}])
=\tfrac 14(\theta_0^2-\theta_1^2+\sigma\theta_{\infty})+\rho_-\eta_+[x^{-1}]
-\rho_+\eta_-[x^{-1}].$ The asymptotic series part of $\z=f_0-\frac12\theta_0 $ 
is also given by $-\theta_0+\rho_-+\omega_0$, and we have
the identity $\theta_{\infty}\omega_0=\rho_-\eta_+[x^{-1}]
-\rho_+\eta_-[x^{-1}].$ The application of $\mathbf{s}$ yields
$\theta_{\infty}\omega_0^*=\rho_-^0\eta_+^0[x^{-1}].$ 
The coefficients of each 
$x^{-n}$ on the right-hand side is given by $\rho_-^0\eta_+^0P_n
(\theta_0,\theta_1,\theta_{\infty})$ with some polynomial $P_n$.
Write $(\theta_0,\theta_1)=(\alpha+\beta,\alpha-\beta-\theta_{\infty})$
with $(\alpha,\beta)=(\eta_+^0,\rho_-^0)$. Then $\alpha\beta P_n(\alpha+\beta,
\alpha-\beta-\theta_{\infty},\theta_{\infty})$ vanishes at $\theta_{\infty}=0,$
which implies every coefficient of $x^{-n}$ has the factor $\theta_{\infty}$,
and the error bound on the right-hand side is written in the form 
$\rho_-^0\eta_+^0\cdot O(\theta_{\infty}x^{-n-1})$. Thus the lemma is verified.
\end{proof}
In \eqref{5.1}, set $\eta_-=\varepsilon$ and $\sigma=4\varepsilon-2\theta_1
-\theta_{\infty}$, and replace $c_x$ with $c_x\varepsilon$. 
Then $\eta_-c^-$ is replaced with $c^-$, and we have  
$\rho_-=\rho_-^0 +O(\varepsilon)$, 
$\rho_+= \theta_1+O(\varepsilon)$ and $\eta_+=\eta_+^0+O(\varepsilon)$ 
with $\rho_-^0$ and $\eta_+^0$ as above.
Suppose that $\theta_1 \not=0,$ and regard $\rho_-^0$ as a parameter. 
Then the quantity 
\begin{equation*}
ce^x x^{\sigma}\mathcal{G}_+ =
-\rho_-x^{-1}\bigl(\tfrac 12 (\sigma+\theta_{\infty})+[x^{-1}]+\rho_-\eta_-c^{-}
z_-[h_-]\bigr) +ce^xx^{\sigma}\rho_+(1+\cdots)
\end{equation*}
is replaced with
\begin{align*}
&-\rho_-^0 x^{-1}\Bigl(-\theta_1+[x^{-1}]+\rho_-^0 c^- z_-
\Bigl(1+[x^{-1}]+\sum_{n=1}^{\infty}[x^{-n}](\rho_-^0c^-z_-)^n\Bigr)\Bigr)
+O(\varepsilon)
\\
=&\rho_-^0\theta_1 x^{-1}\Bigl(1+[x^{-1}]-\rho_-^0\theta_1^{-1}c^-z_-
\Bigl(1+[x^{-1}]+\sum_{n=1}^{\infty}[x^{-n}](\rho_-^0c^-z_-)^n\Bigr)\Bigr)
+O(\varepsilon),
\end{align*}
and, by Lemma \ref{lem5.2},
\begin{align*}
\frac{f_0+\frac12 \theta_0}{(g_0+\frac 12 \theta_1)\mathcal{F}_+}
=&\frac{\rho_-+\omega_0+\rho_-\eta_-c^-z_-[h_-]+c\,(\cdots)}
{\rho_+ -\omega_0-\rho_-\eta_-c^-z_-[h_-]+c\,(\cdots)}
\\
&\times
\Bigl(\rho_-(1+[x^{-1}]+2\rho_-\eta_-c^-x^{-1}z_-[h_-])
-\rho_+cz_+(\cdots)\Bigr)^{-1}
\end{align*}
becomes
\begin{equation*}
\frac{\rho_-^0\left(1+[x^{-1}]+c^-z_-(1+[x^{-1}]+\sum_{n=1}^{\infty}
[x^{-n}](\rho_-^0c^-z_-)^n)\right)+O(\varepsilon)}
{\rho_-^0\theta_1\left(1+[x^{-1}] -\rho_-^0\theta_1^{-1}c^-z_-(1+[x^{-1}]
+\sum_{n=1}^{\infty}[x^{-n}](\rho_-^0c^-z_-)^n)\right)+O(\varepsilon)},
\end{equation*}
where the coefficients of $[x^{-1}]$ and $[1]$ are in 
$\mathbb{Q}_*|_{\sigma=-2\theta_1-\theta_{\infty}}[\rho_-^0,\theta_1^{-1}].$ 
The reduction on $\rho_-^0\theta_1$ after putting $\varepsilon=0$ yields
\begin{align*}
\frac{(f_0+\frac12\theta_0)g_+}{(g_0+\frac12\theta_1)f_+}\biggl|_{\varepsilon=0}
=&\rho_-^0\theta_1 x^{-1}\Bigl(1+[x^{-1}]-\rho_-^0\theta_1^{-1}c^-z_-
\Bigl(1+[x^{-1}]+\sum_{n=1}^{\infty}[x^{-n}](\rho_-^0c^-z_-)^n\Bigr)\Bigr)
\\
&\times
\frac{\rho_-^0\left(1+[x^{-1}]+c^-z_-(1+[x^{-1}]+\sum_{n=1}^{\infty}
[x^{-n}](\rho_-^0c^-z_-)^n)\right)}
{\rho_-^0\theta_1\left(1+[x^{-1}] -\rho_-^0\theta_1^{-1}c^-z_-(1+[x^{-1}]
+\sum_{n=1}^{\infty}[x^{-n}](\rho_-^0c^-z_-)^n)\right)}
\\
=&\rho_-^0 x^{-1}\Bigl(1+[x^{-1}]-\rho_-^0\theta_1^{-1}c^-z_-
[x^{-1}]+\sum_{n=2}^{\infty}[1](\rho_-^0c^-z_-)^n\Bigr)
\\
&\times
\Bigl(1+[x^{-1}]+c^-z_-(1+[x^{-1}])
+\sum_{n=2}^{\infty}[1](\rho_-^0)^{n-1}(c^-z_-)^n\Bigr),
\end{align*}
which leads to
$$
xy=\rho_-^0\Bigl(1+[x^{-1}]+c^-e^{-x} 
x^{2\theta_1+\theta_{\infty}-1}\Bigl(1+[x^{-1}]+\sum_{n=1}^{\infty}
[1](\rho_-^0)^{n-1}(c^-e^{-x}x^{2\theta_1+\theta_{\infty}-1})^n\Bigr)\Bigr).
$$
The replacement $\rho^0_-c^- \mapsto -c^-$ or $\rho_-^0c_0\mapsto -c_0$, 
$\rho_-^0=\frac12(\theta_0-\theta_1
-\theta_{\infty})$ yields a solution as in Theorem \ref{thm2.6} (1).
The extension of the sector to $-\pi/2<\arg x\le \pi/2$
is possible by the uniqueness of the truncated solution with the same 
doubly-truncated part and small exponential part, 
which is verified by classical analysis (see Subsection \ref{ssc6.2} with
Proposition \ref{prop6.2}).
The monodromy data become
\begin{align*}
& m^1_{12}=\frac{2\pi i \varepsilon}{\Gamma(1-\varepsilon)\Gamma(\theta_1
-\varepsilon)} \to 0, 
\\
& m^0_{21}\to \frac{2\pi i e^{-\pi i\theta_{\infty}}c_0^{-1}}
{\Gamma(1-\frac{\theta_0-\theta_1-\theta_{\infty}}2)
\Gamma(\frac{\theta_0+\theta_1+\theta_{\infty}}2)}\Bigl|_{\rho_-^0c_0\mapsto
- c_0}= \frac{2\pi i e^{-\pi i\theta_{\infty}}c_0^{-1}}
{\Gamma(-\frac{\theta_0-\theta_1-\theta_{\infty}}2)
\Gamma(\frac{\theta_0+\theta_1+\theta_{\infty}}2)}
\end{align*}
as long as $\tfrac 12(\theta_0+\theta_1+\theta_{\infty}) \not
\in -\mathbb{N}\cup\{0\},$ $\frac 12(\theta_0-\theta_1-\theta_{\infty})
\not\in \mathbb{N}\cup\{0\}.$ Since $m^1_{11}  
=e^{\pi i(\theta_1-2\varepsilon)}$ as $\phi \to \pi/2,$ we have
\begin{align*}
m^1_{21}=\frac{m_{11}^1m_{22}^1-1}{m^1_{12}}
&=\Bigl(e^{\pi i(\theta_1-2\varepsilon)} (2\cos \pi\theta_1
-e^{\pi i(\theta_1-2\varepsilon)})-1\Bigr)\frac{\Gamma(1-\varepsilon)\Gamma
(\theta_1-\varepsilon)}{2\pi i \varepsilon}c_x^{-1}
\\
&\to \frac{\Gamma(\theta_1)}{2\pi i}(-2\pi i e^{\pi i\theta_1}
\cdot 2\cos\pi\theta_1+4\pi i e^{2\pi i\theta_1})c_x^{-1}
=\frac{2\pi i e^{\pi i\theta_1}}{\Gamma(1-\theta_1)}c_x^{-1}.
\end{align*} 
Thus we have the asymptotic form of the truncated solution with 
the monodromy data as in Theorem \ref{thm2.6}, (1).
\par
In \eqref{5.1}, set $\rho_+=\varepsilon,$ $\sigma=2\theta_1-\theta_{\infty}
-4\varepsilon$. Then $\eta_-= \theta_1+O(\varepsilon),$ 
$\eta_+ = \eta_+^0 +O(\varepsilon)$ and $\rho_-= \rho_-^0+O(\varepsilon)$,
where $\eta_+^0:=\frac 12(\theta_0-\theta_1+\theta_{\infty})$ and $\eta_-^0
:=\frac 12(\theta_0+\theta_1-\theta_{\infty})$.  
Under the supposition $\theta_1\not=0,$ observing that $\omega_0=\rho_-\eta_+
(-\frac12(\sigma+\theta_{\infty})x^{-2}+[x^{-3}])=-\rho_-^0\eta_+^0
(\theta_1x^{-2}+[x^{-3}])+O(\varepsilon)$ by Lemma \ref{lem5.2} with
$\sigma\mapsto 2\theta_1-\theta_{\infty}$, we have
\begin{align*}
\frac{g_+}{f_+}\biggl|_{\varepsilon=0}=& c e^xx^{\sigma}\frac{-c^{-}z_-
\left(\theta_1+[x^{-1}]+\theta_1\rho_-^0c^-z_-(1+[x^{-1}]+\sum_{n=1}^{\infty}
[1](c^-z_-)^n)\right)}{1+[x^{-1}]+2\theta_1\rho_-^0c^-z_-x^{-1}
(1+[x^{-1}]+\sum_{n=1}^{\infty}[1](c^-z_-)^n)}
\\
=&  \frac{-x^{-1}\theta_1 \left(1+[x^{-1}]+\rho_-^0c^-z_-(1+[x^{-1}]
+\sum_{n=1}^{\infty}[1](c^-z_-)^n)\right)}
{1+[x^{-1}]+2\theta_1\rho_-^0c^-z_-x^{-1}
(1+[x^{-1}]+\sum_{n=1}^{\infty}[1](c^-z_-)^n)},
\end{align*}
and
\begin{equation*}
\frac{f_0+\frac12\theta_0}{g_0+\frac12\theta_1}\biggl|_{\varepsilon=0} =
\frac{\rho_-^0\left(1+[x^{-2}]+\theta_1c^-z_-(1+[x^{-1}]+\sum_{n=1}^{\infty}
[x^{-n}](c_-z_-)^n)\right)}
{x^{-2}\theta_1\rho_-^0\left(\eta_+^0(1+[x^{-1}])-c^-z_-x^2(1+[x^{-1}]
+\sum_{n=1}^{\infty}[x^{-n}](c^-z_-)^n)\right)},
\end{equation*}
from which it follows that
$$
\frac{x}{y}=-\frac{\theta_0-\theta_1-\theta_{\infty}}2(1+[x^{-1}])
+c^-x^{1-2\theta_1+\theta_{\infty}}e^{-x}\Bigl(1+[x^{-1}]+\sum_{n=1}^{\infty}[1]
(c^-x^{-2\theta_1+\theta_{\infty}-1}e^{-x})^n\Bigr)
$$
as in Theorem \ref{thm2.8},(1).
\par
Using \eqref{5.2} with $\rho_-=\varepsilon,$ $\sigma=\theta_{\infty}
-2\theta_0+4\varepsilon$ and $c_0/\varepsilon$ in place of $c_0$, 
and the replacement $-\eta_-^0c^- \mapsto c^-$ or $-c_x/\eta_-^0 \mapsto c_x$,
$\eta_-^0=-\frac12(\theta_0-\theta_1-\theta_{\infty})$
(respectively, with $\eta_+
=\varepsilon,$ $\sigma=2\theta_0+\theta_{\infty}-4\varepsilon$), we obtain 
truncated solutions as in Theorem \ref{thm2.6},(2) 
(respectively, Theorem \ref{thm2.8},(2)). Thus 
Theorems \ref{thm2.6} and \ref{thm2.8} are proved. 
\begin{rem}\label{rem5.3}
It is also possible to derive truncated solutions in the sector $\pi/2 \le
\arg x<3\pi/2$ or $-3\pi/2 <\arg x\le -\pi/2.$ Indeed, say, from \eqref{5.1},
setting $\eta_-=\varepsilon,$ $\sigma=-2\theta_1-\theta_{\infty}+4\varepsilon,$
we have
\begin{equation}\label{5.3}
xy=\frac 12(\theta_0-\theta_1-\theta_{\infty})(1+O(x^{-1})) +c x^{1-2\theta_1
-\theta_{\infty}} e^x(1+O(x^{-r}))
\end{equation}
as $x\to \infty$ through the sector $\pi/2 \le \arg x <3\pi/2$ with
$cx^{-1-2\theta_1-\theta_{\infty}}e^x \ll x^{-r},$ and the corresponding
monodromy data 
$$
M^0=\begin{pmatrix} 
m^0_{11} & m^0_{12} \\
\frac{2\pi ie^{-\pi i\theta_{\infty}}}{\Gamma(1-\frac 12(\theta_0-\theta_1
-\theta_{\infty}))\Gamma(\frac 12 (\theta_0+\theta_1+\theta_{\infty}))}c_0^{-1}
& m^0_{22} \end{pmatrix},
\quad
M^1=\begin{pmatrix} e^{\pi i\theta_1} & \frac{2\pi i}{\Gamma(\theta_1)}c_x
\\ 0 & e^{-\pi i\theta_1} \end{pmatrix}
$$
with $m^0_{11}= e^{-\pi i\theta_1}(e^{-\pi i\theta_{\infty}}-m_{21}^0m_{12}^1)$.
Setting $\eta_+=\varepsilon,$ $\sigma=2\theta_0+\theta_{\infty}
-4\varepsilon,$ and $c_0\varepsilon$ instead of $c_0$ in \eqref{5.1}, 
and making the replacement
$-\frac12(\theta_0-\theta_1+\theta_{\infty})c_x \mapsto  c_x$, 
we have
\begin{equation}\label{5.4}
\frac xy=\frac12(\theta_0-\theta_1+\theta_{\infty}) (1+O(x^{-1}))
+c x^{2\theta_0+\theta_{\infty}-1}e^x(1+O(x^{-r}))
\end{equation}
as $x\to \infty$ through the sector $-3\pi/2<\arg x \le -\pi/2$ with
$cx^{2\theta_0+\theta_{\infty}-1} e^x \ll x^{-r},$ and 
\begin{equation}\label{5.5}
M^0=\begin{pmatrix}
e^{\pi i\theta_0} & 0 \\ \frac{2\pi ie^{-\pi i\theta_{\infty}}}
{\Gamma(1-\theta_0)}c_0^{-1} & e^{-\pi i\theta_0}
\end{pmatrix},
\quad
M^1=\begin{pmatrix}  
m^1_{11}  & \frac{2\pi i }
{\Gamma(1-\frac 12(\theta_0+\theta_1+\theta_{\infty}))\Gamma(1-\frac 12
(\theta_0-\theta_1+\theta_{\infty}))} c_x \\
m^1_{21} & m^1_{22}
\end{pmatrix}
\end{equation}
with $m^1_{11}=e^{-\pi i\theta_0}(e^{-\pi i\theta_{\infty}}-m^0_{21}m^1_{12})$.
Furthermore, by $\rho_+=\varepsilon,$ $\sigma=2\theta_1-\theta_{\infty}
-4\varepsilon,$ $c_x \mapsto c_x/\varepsilon,$ and the replacement
$\frac12(\theta_0-\theta_1+\theta_{\infty})c_0^{-1} \mapsto c_0^{-1}$, we have
$$
\frac xy=-\frac12(\theta_0-\theta_1+\theta_{\infty}) (1+O(x^{-1}))
+c x^{2\theta_1-\theta_{\infty}-1}e^x(1+O(x^{-r}))
$$
as $x\to\infty$ through the sector $\pi/2 \le \arg x<3\pi/2$ with
$cx^{2\theta_1-\theta_{\infty}-1}e^x \ll x^{-r}$, and
$$
M^0=\begin{pmatrix} 
m^0_{11} & m^0_{12} \\
\frac{2\pi ie^{-\pi i\theta_{\infty}}}{\Gamma(1-\frac 12(\theta_0+\theta_1-
\theta_{\infty}))\Gamma(1+\frac 12(\theta_0-\theta_1+\theta_{\infty}))}c_0^{-1}
& m^0_{22}
\end{pmatrix},
\quad
M^1=\begin{pmatrix}
e^{-\pi i\theta_1}  & \frac{2\pi i}{\Gamma(1-\theta_1)} c_x \\
0 & e^{\pi i \theta_1}
\end{pmatrix} 
$$
with $m^0_{11}=e^{\pi i\theta_1}(e^{-\pi i\theta_{\infty}}-m^0_{21}m^1_{12});$
and by $\rho_-=\varepsilon,$ $\sigma=\theta_{\infty}-2\theta_0+4\varepsilon,$
$$
xy=-\frac 12(\theta_0-\theta_1-\theta_{\infty})(1+O(x^{-1}))+cx^{\theta_{\infty}
-2\theta_0+1}e^{x}(1+O(x^{-r}))
$$
as $x\to\infty$ through the sector $-3\pi/2<\arg x\le -\pi/2$ with $cx^{\theta
_{\infty}-2\theta_0-1}e^x \ll x^{-r}$, and 
$$
M^0=\begin{pmatrix}
e^{-\pi i\theta_0} & 0 \\ \frac{2\pi ie^{-\pi i\theta_{\infty}}}{\Gamma(\theta_0)
}c_0^{-1} & e^{\pi i\theta_0}
\end{pmatrix},
\quad
M^1=\begin{pmatrix} 
m^1_{11} & \frac{2\pi i}{\Gamma(1+\frac 12(\theta_0-\theta_1-\theta
_{\infty}))\Gamma(\frac 12(\theta_0+\theta_1-\theta_{\infty}))}c_x
\\
m^1_{21} & m^1_{22}
\end{pmatrix}
$$
with $m^1_{11}=e^{\pi i\theta_0}(e^{-\pi i\theta_{\infty}}-m^0_{21}m^1_{12})$.
\end{rem}
\subsection{Proof of Theorem \ref{thm2.9}}\label{ssc5.3}
For $M^0$ in Theorem \ref{thm2.6}, (1) of the generic case, we have
\begin{align*}
m^0_{12}&= (m^0_{11}m^0_{22}-1)(m^0_{21})^{-1}
\\
&=(e^{-\pi i(\theta_1+\theta_{\infty})}(2\cos\pi\theta_0-
e^{-\pi i(\theta_1+\theta_{\infty})} )-1)(m^0_{21})^{-1}
\\
&=(e^{\pi i(\theta_0-\theta_1-\theta_{\infty})}-1)(1-e^{-\pi i(\theta_0
+\theta_1 +\theta_{\infty})})(m^0_{21})^{-1}
\\
&=e^{\frac{\pi}2(\theta_0-\theta_1-\theta_{\infty})}2i\sin \tfrac{\pi
(\theta_0-\theta_1-\theta_{\infty})}2
e^{-\frac{\pi}2(\theta_0+\theta_1+\theta_{\infty})}2i\sin\tfrac{\pi
(\theta_0+\theta_1+\theta_{\infty})}2
\\
&\phantom{---} \times
\frac{e^{\pi i\theta_{\infty}}}{2\pi i}
\Gamma(-\tfrac{\theta_0-\theta_1-\theta_{\infty}}2)
\Gamma(\tfrac{\theta_0+\theta_1+\theta_{\infty}}2) c_0
\\
&=\frac{-2\pi ie^{-\pi i\theta_1} c_0}{\Gamma(1+\frac{\theta_0-\theta_1-
\theta_{\infty}}2)\Gamma(1-\frac{\theta_0+\theta_1+\theta_{\infty}}2)}.
\end{align*}
Hence, if $\theta_0-\theta_1-\theta_{\infty}=2(\nu-1),$ then
$m^0_{12}=\frac{-2\pi i e^{-\pi i\theta_1}}{\Gamma(\nu)\Gamma(\nu-\theta_0)}
c_0,$ and 
if $\theta_0+\theta_1+\theta_{\infty}=-2(\nu-1),$ then
$m^0_{12}=\frac{-2\pi i e^{-\pi i\theta_1}}{\Gamma(\nu)\Gamma(\nu+\theta_0)}
c_0,$ which implies the assertion (1) of Theorem \ref{thm2.9}.
For $M^1$ of Theorem \ref{thm2.6},(2), $M^0$ of Theorem \ref{thm2.8},(1)
and $M^1$ of (2), we have, respectively,
\begin{align*}
&m^1_{21}=\frac{-2\pi i e^{\pi i(\theta_0-\theta_{\infty})}}
{\Gamma(1-\frac{\theta_0+\theta_1-\theta_{\infty}}2)
\Gamma(1-\frac{\theta_0-\theta_1-\theta_{\infty}}2) }, 
\\
&m^0_{12}=\frac{2\pi i e^{\pi i\theta_1}}
{\Gamma(\frac{\theta_0+\theta_1-\theta_{\infty}}2)
\Gamma(1-\frac{\theta_0-\theta_1+\theta_{\infty}}2) }, \quad
m^1_{21}=\frac{2\pi i e^{-\pi i(\theta_0+\theta_{\infty})}}
{\Gamma(1+\frac{\theta_0-\theta_1+\theta_{\infty}}2)
\Gamma(\frac{\theta_0+\theta_1+\theta_{\infty}}2) }. 
\end{align*}
The remaining assertions of Theorem \ref{thm2.9} are verified by using these
relations.  
\subsection{$\z$ and $u$ for truncated solutions}\label{sc5.4}
For the trigonometric and related truncated solutions of Theorems \ref{thm2.3}
and \ref{thm2.4}, the functions $\z$ and $u$ for system \eqref{1.1} are known 
\cite[Theorems 3.1--3.4, 4.1--4.3]{Andreev-Kitaev}.
Let us calculate $\z$ and $u$ 
for the truncated solutions of Theorems \ref{thm2.6} and \ref{thm2.8}. 
Consider system \eqref{3.1} and Schlesinger-type equation \eqref{5.00}.
In the expression
$
 \z=f_0-\tfrac12{\theta_0}
=\rho_--\theta_0+[x^{-2}]
+\rho_+\eta_+cz_+[h_+]+\rho_-\eta_-c^-z_-[h_-]
$
with $z_+=e^xx^{\sigma-1}$ and $ z_-=e^{-x}x^{-\sigma-1}$,
set $\eta_-=\varepsilon$, $\sigma=4\varepsilon-2\theta_1-\theta_{\infty},$ 
$c_x\mapsto c_x\varepsilon$ with $\varepsilon \to 0$, and make the
replacement $c^-\rho^0_-
\mapsto -c^-$  
as in the proof of Theorem \ref{thm2.6},(1) of Subsection \ref{ssc5.2}. 
Then we have
\begin{align*}
\z= &-\tfrac12 (\theta_0+\theta_1+\theta_{\infty})(1+[x^{-2}])
\\
&\phantom{-}- c^-e^{-x}x^{2\theta_1
+\theta_{\infty}-1}\Bigl(1+[x^{-1}]+\sum_{n=1}^{\infty}[x^{-n}]
(c^-e^{-x}x^{2\theta_1+\theta_{\infty}-1})^n\Bigr).
\end{align*}
By the same procedure, from 
\begin{align*}
x^{-\theta_{\infty}}u=&-\frac{f_+}{f_0+\frac12\theta_0}
=-\frac{c_0x^{-(\sigma+\theta_{\infty})/2}\mathcal{F}_+}
{\rho_-+[x^{-2}]+\rho_+\eta_+cz_+[h_+]+\rho_-\eta_-c^-z_-[h_-]},
\\
\mathcal{F}_+=& 
\rho_-(1+[x^{-1}]+2\rho_-\eta_-c^-z_-x^{-1}[h_-])
\\
&\phantom{--} -\rho_+cz_+
(\tfrac12(\sigma-\theta_{\infty})+[x^{-1}]+\rho_+\eta_+cz_+[h_+]),
\end{align*}
we derive
$$
x^{-\theta_{\infty}}u=\frac{c_0}{\rho_-^0} x^{\theta_1}
\Bigl(1+[x^{-1}]+\frac{c^-}{\rho_-^0}e^{-x}x^{2\theta_1+\theta_{\infty}-1}
\Bigl(1+[x^{-1}]+\sum_{n=1}^{\infty}[1]
(c^-e^{-x}x^{2\theta_1+\theta_{\infty}-1})^n\Bigr)\Bigr)
$$
(for $x^{-\theta_{\infty}}$, cf.~Proposition \ref{prop3.0}).
By $\eta_+=\varepsilon,$ $\sigma=2\theta_0+\theta_{\infty}-4
\varepsilon,$ we derive $\z$ and $u$ for the solution of
Theorem \ref{thm2.8},(2). The expressions 
$x^{-\theta_{\infty}}u=(f_0-\frac12\theta_0)(f_-)^{-1}$,
$\z=f_0-\tfrac12\theta_0$
with $\rho_-=\varepsilon,$ $\sigma=\theta_{\infty}-2\theta_0+4\varepsilon,$
$c_0\mapsto c_0/\varepsilon$ and $-\eta_-^0 c^-\mapsto c^-$, 
and with $\rho_+=\varepsilon,$ $\sigma=2\theta_1
-\theta_{\infty}-4\varepsilon$ lead to 
functions related to 
Theorems \ref{thm2.6},(2) and \ref{thm2.8},(1), respectively.
Thus we have the following, in which $r$ is a given number such that 
$0<r\le 1$.
\begin{prop}\label{prop5.2}
$(1)$ For the solution $\mathrm{trunc}^0_0(x,m^1_{21}/m^0_{21},e^{\pi i\theta_1}
)$ of Theorem $\ref{thm2.6},(1),$
\begin{align*}
\z=& -\tfrac{\theta_0+\theta_1+\theta_{\infty}}2(1+O(x^{-2}))-
c^-e^{-x}x^{2\theta_1+\theta_{\infty}-1}(1+O(x^{-r})),
\\
x^{-\theta_{\infty}}u=&\tfrac{2}{\theta_0-\theta_1-\theta_{\infty}}c_0 
x^{\theta_1}\bigl(
1+O(x^{-1})+\tfrac{2}{\theta_0-\theta_1-\theta_{\infty}}c^-
e^{-x}x^{2\theta_1+\theta_{\infty}-1}(1+O(x^{-r}))\bigr)
\end{align*}
as $x\to \infty$ through $-\pi/2<\arg x\le \pi/2$ with $c^-e^{-x}
x^{2\theta_1+\theta_{\infty}-1} \ll x^{-r}.$
\par
$(2)$ For the solution $\mathrm{trunc}^1_0(x,m^0_{12}/m^1_{12},e^{-\pi i\theta_0}
)$ of Theorem $\ref{thm2.6},(2),$
\begin{align*}
\z=& -\theta_0+O(x^{-2})-c^-e^{-x}x^{2\theta_0-\theta_{\infty}-1}
(1+O(x^{-r})),
\\
x^{-\theta_{\infty}}u=&\tfrac 2{\theta_0-\theta_1-\theta_{\infty}}
c_x e^xx^{-\theta_0+1}\bigl(1+O(x^{-1})
\\
&\phantom{----} +[x^{-1}]c^-e^{-x}x^{2\theta_0-\theta_{\infty}-1} +\theta_0^{-2}
(c^-e^{-x}x^{2\theta_0-\theta_{\infty}-1})^2(1+O(x^{-r}))\bigr)
\end{align*}
as $x\to \infty$ through $-\pi/2\le \arg x< \pi/2$ with $c^-e^{-x}
x^{2\theta_0-\theta_{\infty}-1} \ll x^{-r}.$
\par
$(3)$ For the solution $\mathrm{trunc}^0_{\infty}(x,m^1_{21}/m^0_{21},
e^{-\pi i\theta_1})$ of Theorem $\ref{thm2.8},(1),$
\begin{align*}
\z=& -\tfrac{\theta_0-\theta_1+\theta_{\infty}}2(1+O(x^{-2}))
 +\tfrac 
{\theta_0+\theta_1-\theta_{\infty}}2\theta_1 c^-e^{-x}x^{-2\theta_1
+\theta_{\infty}-1}(1+O(x^{-r})),
\\
x^{-\theta_{\infty}}u=&-c_0 x^{-\theta_1}
\bigl(1+O(x^{-1})-\theta_1c^-e^{-x}x^{-2\theta_1
+\theta_{\infty}-1}(1+O(x^{-r}))\bigr)
\end{align*}
as $x\to \infty$ through $-\pi/2<\arg x\le \pi/2$ with $c^-e^{-x}
x^{1-2\theta_1+\theta_{\infty}} \ll x^{-r}.$
\par
$(4)$ For the solution $\mathrm{trunc}^1_{\infty}(x,m^0_{12}/m^1_{12},
e^{\pi i\theta_0})$ of Theorem $\ref{thm2.8},(2),$
\begin{align*}
\z=& \tfrac{\theta_0+\theta_1+\theta_{\infty}}2\bigl(\tfrac{\theta_0-\theta_1
+\theta_{\infty}}2(\theta_0x^{-2}+O(x^{-3})) 
+ \theta_{0}c^-e^{-x}
x^{-2\theta_0-\theta_{\infty}-1}(1+O(x^{-r})) \bigr),
\\
x^{-\theta_{\infty}}u=&c_x e^xx^{\theta_0-1}\bigl(-\tfrac
{\theta_0-\theta_1+\theta_{\infty}}2
(1+O(x^{-1})) -c^-e^{-x}x^{-2\theta_0-\theta_{\infty}+1}(1+O(x^{-1})) 
\\
&\phantom{--}+\tfrac {\theta_0+\theta_1+\theta_{\infty}}2
(c^-e^{-x}x^{-2\theta_0-\theta_{\infty}})^2(1+O(x^{-r}))\bigr)
\end{align*}
as $x\to \infty$ through $-\pi/2\le \arg x< \pi/2$ with $c^-e^{-x}
x^{1-2\theta_0-\theta_{\infty}} \ll x^{-r}.$
\end{prop}
\section{Series expression and uniqueness of truncated solutions}
\label{sc6}
By $y=-\tan^2Y$, (P$_{\mathrm{V}}$) is written in the form
$$
x(xY_x)_x =\frac{a_{\theta}}2 \frac{\sin Y}{\cos^3Y} -\frac{b_{\theta}}2
\frac{\cos Y}{\sin^3Y} +\frac{c_{\theta}}4 x\sin 2Y +\frac{x^2}{16}\sin 4Y.
$$
Further changes of variables $Y=\pi/4 +v$, $t=x/2$ and $v=\psi(t)+t^{-1/2}w$ 
lead to
\begin{align}\label{6.1}
w_{tt}+w= &f(t,w)
\\
\notag
 = & t^{-2}g_1w +t^{-3/2}g_2 w^2+ \tfrac 83 t^{-1}(1+t^{-1}g_3)w^3
+\sum_{j\ge 4}t^{-(j+(-1)^j)/2}g_{j}w^{j},
\end{align}
where $\psi(t)$ is holomorphic in the sector $|\arg t|<\pi$ and admits
the asymptotic representation $\psi(t)\sim \tfrac 12(1-\theta_0-\theta_1)t^{-1}
+\sum_{j\ge 2}\psi_j t^{-j}$, and so are $g_j=O(1)$ in the same sector 
\cite[Section 4.2]{S1}. To this equation \cite[Corollary 2.3]{S1} applies
and yields the following solutions expanded into convergent series
valid as $x\to \infty$ along the positive real axis.
\begin{prop}\label{prop6.1}
Let $R_0>0$ be a given positive number. Then there exists 
a two-parameter family of solutions of $(\mathrm{P}_{\mathrm{V}})$ consisting of
$$
y(c_1,c_2,x)=\psi_0(x)-4\sqrt{2}e^{\pi i/4}x^{-1/2}\Psi(x,
c_1x^{4c_1c_2}e^{ix/2},c_2x^{-4c_1c_2}e^{-ix/2})
$$ 
for $|c_1c_2|<R_0,$ where $\psi_0(x)$ is holomorphic
and admits the asymptotic representation
$$
\psi_0(x)\sim -1 -4(1-\theta_0-\theta_1)x^{-1}+\sum_{j\ge 2}\psi_j x^{-j}
$$
as $x\to\infty$ through the sector $\Sigma_0:$ $|\arg x|<\pi$, 
and $\Psi(x,Y,Z)$ is holomorphic for $x\in \Sigma_0,$ $|Y|,|Z|<r_0|x|^{1/2}$, 
$|YZ|<R_0$ with sufficiently small $r_0,$  
and expanded into the convergent series
$$
\Psi(x,Y,Z)=(1+v_{10}(x))Y+(1+v_{01}(x))Z+\sum_{j+k\ge 2}v_{jk}(x)Y^jZ^k
$$
with $v_{jk}(x)$ admitting asymptotic representations in $x^{-1/2}$ 
as $x\to\infty$ in 
$\Sigma_0$ and satisfying $v_{10}(x),$ $v_{01}(x)=O(x^{-1})$ and
$v_{jk}(x)=O(x^{-(\mathrm{max}\{j,k\}+\delta_{jk}-1)/2}).$
\end{prop}
The following solutions are expressed by series expansions along imaginary axes, 
which are obtained from \cite[Theorem 2.2]{S1}, \cite{T,Y}.
\begin{prop}\label{prop6.2}
Let $R_0$ be as in Proposition $\ref{prop6.1}$. Then there exists a
two-parameter family of solutions of $(\mathrm{P}_{\rm{V}})$ consisting of 
$$
y(c_1,c_2,x) =\varphi(x)+x^{-1/2}V(x, 
c_1 e^{x}x^{3/2-2\theta_1-\theta_{\infty}-4c_1c_2},
c_2 e^{-x}x^{-3/2+2\theta_1+\theta_{\infty}+4c_1c_2})
$$
for $|c_1c_2|<R_0,$ where
$\varphi(x)$ is holomorphic and admits the asymptotic representation
$$
\varphi(x) \sim \tfrac 12(\theta_0-\theta_1-\theta_{\infty})x^{-1}+\sum_{j\ge 2}
\varphi_j x^{-j}
$$ 
as $x\to\infty$ through the sector $\Sigma_1:$ 
$|\arg x-\pi/2|<\pi$, 
and $V(x,Y,Z)$ is holomorphic for $x\in \Sigma_1,$ $|Y|,|Z|<r_0|x|^{1/2}$, 
$|YZ|<R_0$ with sufficiently small $r_0,$  
and expanded into the convergent series
$$
V(x,Y,Z)=(1+v_{01}(x))Z+v_{10}(x)Y+\sum_{j+k\ge 2}v_{jk}(x)Y^jZ^k,
$$
whose coefficients $v_{jk}(x)$ admit asymptotic representations in $x^{-1/2}$
as $x\to \infty$ in $\Sigma_1$ and satisfy
$v_{01}(x)=O(x^{-1}),$ $v_{10}(x)=O(x^{-1})$,
and $v_{jk}(x)=O(x^{-(\max\{j,k\}+\delta_{jk}-1)/2})$. 
\end{prop}
\begin{proof}
Equation (P$_{\mathrm{V}}$) is written in the Hamiltonian system
$$
dq/dx=\partial H_{\mathrm{V}}/\partial p,\quad 
dp/dx=-\partial H_{\mathrm{V}}/\partial q 
$$
with the Hamiltonian function
$$
H_{\mathrm{V}}(x,q,p)=x^{-1}(q(q-1)^2p^2 -(\kappa_0(q-1)^2+\theta q(q-1)
+xq)p+\kappa(q-1))
$$
given by Okamoto \cite{O}, where 
$$
q=y, \quad
\kappa_0=\pm\sqrt{2b_{\theta}}=\pm\tfrac 12 (\theta_0-\theta_1-\theta_{\infty}),
\quad \theta=-1-c_{\theta}=-2+\theta_0+\theta_1
$$
and $\kappa=\tfrac 14(\kappa_0+\theta)^2-\tfrac 14\kappa_{\infty}^2$, 
$\kappa_{\infty}^2=2a_{\theta}.$ By a canonical transformation of the form
$$
q=x^{-1/2}P+\varphi(x),\quad p=-x^{1/2}Q+\tilde{\varphi}(x)
$$
with $\varphi(x)=-\kappa_0 x^{-1}+O(x^{-2})$ and $\tilde{\varphi}(x)=-\kappa
x^{-1}+O(x^{-2})$ admitting asymptotic representations in $|\arg x-\pi/2|<\pi$,
the Hamiltonian function $H_{\mathrm{V}}(x,q,p)$ is taken to
\begin{align*}
H(x,Q,P)=& (1-ax^{-1}+x^{-2} h_{11})QP +x^{-3}h_{20}Q^2 +x^{-1}h_{02}P^2 
\\
& +x^{-9/2}h_{30}Q^3+x^{-3/2}h_{21}Q^2P +x^{-1/2}h_{12}QP^2 
\\
& -2x^{-1}(1+x^{-1}h_{22})Q^2P^2
+x^{-3}h_{31}Q^3P +x^{-3/2}h_{32}Q^3P^2,
\end{align*} 
where $a=\tfrac 12 +\theta+2\kappa_0=-\tfrac 32+\theta_0+\theta_1 \pm
(\theta_0-\theta_1-\theta_{\infty})$, and $h_{jk}=O(1)$ admit asymptotic
representations in $|\arg x-\pi/2|<\pi.$ 
Using \cite[Theorem 2.2]{S1}, and setting $V(x,Y,Z)=V^-(x,Y,Z)$ and 
$\kappa_0=-\tfrac12(\theta_0-\theta_1-\theta_{\infty})$,
we arrive at the proposition.  
\end{proof}
\begin{rem}\label{rem6.1}
In the proof of Proposition \ref{prop6.2}, we may use 
$\kappa_0=\tfrac12(\theta_0-\theta_1-\theta_{\infty})$ and
$a=-\tfrac 32+2\theta_0-\theta_{\infty}$. Then we have a family of solutions
with $c_1 e^{x}x^{3/2-2\theta_0+\theta_{\infty}-4c_1c_2},$
$c_2 e^{-x}x^{-3/2+2\theta_0-\theta_{\infty}+4c_1c_2},$ and
$\varphi(x) \sim -\tfrac 12(\theta_0-\theta_1-\theta_{\infty})x^{-1}
+\sum_{j\ge 2}\varphi_j x^{-j}.$
\end{rem}
\subsection{Derivation of Theorem \ref{thm2.4}}\label{ssc6.1} 
The truncated solutions in \cite[Corollaries 5.1--5.3]{Andreev-Kitaev}
may be supposed to be given in a strip containing the 
positive real axis. By Proposition \ref{prop6.1} the solution of 
(P$_{\mathrm{V}}$)
$$
y(a,x):=y(c_1,0,x)=\psi_0(x)+ax^{-1/2}e^{ix/2}(1+O(x^{-1/2})) 
$$
with $\psi_0(x)=-1-4(1-\theta_0-\theta_1)x^{-1}+O(x^{-2}),$
$ a=-4\sqrt{2}e^{\pi i/4}c_1$ is given in the
sector $0\le \arg x <\pi$. On the other hand, the solution of \cite[Corollary
5.2]{Andreev-Kitaev} is of the form
$$
\hat{y}(a,x)=-1+ax^{-1/2}e^{ix/2}+O(x^{-1})
$$
as $x\to +\infty$ in the strip $\mathcal{S}$. 
To prove Theorem \ref{thm2.4},(3), it is sufficient to show that 
$y(a,x)\equiv \hat{y}(a,x)$ in $\mathcal{S}.$  Let $y(a,x)$ and
$\hat{y}(a,x)$ correspond to $w(a,t)$ and $\hat{w}(a,t)$, respectively, by the
transformation $t=x/2$, $y=-\tan^2(\pi/4+\psi(x)+t^{-1/2}w).$ Then 
$w(a,t)=ae^{it}(1+O(t^{-1/2}))$ and $\hat{w}(a,t)=ae^{it}+O(t^{-1/2})$
for $2t\in \mathcal{S}$, and both functions solving \eqref{6.1} also fulfil the
integral equation
$$
w(t)=ae^{it}-\int_{+\infty}^t\sin(t-s) f(s,w(s))ds,
\quad  \re t< \re s<+\infty, \,\,\, 2t, 2s\in \mathcal{S},
$$
in which $f(s,w)=\tfrac 83s^{-1}w^3 +f_0(s,w)s^{-3/2}w$ with $f_0(s,w)\ll 1$,
and, for $w(s)=w(a,s)$ and $\hat{w}(a,s)$, we have
$\int^t_{+\infty}\sin(t-s)f(s,w(s))ds\ll t^{-1/2}$. 
Now let $v(t)$ be such that $|\hat{w}(a,t)-w(a,t)|=t^{-1/2}v(t),$ $v(t)\ll 1.$
Then
\begin{align*}
v(t) &\le t^{1/2} \int^t_{+\infty} ( 8(|w(a,s)|^2+|w(a,s)|) s^{-1}
 +Ks^{-3/2})s^{-1/2}v(s)ds
\\
&\le (64|a|(|a|+1)+Kt^{-1/2}) \sup_{\re s \ge \re t,\,\, 
2s\in\mathcal{S}}v(s),
\end{align*}
if $2t \in \mathcal{S}$ is sufficiently large, where $K$ is some positive 
number independent of $|a|$. This implies $v(t)\equiv 0$ in $\mathcal{S},$
provided that, say, $|a|<1/200.$ Since $w(a,t)$ and $\hat{w}(a,t)$ are 
analytic in $a$, we arrive at the expression of Theorem \ref{thm2.4},(3).
In particular the special case $a=0$ implies the uniqueness of the 
doubly-truncated solution. 
\begin{rem}\label{rem6.2}
By the supposition \cite[(7.6)]{Andreev-Kitaev} in the derivation of
\cite[Corollaries 5.1--5.3]{Andreev-Kitaev}, the expression of a truncated 
solution, say, of \cite[Corollary 5.2]{Andreev-Kitaev} is valid in the
sector $0\le \arg t<\pi$. This fact enable us to simplify the discussion above
(cf. \cite{S0}).
\end{rem}
\subsection{Uniqueness of truncated solutions}\label{ssc6.2}
Solutions of Proposition \ref{prop6.2}, Remark \ref{rem6.1}
or the reciprocals with vanishing $c_1$ or $c_2$ become basic 
truncated solutions (respectively, 
doubly-truncated solutions) in the sector $-\pi/2<\arg x\le \pi/2$ or
$\pi/2\le \arg x< 3\pi/2$ (respectively, $|\arg x-\pi/2|<\pi$), each of
which is uniquely determined by its leading terms. In the proofs of 
Theorems \ref{thm2.6} and \ref{thm2.8} in Subsection \ref{ssc5.2}, 
we have used the fact that a solution given for $0<\arg x<\pi/2$ 
or $-\pi/2<\arg x<0$ coincides with one of these basic truncated solutions 
with the same leading terms. The coincidence is verified in a small sector,
say, $|\arg x-\pi/4|<\pi/10$ or $|\arg x+\pi/4|<\pi/10$ 
by using the equation \cite[(5.1)]{S0} with $z=ix$ and the exponentially
decaying property of the difference. 
\section{Derivation of the monodromy data in Proposition \ref{prop5.1}}\label{sc7}
Let us derive the monodromy data of Proposition \ref{prop5.1}.
Suppose that $\arg x \sim \pi/2$ and that $|e^xx^{\sigma}|, |e^{-x}x^{-\sigma}|
\ll 1$ as $x\to \infty$. As described in Subsection \ref{ssc5.1}
the monodromy data corresponding to $y(c,\sigma,x)$ is given by the invariant 
monodromy for the linear system \eqref{3.1} with $(A_0,A_1)= 
(A_0(x),A_1(x)):=(A_0(\mathbf{c},\sigma, x),A_1(\mathbf{c},\sigma,x))$ such that
\begin{align*}
& A_0(x) \sim \tfrac 14 (\sigma-\theta_{\infty})\sigma_3 +\rho_-c_0 x^{-(\sigma
+\theta_{\infty})/2}\sigma_{12}+\eta_+c^{-1}_0 x^{(\sigma+\theta_{\infty})/2}
\sigma_{21},
\\
& A_1(x) \sim -\tfrac 14 (\sigma+\theta_{\infty})\sigma_3 +\rho_+c_x e^x 
x^{(\sigma-\theta_{\infty})/2}\sigma_{12}+\eta_-c^{-1}_x e^{-x}
x^{-(\sigma-\theta_{\infty})/2}\sigma_{21},
\end{align*} 
where $\mathbf{c}=(c_0,c_x)$, $c=c_x/c_0,$ and $\rho_{\pm},$ $\eta_{\pm}$
are as in \eqref{5.b},
since $(A_0(x),A_1(x))$ solves \eqref{5.00}. By $\Psi=e^{\frac 14 x\sigma_3}
x^{-\frac 14\theta_{\infty}\sigma_3} Z$, system \eqref{3.1} with
$(A_0(x),A_1(x))$ is changed into
\begin{align}\label{7.1}
& \frac{dZ}{d\lambda}=\Bigl(\frac{\sigma_3}{2}+\frac{\hat{A}_0(x)
}{\lambda}+\frac{\hat{A}_1(x)}{\lambda-x} \Bigr)Z,
\\[0.3cm]
\notag
&\hat{A}_0(x)=\bigl(\tfrac 14(\sigma-\theta_{\infty})+O(x^{-1})\bigr)\sigma_3
\\
\notag
&\phantom{----} +(\rho_-c_0+O(x^{-1}))e^{-x/2}x^{-\sigma/2}\sigma_{12}
+(\eta_+c_0^{-1}+O(x^{-1}))e^{x/2}x^{\sigma/2}\sigma_{21},
\\
\notag
&\hat{A}_1(x)=\bigl(-\tfrac 14(\sigma+\theta_{\infty})+O(x^{-1})\bigr)\sigma_3
\\
\notag
&\phantom{----}
+(\rho_+c_x+O(x^{-1}))e^{x/2}x^{\sigma/2}\sigma_{12}
+(\eta_-c_x^{-1}+O(x^{-1}))e^{-x/2}x^{-\sigma/2}\sigma_{21}.
\end{align}
\subsection{Necessary facts}\label{ssc7.1}
We summarise necessary facts, whose proofs are found in \cite[Sections 7.1
-- 7.3]{S-2018}. The characteristic roots of the coefficient $A(x,\lambda)
=\tfrac 12\sigma_3+\lambda^{-1}\hat{A}_0(x) +(\lambda-x)^{-1}\hat{A}_1(x)$ 
of \eqref{7.1} are $\pm \mu(x,\lambda)$ with
$$
\mu(x,\lambda)=\frac 12 
+\frac {\frac 14(\sigma-\theta_{\infty})+O(x^{-1})}{\lambda}
-\frac {\frac 14(\sigma+\theta_{\infty})+O(x^{-1})}{\lambda-x}
+(|\lambda|^{-2}+|\lambda-x|^{-2})
$$
as long as $|\lambda|>|x|^{1/2},$ $|\lambda-x|>|x|^{1/2}.$
By $Z=(\mu(x,\lambda)\sigma_3+A(x,\lambda))\hat{Z}=(\sigma_3+O(|\lambda|^{-1}
+|\lambda-x|^{-1}))\hat{Z},$ system \eqref{7.1} is reduced to
\begin{equation*}\label{7.2}
\frac{d\hat{Z}}{d\lambda}=(\mu(x,\lambda)\sigma_3+
O(|\lambda|^{-2}+|\lambda-x|^{-2}))\hat{Z}.
\end{equation*}
This fact implies the following \cite[Lemma 7.1]{S-2018}.
\begin{lem}\label{lem7.1}
System \eqref{7.1} admits the matrix solution
\begin{align*}
&Z^1(x,\lambda)=(\sigma_3+O(|\lambda|^{-1}+|\lambda-x|^{-1}))
e^{\frac 12\lambda\sigma_3}\lambda^{\alpha(x)\sigma_3}
(\lambda-x)^{\beta(x)\sigma_3},
\\
& \alpha(x)=\tfrac 14(\sigma-\theta_{\infty})+O(x^{-1}),\quad
 \beta(x)=-\tfrac 14(\sigma+\theta_{\infty})+O(x^{-1})
\end{align*}
uniformly in sufficiently large $x$ as $\lambda\to \infty$ through the region
$$
\Sigma_{\pi/2}(x):\,\,\, |\arg\lambda-\pi/2|<\pi,\quad |\arg(\lambda-x)-\pi/2|
<\pi/4,\quad |\lambda-x|>|x|^{1/2}.
$$
System \eqref{7.1} also admits a solution $Z^0(x,\lambda)$ having the same 
asymptotic form as of $Z^1(x,\lambda)$ as $\lambda\to\infty$ through the region
$$
\Sigma_{3\pi/2}(0):\,\,\,|\arg\lambda-3\pi/2|<\pi/4,\quad |\lambda|>|x|^{1/2}.
$$
\end{lem}
If $|\lambda|<2|x|^{1/2},$ by $V=e^{\frac 14 x\sigma_3}x^{\frac 14 \sigma
\sigma_3} c_0^{-\sigma_3/2}Z$, \eqref{7.1} is taken to
\begin{equation}\label{7.3}
\frac{dV}{d\lambda} =\Bigl(\frac{\sigma_3}2+\frac{\Lambda}{\lambda}+E(x,\lambda)
\Bigr)V, \quad \Lambda=\tfrac14(\sigma-\theta_{\infty})\sigma_3+\rho_-\sigma_{12}
+\eta_+\sigma_{21},
\end{equation}
where $\rho_-=\frac 14(\sigma+2\theta_0-\theta_{\infty}),$
$\eta_+=\frac 14(-\sigma+2\theta_0+\theta_{\infty}),$ and
$E(x,\lambda)\ll x^{-1}$. This is a perturbation of the Whittaker system
\begin{equation}\label{7.4}
\frac{dW}{d\lambda} =\Bigl(\frac{\sigma_3}2 +\frac{\Lambda}{\lambda} \Bigr)W
\end{equation}
admitting a matrix solution 
$$
W_{\infty}(\lambda)=(I+O(\lambda^{-1}))e^{\frac 12\lambda \sigma_3}\lambda
^{\frac 14(\sigma-\theta_{\infty})\sigma_3}
$$
as $\lambda\to \infty$ through the sector $|\arg\lambda-\pi/2|<\pi.$ Let
$W^1_{\infty}(\lambda)$ and $W^3_{\infty}(\lambda)$ be canonical solutions
having the same asymptotic representations as $\lambda\to\infty$ through the
sectors $|\arg\lambda+\pi/2|<\pi$ and $|\arg\lambda-3\pi/2|<\pi$, respectively.  
\begin{lem}\label{lem7.2}
In the domain $|\arg\lambda-3\pi/2|<\pi/4,$ $|\ln x|<|\lambda|<2|x|^{1/2},$
system \eqref{7.3} admits the matrix solution
$$
V(x,\lambda)=(I+V_{*}(x,\lambda)) e^{\frac 12 \lambda\sigma_3}
\lambda^{\frac14(\sigma-\theta_{\infty})\sigma_3}, \quad
V_{*}(x,\lambda) \ll |\ln  x|^{-1}.
$$
This is also written as $V(x,\lambda)=(I+O(|\ln x|^{-1}))W_{\infty}^3(\lambda).$
\end{lem}
\noindent
For the proof of Lemma \ref{lem7.2} see \cite[Proof of Lemma 7.4]{S-2018}.
\par
The solution $W_{\infty}(\lambda)$ of \eqref{7.4} is written in the form
$$
\begin{pmatrix} 
e^{\pi i(\sigma-\theta_{\infty}+2)/4}W_{(\sigma-\theta_{\infty}+2)/4,
\,\theta_0/2} (e^{-\pi i}\lambda)   &
-\theta_+ W_{-(\sigma-\theta_{\infty}+2)/4,\, \theta_0/2} (\lambda)  \\ 
\theta_- e^{\pi i(\sigma-\theta_{\infty}+2)/4}W_{(\sigma-\theta_{\infty}-2)/4, 
\,\theta_0/2} (e^{-\pi i}\lambda)   &
 W_{-(\sigma-\theta_{\infty}-2)/4,\, \theta_0/2} (\lambda)  
\end{pmatrix} \lambda^{-1/2},
$$
in which $\theta_{\pm}=\frac 14(\sigma-\theta_{\infty} \pm 2\theta_0),$
or $\theta_+=\rho_-,$ $\theta_-=-\eta_+,$ and $W_{\kappa,\nu}(z)\sim
e^{-z/2}z^{\kappa}$ as $z\to \infty$ in $|\arg z|<3\pi/2$ is the Whittaker
function \cite[13.1.33]{AS}, \cite[(3.10)]{J}. The formulas
\begin{align*}
W_{\kappa,\nu}(e^{-\pi i}\lambda) & =e^{-\pi i(\nu+1/2)}e^{\lambda/2}\lambda
^{\nu+1/2}\Bigl(\frac{(1-e^{4\pi i\nu})\Gamma(-2\nu)}{\Gamma(1/2-\nu-\kappa)}
M(\nu-\kappa+1/2, 2\nu+1, e^{\pi i}\lambda)
\\
&\phantom{--------}+e^{4\pi i\nu}U(\nu-\kappa+1/2,2\nu+1,e^{\pi i}\lambda)\Bigr),
\\
W_{\kappa,\nu}(\lambda) & =e^{-\lambda/2}\lambda^{\nu+1/2}\Bigl(\frac
{(1-e^{-4\pi i\nu})\Gamma(-2\nu)}{\Gamma(1/2-\nu-\kappa)}
M(\nu-\kappa+1/2, 2\nu+1, e^{-2\pi i}\lambda)
\\
&\phantom{--------}+e^{-4\pi i\nu}U(\nu-\kappa+1/2,2\nu+1,
 e^{-2\pi i}\lambda)\Bigr)
\end{align*}
near $\lambda=\infty$ follow from \cite[13.1.10, 13.1.33]{AS} with $n=\mp 1.$ 
Here, by \cite[13.5.1,13.5.2]{AS},
\begin{align*}
&M(\nu-\kappa+1/2, 2\nu+1,e^{\pi i}\lambda)=\frac{\Gamma(2\nu+1)\lambda
^{-(\nu-\kappa+1/2)} }{\Gamma(\nu+\kappa+1/2)}(1+O(\lambda^{-1}))
\\
&\phantom{--------} +\frac{\Gamma(2\nu+1)}{\Gamma(\nu-\kappa+1/2)}
e^{-\lambda}(e^{\pi i}\lambda)^{-(\nu+\kappa+1/2)}(1+O(\lambda^{-1})),
\\
&U(\nu-\kappa+1/2,2\nu+1,e^{\pi i}\lambda)=(e^{\pi i}\lambda)^{-(\nu-\kappa+1/2)}
(1+O(\lambda^{-1}))
\end{align*}
as $\lambda\to\infty$ in $|\arg\lambda+\pi/2|<\pi.$ 
These relations lead to Stokes matrices for \eqref{7.4}.
\begin{lem}\label{lem7.3}
Let $W_{\infty}(\lambda)=W^1_{\infty}(\lambda)S_*$ and $W_{\infty}(\lambda)=
W^3_{\infty}(\lambda)S_{**}.$ Then
\begin{align*}
& S_* =I-\frac{2\pi i}{\Gamma(-\frac 14(\sigma-2\theta_0-\theta_{\infty}))
\Gamma(1-\frac 14(\sigma+2\theta_0-\theta_{\infty}))}\sigma_{21},
\\
& S_{**} =I+\frac{2\pi i e^{-\pi i(\sigma-\theta_{\infty})/2}}
{\Gamma(\frac 14(\sigma+2\theta_0-\theta_{\infty}))
\Gamma(1+\frac 14(\sigma-2\theta_0-\theta_{\infty}))}\sigma_{12}.
\end{align*}
\end{lem}
\begin{lem}\label{lem7.4}
Set $W^3_{\infty}(e^{2\pi i}\lambda)=W^3_{\infty}(\lambda)M^0_{\mathrm{local}}$.
Then $M^0_{\mathrm{local}}=S_{**}S_*^{-1}e^{\frac{\pi i}2(\sigma-\theta_{\infty})
\sigma_3}.$
\end{lem}
\begin{proof}
For $|\arg(e^{2\pi i}\lambda)-3\pi/2|=|\arg\lambda+\pi/2|<\pi$ we have
$$
W^3_{\infty}(e^{2\pi i}\lambda)\sim e^{\frac12\lambda\sigma_3}\lambda^{\frac14
(\sigma-\theta_{\infty})\sigma_3} e^{\frac{\pi i}2(\sigma-\theta_{\infty})
\sigma_3},
$$
and hence $W^3_{\infty}(e^{2\pi i}\lambda)=W^1_{\infty}(\lambda)
e^{\frac{\pi i}2(\sigma-\theta_{\infty})\sigma_3}$, which implies the lemma.
\end{proof}
\subsection{Derivation of $M^0$ for $\arg x\sim \pi/2$}\label{ssc7.2}
System \eqref{3.1} admits the canonical solution $\Psi(x,\lambda)=(I+O(\lambda
^{-1}))e^{\frac 12\lambda\sigma_3}\lambda^{-\frac12\theta_{\infty}\sigma_3}$
as $\lambda\to\infty$ in the sector $|\arg\lambda-\pi/2|<\pi.$ The monodromy
matrices $M^0$ and $M^1$ are defined by the analytic continuations along loops
$l_0$ and $l_x$ surrounding $0$ and $x$, respectively as in 
Figure \ref{loops1},(a). 
For the canonical solutions $\Psi_1(x,\lambda)$ and $\Psi_3(x,\lambda)$ with the
same asymptotic representations in $|\arg\lambda+\pi/2|<\pi$ and
$|\arg\lambda-3\pi/2|<\pi$, respectively, the Stokes matrices $S_1$ and $S_2$
are given by $\Psi(x,\lambda)=\Psi_1(x,\lambda)S_1,$ $\Psi_3(x,\lambda)=
\Psi(x,\lambda)S_2.$ To calculate $M^0$ let us examine the analytic continuation
of $\Psi(x,\lambda)$ along a suitably modified loop $l_0$. 
Suppose that $\arg x\sim \pi/2$ and that $l_0$ starts 
from $\lambda_{\mathrm{st}}$ such that $\arg \lambda_{\mathrm{st}} \sim \pi/2,$ 
$\arg (\lambda_{\mathrm{st}}-x) \sim \pi/2$ and $|\lambda_{\mathrm{st}}|>2|x|.$  
The loop $l_0$ may be modified to the form $l_0=(\Gamma_{\mathrm{left}}\circ
L_-)\circ \Gamma_0 \circ (\Gamma_{\mathrm{left}}\circ L_-)^{-1}$, where
$\Gamma_{\mathrm{left}}$ is an arc $\lambda=|\lambda_{\mathrm{st}}|e^{it}$
$(\pi/2 \le t\le 3\pi/2)$, $L_-$ is a segment joining $\lambda=-i|\lambda
_{\mathrm{st}}|$ to $\lambda=-i|x|^{1/4},$ and $\Gamma_0$ is defined by
$\lambda=|x|^{1/4}e^{it}$ $(3\pi/2\le t\le 7\pi/2)$ as in Figure 
\ref{loops1},(b). Along $\Gamma_{\mathrm{left}}\circ L_-$ matchings are
carried out according to the scheme:
$$
\Psi(x,\lambda)=\Psi_3(x,\lambda)S_2^{-1} \,\,\, \longleftrightarrow \,\,\,
Z^0(x,\lambda) \,\,\, \longleftrightarrow \,\,\, V(x,\lambda),
$$
and $W^3_{\infty}(\lambda)$ along $\Gamma_0$ yields the local monodromy 
$M^0_{\mathrm{local}}$.
{\small
\begin{figure}[htb]
\begin{center}
\unitlength=0.75mm
\begin{picture}(60,60)(-30,-30)
\put(0,8){\circle*{1}}   
\put(0,8){\circle{6}}   
\put(0,-8){\circle*{1}}
\put(0,-8){\circle{6}}  

\put(5.5,5.5){\makebox{$x$}}
\put(5.5,-8.5){\makebox{$0$}}
\put(0,23){\line(0,-1){12}}  

\put(0,23){\circle*{1.5}}
\put(3,22){\makebox{$\lambda_{\mathrm{st}}$}}
 \qbezier(0,23)(-18,-4)(-3,-8)
\thicklines
\put(-12,-34){\makebox{(a) Loops}}
\end{picture}
\quad
\begin{picture}(60,60)(-30,-30)
\put(0,15){\circle*{1}}   
\put(0,0){\circle*{1}}
\put(0,0){\circle{10}}  

\put(2.5,12.5){\makebox{$x$}}
\put(-8.5,-3.5){\makebox{$0$}}
\put(5.5,2.5){\makebox{$\Gamma_0$}}
\qbezier(-7,-7)(0,-12)(7,-7)  
\put(-7,-7){\line(-1,-1){14}}
\put(7,-7){\line(1,-1){11}}
\put(0,-23){\line(0,1){18}}  

\put(10,-23){\makebox{$\Sigma_{3\pi/2}(0)$}}
\put(0,23){\circle*{1.5}}
\put(3,22){\makebox{$\lambda_{\mathrm{st}}$}}
\put(-34,5){\makebox{$\Gamma_{\mathrm{left}}$}}
\put(1,-16){\makebox{$L_-$}}
\thicklines
\qbezier(0,23)(-9.52,23)(-16.26,16.26)
\qbezier(-23,0)(-23,9.52)(-16.26,16.26)
\qbezier(0,-23)(-9.52,-23)(-16.26,-16.26)
\qbezier(-23,0)(-23,-9.52)(-16.26,-16.26)
\put(-6,-34){\makebox{(b) $l_0$}}
\end{picture}
\quad
\begin{picture}(60,60)(-30,-30)
\put(0,-20){\circle*{1}}
\put(0,-5){\circle*{1}}
\put(0,-5){\circle{10}}   
\put(6,-8){\makebox{$x$}}
\put(-8,-12){\makebox{$\Gamma_x$}}
\put(-7.5,8){\makebox{$L_+$}}

\put(2.5,-21){\makebox{$0$}}
\qbezier(-7,2)(0,7)(7,2)
\put(-7,2){\line(-1,1){14}}
\put(7,2){\line(1,1){11}}
\put(0,12){\line(0,-1){12}}
\put(10,16){\makebox{$\Sigma_{\pi/2}(x)$}}
\put(0,23){\circle*{1.5}}
\put(-8,22){\makebox{$\lambda_{\mathrm{st}}$}}
\thicklines
\put(-0.2,23){\line(0,-1){11}}
\put(0,23){\line(0,-1){11}}
\put(0.2,23){\line(0,-1){11}}
\put(-6,-34){\makebox{(c) $l_x$}}
\end{picture}
\end{center}
\caption{Loops for $\arg x\sim \pi/2$}
\label{loops1}
\end{figure}
}
\par
(1) Continue $\Psi(x,\lambda)$ along $\Gamma_{\mathrm{left}}$ entering 
$\Sigma_{3\pi/2}(0)\cap \{|\lambda|>2|x|\},$ in which $|\arg(\lambda-x)-3\pi/2|
<\pi/4.$ Let us match $\Psi_3(x,\lambda)$ with $Z^0(x,\lambda)$ in this domain.
Since $Z^0(x,\lambda)$ solves \eqref{7.1} that follows from \eqref{3.1} via
$\Psi=e^{\frac 14x\sigma_3} x^{-\frac14\theta_{\infty}\sigma_3} Z$, we set
\begin{equation}\label{7.4a}
\Psi_3(x,\lambda)=e^{\frac 14x\sigma_3}x^{-\frac 14\theta_{\infty}\sigma_3}
Z^0(x,\lambda)\Upsilon_1(x)
\end{equation}
with invertible $\Upsilon_1(x).$ By Lemma \ref{lem7.1}, the right-hand side is
\begin{align*}
& e^{\frac 14x\sigma_3} x^{-\frac14\theta_{\infty}\sigma_3}(\sigma_3+O(\lambda
^{-1})) e^{\frac12\lambda\sigma_3 } \lambda^{\alpha(x)\sigma_3}(\lambda-x)
^{\beta(x)\sigma_3}\Upsilon_1(x)
\\
=&(I+O(|\ln x|^{-1})) e^{\frac12\lambda\sigma_3}\lambda^{-(\theta_{\infty}/2
+O(x^{-1}))\sigma_3} e^{\frac14 x\sigma_3}x^{-\frac14\theta_{\infty}\sigma_3}
\sigma_3 \Upsilon_1(x),
\end{align*}
provided that, say, $|\lambda|\gg |x\ln x|.$
Under the condition $|\lambda|\ll e^{|x|^{1/2}},$ which implies 
$\lambda^{O(x^{-1})}=1+O(x^{-1/2}),$ from \eqref{7.4a} we have $\Upsilon_1(x)
=e^{-\frac14 x\sigma_3}x^{\frac14 \theta_{\infty}\sigma_3}(\sigma_3+O(|\ln x
|^{-1})).$
\par
(2) The line $L_-$ is contained in $|\lambda-3\pi/2|<\pi/4.$ Recall that
$V=e^{\frac 14x\sigma_3}x^{\frac14 \sigma\sigma_3}c_0^{-\sigma_3/2}Z$ takes
\eqref{7.1} to \eqref{7.3}. Set
\begin{equation}\label{7.5}
e^{\frac14 x\sigma_3}x^{\frac14 \sigma\sigma_3}c_0^{-\sigma_3/2}Z^0(x,\lambda)
=V(x,\lambda)\Upsilon_2(x)
\end{equation}
in the domain $|\arg\lambda-3\pi/2|<\pi/4,$ $|x|^{1/2}<|\lambda|<2|x|^{1/2}.$
The left-hand side is
\begin{align*}
&e^{\frac14 x\sigma_3}x^{\frac14\sigma\sigma_3}c_0^{-\sigma_3/2}(\sigma_3+
O(\lambda^{-1}))e^{\frac12\lambda\sigma_3}\lambda^{\alpha(x)\sigma_3}
(e^{\pi i}x(1-\lambda/x))^{\beta(x)\sigma_3}
\\
=& (I+O(\lambda^{-1}))e^{\frac12\lambda\sigma_3}\lambda^{((\sigma-
\theta_{\infty})/4+O(x^{-1}))\sigma_3}(1-\lambda/x)^{\beta(x)\sigma_3}
e^{\beta(x)\pi i\sigma_3}
\\
&\phantom{--------}\times e^{\frac14x\sigma_3}x^{(-\theta_{\infty}/4+O(x^{-1}))
\sigma_3}c_0^{-\sigma_3/2}\sigma_3,
\end{align*}
since $|\arg(\lambda-x)-3\pi/2|<\pi/4,$ $\arg x\sim \pi/2.$ Using Lemma 
\ref{lem7.2}, we derive $\Upsilon_2(x)=e^{-\frac{\pi i}4(\sigma+\theta_{\infty})
\sigma_3}e^{\frac14 x\sigma_3}x^{-\frac14\theta_{\infty}\sigma_3}c_0^{-\sigma_3
/2}(\sigma_3+O(x^{-1/2})).$
\par
(3) By Lemma \ref{lem7.2}, $V(x,\lambda)=(I+O(|\ln x|^{-1}))W^3_{\infty}
(\lambda)$ around $\lambda=-i|x|^{1/4}.$ From \eqref{7.4a} and \eqref{7.5}, it
follows that
$$
\Psi(x,\lambda)=x^{-\frac14(\sigma+\theta_{\infty})\sigma_3}c_0^{\sigma_3/2}
(I+O(|\ln x|^{-1}))W^3_{\infty}(\lambda)\Upsilon_0(x)
$$
with $\Upsilon_0(x)=\Upsilon_2(x)\Upsilon_1(x)S_2^{-1}=e^{-\frac{\pi i}4(\sigma+
\theta_{\infty})\sigma_3}c_0^{-\sigma_3/2}(I+O(|\ln x|^{-1}))S_2^{-1}.$
By Lemma \ref{lem7.4}, 
\begin{align*}
M^0=& \Upsilon_0(x)^{-1}M^0_{\mathrm{local}}\Upsilon_0(x)
\\
=&S_2(I+O(|\ln x|^{-1}))c_0^{\sigma_3/2}
e^{\frac{\pi i}4(\sigma+\theta_{\infty})\sigma_3}
S_{**}S_*^{-1}e^{\frac{ \pi i}2(\sigma-\theta_{\infty})\sigma_3}
e^{-\frac{\pi i}4(\sigma+\theta_{\infty})\sigma_3}c_0^{-\sigma_3/2}S_2^{-1}.
\end{align*}
The passage to the limit $x\to\infty$ leads to
\begin{equation}\label{7.6}
M^0=S_2c_0^{\sigma_3/2}e^{\frac{\pi i}4(\sigma+\theta_{\infty})\sigma_3}S_{**}
S_*^{-1}e^{\frac{\pi i}2(\sigma-\theta_{\infty})\sigma_3}
e^{-\frac{\pi i}4(\sigma+\theta_{\infty})\sigma_3} c_0^{-\sigma_3/2}S_2^{-1}.
\end{equation}
\subsection{Derivation of $M^1$ for $\arg x\sim \pi/2$}\label{ssc7.3}
In $|\lambda-x|<2|x|^{1/2},$ by $V=e^{-\frac 14x\sigma_3}x^{-\frac14\sigma
\sigma_3}c_1^{-\sigma_3/2}Z$, \eqref{7.1} is changed into
\begin{equation}\label{7.7}
\frac{dV}{d\lambda}=\Bigl(\frac{\sigma_3}2+\frac{\tilde{\Lambda}}{\lambda-x}
+\tilde{E}(x,\lambda)\Bigr)V, \quad \tilde{\Lambda}=-\tfrac14(\sigma+\theta
_{\infty})\sigma_3 +\rho_+\sigma_{12}+\eta_-\sigma_{21},
\end{equation}
where $\rho_+=\frac14(-\sigma+2\theta_1-\theta_{\infty})$, 
$\eta_-=\frac 14(\sigma+2\theta_1+\theta_{\infty}),$ and 
$\tilde{E}(x,\lambda) \ll x^{-1}.$ For this system consider the Whittaker system
\begin{equation}\label{7.8}
\frac{dW}{d\lambda}=\Bigl(\frac{\sigma_3}2+\frac{\tilde{\Lambda}}{\lambda-x}
\Bigr)W,
\end{equation}
which has the canonical solutions 
$$
\tilde{W}_{\infty}(\lambda)=(I+O((\lambda-x)^{-1})) e^{\frac12(\lambda-x)
\sigma_3}(\lambda-x)^{-\frac 14(\sigma+\theta_{\infty})\sigma_3}
$$
as $\lambda\to\infty$ through the sector $|\arg(\lambda-x)-\pi/2|<\pi,$ and
$\tilde{W}_{\infty}^1(\lambda)$ and $\tilde{W}^3_{\infty}(\lambda)$ admitting
the same asymptotic representations in the sectors $|\arg(\lambda-x)+\pi/2|<\pi$
and $|\arg(\lambda-x)-3\pi/2|<\pi$, respectively.
Let $l_x$ be a loop starting from $\lambda_{\mathrm{st}}$ with
$\arg\lambda_{\mathrm{st}} \sim\pi/2$ and $|\lambda_{\mathrm{st}}|>2|x|$, 
and written in the form $l_x=L_+\circ\Gamma_x \circ L_+^{-1}$ as in 
Figure \ref{loops1},(c),
where $L_+$ is a segment joining $\lambda_{\mathrm{st}}$ to $i|x|^{1/4}+x,$  
and $\Gamma_x$ is given by $\lambda=x+|x|^{1/4}e^{it}$ with $\pi/2\le t\le 5\pi
/2.$
In the calculation of $M^1$, we consider systems \eqref{7.7} and \eqref{7.8}
instead of \eqref{7.3} and \eqref{7.4}, and then the Stokes matrices defined by
$\tilde{W}_{\infty}(\lambda)=\tilde{W}_{\infty}^1(\lambda)\tilde{S}_*$ and
$\tilde{W}_{\infty}(\lambda)=\tilde{W}_{\infty}^3(\lambda)\tilde{S}_{**}$ are
\begin{align*}
& \tilde{S}_* =I-\frac{2\pi i}{\Gamma(\frac 14(\sigma+2\theta_1+\theta_{\infty}))
\Gamma(1+\frac 14(\sigma-2\theta_1+\theta_{\infty}))}\sigma_{21},
\\
& \tilde{S}_{**} =I+\frac{2\pi i e^{\pi i(\sigma+\theta_{\infty})/2}}
{\Gamma(\frac 14(-\sigma+2\theta_1-\theta_{\infty}))
\Gamma(1-\frac 14(\sigma+2\theta_1+\theta_{\infty}))}\sigma_{12},
\end{align*}
and the local monodromy $M^1_{\mathrm{local}}$ such that $\tilde{W}_{\infty}
(e^{2\pi i}(\lambda-x)+x)=\tilde{W}_{\infty}(\lambda)M^1_{\mathrm{local}}$
is given by $M^1_{\mathrm{local}}= \tilde{S}_*^{-1}e^{-\frac 12 \pi i
(\sigma+\theta_{\infty})\sigma_3}\tilde{S}_{**}.$
Along $L_+$ matchings are carried out according to the scheme:
$$
\Psi(x,\lambda)\,\,\, \longleftrightarrow \,\,\, Z^1(x,\lambda) \,\,\,
\longleftrightarrow \,\,\, \tilde{V}(x,\lambda),
$$
and $\tilde{W}_{\infty}(\lambda)$ along $\Gamma_x$ yields $M^1_{\mathrm{local}}$.
Here $\tilde{V}(x,\lambda)$ is a solution of \eqref{7.7} such that
$\tilde{V}(x,\lambda)=(I+o(1))\tilde{W}_{\infty}(\lambda)$ in the domain
$|\arg (\lambda-x)-\pi/2|<\pi/4,$ $|\ln x| <|\lambda-x|<2|x|^{1/2}.$ 
\par
(1) In the domain $\Sigma_{\pi/2}(x)\cap \{|\lambda|>2|x|\}$, setting
$\Psi(x,\lambda)=e^{\frac14x\sigma_3} x^{-\frac14\theta_{\infty}\sigma_3}
Z^1(x,\lambda)\Upsilon_1(x)$, we have $\Upsilon_1(x)=e^{-\frac14 x\sigma_3}
x^{\frac14\theta_{\infty}\sigma_3}(\sigma_3+O(|\ln x|^{-1})).$ 
\par
(2) Set $e^{-\frac14 x\sigma_3} x^{-\frac14\sigma\sigma_3}c_x^{-\sigma_3/2}
Z^1(x,\lambda)=\tilde{V}\Upsilon_2(x)$ for $|\arg(\lambda-x)-\pi/2|<\pi/4,$
$|x|^{1/2}<|\lambda-x|<2|x|^{1/2}.$ Writing the left-hand side in the form
$$
e^{\frac 12\lambda\sigma_3}(\lambda-x)^{-\frac14(\sigma+\theta_{\infty})\sigma
_3} e^{-\frac14 x\sigma_3}x^{-\frac14 \theta_{\infty}\sigma_3}c_x^{-\sigma_3/2}
(\sigma_3+O(x^{-1/2})),
$$ 
we have $\Upsilon_2(x)=e^{\frac14x\sigma_3}x^{-\frac14\theta
_{\infty}\sigma_3}c_x^{-\sigma_3/2}(\sigma_3 +O(x^{-1/2})).$
\par
(3) By using $\Upsilon_1(x)$ and $\Upsilon_2(x)$ above, we have the connection
formula $\Psi(x,\lambda)=e^{\frac12x\sigma_3}x^{\frac 14(\sigma-\theta_{\infty})
\sigma_3} c_x^{\sigma_3/2}\tilde{W}_{\infty}(\lambda)c_x^{-\sigma_3/2}$,
which leads to $M^1=c_x^{\sigma_3/2} M^1_{\mathrm{local}}c_x^{-\sigma_3/2},$
that is,
\begin{equation}\label{7.9}
M^1=c_x^{\sigma_3/2}\tilde{S}_*^{-1}e^{-\frac{\pi i}2(\sigma+\theta_{\infty})
\sigma_3}\tilde{S}_{**}c_x^{-\sigma_3/2}.
\end{equation}
\subsection{$M^0_{-\pi/2}$ and $M^1_{-\pi/2}$ for $\arg x\sim -\pi/2$}
\label{ssc7.4}
In the case where $\arg x\sim -\pi/2$, monodromy data $(M^0_{-\pi/2},
M^1_{-\pi/2})$ are defined by analytic continuations of $\Psi(x,\lambda)$ 
along the loops $l_0$ and $l_x$ starting from the point $\lambda_{\mathrm{st}}$ 
as in Figure \ref{loops2},(a), and are calculated by an analogous method.
For our purpose, these loops are modified as in Figure \ref{loops2},(b) and (c).
\par
Calculation of $M^0_{-\pi/2}$ is due to the matching: 
$\Psi(x,\lambda) \,\,\, \longleftrightarrow \,\,\, Z^0(x,\lambda) \,\,\,
\longleftrightarrow \,\,\, V(x,\lambda)$. 
From \eqref{7.4a} and \eqref{7.5} with $\Psi(x,\lambda)$ in place of $\Psi_3(x,
\lambda)$, we have
$\Upsilon_1(x)=e^{-\frac14x\sigma_3}x^{\frac14\theta_{\infty}\sigma_3}
\sigma_3,$ and
$\Upsilon_2(x)=e^{\frac14 x\sigma_3}x^{-\frac14\theta_{\infty}
\sigma_3}e^{-\frac{\pi i}4(\sigma+\theta_{\infty})\sigma_3}c_0^{-\sigma_3/2}
\sigma_3$ up to $1+o(1).$ In this calculation, we set 
$\lambda-x=e^{\pi i}x(1-\lambda/x)$
for $|\arg\lambda-\pi/2|<\pi/4,$ $|x|^{1/2}<|\lambda|<2|x|^{1/2}.$ 
Then $\Upsilon_2(x)\Upsilon_1(x)=e^{-\frac{\pi i}4(\sigma+\theta_{\infty})\sigma
_3}c_0^{-\sigma_3/2}$ and $W_{\infty}(e^{2\pi i}\lambda)=W_{\infty}(\lambda)
M^0_{-\pi/2,\,\mathrm{local}}$ yields
\begin{equation}\label{7.10}
M^0_{-\pi/2}=c_0^{\sigma_3/2}e^{\frac{\pi i}4(\sigma+\theta_{\infty})\sigma_3}
S^{-1}_*e^{\frac{\pi i}2(\sigma-\theta_{\infty})\sigma_3}S_{**}e^{-\frac{\pi i}4
(\sigma+\theta_{\infty})\sigma_3}c_0^{-\sigma_3/2}.
\end{equation}
\par
The matching:
$\Psi_1(x,\lambda)\,\,\,\longleftrightarrow \,\,\, Z^1(x,\lambda) \,\,\,
\longleftrightarrow \,\,\, \tilde{V}$ yields $M^1_{-\pi/2}$, 
which is as in Subsection \ref{ssc7.3}
with $\Psi_1(x,\lambda)$ instead of $\Psi(x,\lambda).$
From $\Upsilon_2(x)\Upsilon_1(x)=c_x^{-\sigma_3/2}$ combined with
a local monodromy such that $\tilde{W}^1_{\infty}(e^{2\pi i}(\lambda-x)+x)=
\tilde{W}^1_{\infty}(\lambda)M^1_{-\pi/2,\,\mathrm{local}}$ it follows that
\begin{equation}\label{7.11}
M^1_{-\pi/2}=S_1^{-1} c_x^{\sigma_3/2}e^{-\frac{\pi i}2(\sigma+\theta_{\infty})
\sigma_3}\tilde{S}_{**}\tilde{S}^{-1}_* c_x^{-\sigma_3/2}S_1.
\end{equation}
{\small
\begin{figure}[htb]
\begin{center}
\unitlength=0.75mm
\begin{picture}(60,60)(-30,-30)
\put(0,8){\circle*{1}}   
\put(0,8){\circle{6}}   
\put(0,-8){\circle*{1}}
\put(0,-8){\circle{6}}  

\put(-7.5,5.5){\makebox{$0$}}
\put(-7.5,-8.5){\makebox{$x$}}
\put(0,23){\line(0,-1){12}}  

\put(0,23){\circle*{1.5}}
\put(3,22){\makebox{$\lambda_{\mathrm{st}}$}}
 \qbezier(0,23)(18,-4)(3,-8)
\thicklines
\put(-12,-34){\makebox{(a) Loops}}
\end{picture}
\begin{picture}(60,55)(-30,-30)
\put(0,-20){\circle*{1}}
\put(0,-5){\circle*{1}}
\put(0,-5){\circle{10}}

\put(6,-9){\makebox{$0$}}
\put(2.5,-21){\makebox{$x$}}
\qbezier(-7,0)(0,8)(7,0)
\put(-7,0){\line(-1,1){14}}
\put(7,0){\line(1,1){14}}
\put(0,12){\line(0,-1){12}}

\put(0,23){\circle*{1.5}}
\put(-8,22){\makebox{$\lambda_{\mathrm{st}}$}}
\thicklines
\put(-0.2,23){\line(0,-1){11}}
\put(0,23){\line(0,-1){11}}
\put(0.2,23){\line(0,-1){11}}
\put(-6,-34){\makebox{(b) $l_0$}}
\end{picture}
\quad\qquad
\begin{picture}(60,55)(-30,-30)
\put(0,15){\circle*{1}}
\put(0,0){\circle*{1}}
\put(0,0){\circle{10}}
\put(-5,12.5){\makebox{$0$}}
\put(-8,2.5){\makebox{$x$}}
\qbezier(-7,-7)(0,-12)(7,-7)
\put(-7,-7){\line(-1,-1){14}}
\put(7,-7){\line(1,-1){14}}
\put(0,-23){\line(0,1){18}}
\put(0,23){\circle*{1.5}}
\put(-8,22){\makebox{$\lambda_{\mathrm{st}}$}}
\thicklines
\qbezier(0,23)(9.52,23)(16.26,16.26)
\qbezier(23,0)(23,9.52)(16.26,16.26)
\qbezier(0,-23)(9.52,-23)(16.26,-16.26)
\qbezier(23,0)(23,-9.52)(16.26,-16.26)
\put(-6,-34){\makebox{(c) $l_x$}}
\end{picture}
\end{center}
\caption{Loops for $\arg x\sim -\pi/2$}
\label{loops2}
\end{figure}
}
\subsection{Monodromy data of Proposition \ref{prop5.1}}\label{ssc7.5}
Suppose that $\arg x\sim \pi/2.$ From $M^0$ of \eqref{7.6} it follows that 
$m^0_{21}=c_0^{-1}e^{-\pi i\theta_{\infty}}(S_*^{-1})_{21}$.  
From $M^1$ of \eqref{7.9} we derive $m^1_{11}=e^{-\frac{\pi i}2(\sigma+
\theta_{\infty})}$ and $m^1_{12}=e^{-\frac{\pi i}2(\sigma+\theta_{\infty})}
(\tilde{S}_{**})_{12}c_x.$ 
Thus we obtain the monodromy data as in Proposition \ref{prop5.1}. 
Similarly the monodromy data for $\arg x\sim -\pi/2$ follow from \eqref{7.10} 
and \eqref{7.11}. 
\begin{rem}\label{rem7.1}
The monodromy data are essentially due to the Stokes matrices for the Whittaker
function of Lemmas \ref{lem7.3} and \ref{lem7.4} with no conditions on
$\theta_0,$ $\theta_1,$ $\theta_{\infty}$ and $\sigma.$
\end{rem}


\end{document}